\documentclass[12pt]{amsart}
\usepackage[utf8]{inputenc}
\usepackage[english]{babel}
\usepackage{subcaption}
\usepackage{enumerate}
\usepackage{amsfonts}
\usepackage{comment}
\usepackage{amsmath,amssymb}
\usepackage[makeroom]{cancel}
\usepackage{stmaryrd}
\usepackage{tikz-cd} 
\usepackage{needspace}
\usepackage{enumerate}
\usepackage{enumitem}
\usepackage{graphicx}
\usepackage{amsthm}
\usepackage{wrapfig}
\usepackage{hyperref}
\usepackage{chngcntr}
\counterwithin{figure}{section}
\newtheorem{theorem}{Theorem}
\newtheorem*{theorem*}{Theorem}
\newtheorem{prop}[theorem]{Proposition}

\newtheorem{lemm}[theorem]{Lemma}

\newtheorem{thdef}[theorem]{Theorem-Definition}
\newtheorem{cor-definition}[theorem]{Corollary-Definition}
\theoremstyle{definition}
\newtheorem{defi/}[theorem]{Definition}
\newenvironment{defi}
  {%
   \pushQED{\qed}\begin{defi/}}
  {\popQED\end{defi/}}
\newtheorem*{defi*}{Definition}
\newtheorem*{prop*}{Proposition}
\newtheorem*{lemm*}{Lemma}
\newtheorem{ques}[theorem]{Question}
\newtheorem{rema}[theorem]{Remark}
\newtheorem{obse}[theorem]{Observation}

\newtheorem*{conv*}{Convention}

\numberwithin{theorem}{section}

\numberwithin{equation}{section}
\title{From Markovian actions to Anosov flows}
\author{François Béguin and Ioannis Iakovoglou}

\begin{document}
\sloppy
\maketitle

\begin{abstract}
   In this paper, we establish a necessary and sufficient condition for a group action on a bifoliated plane to arise from a topological Anosov flow supported by an orientable $3$-manifold. More precisely, we show that a group action on a non-singular bifoliated plane arises from such a flow if and only if it is an orientation preserving strong Markovian action, that is, an orientation preserving action of the plane preserving a family of rectangles satisfying suitable Markov-type properties. Furthermore, we prove that this family of rectangles can be lifted to a Markov partition of the associated flow.

\end{abstract}
\section{Introduction}
Thanks to the work of T. Barbot and S. Fenley, to every Anosov flow $\Phi$ on a closed $3$-manifold $M$ one can associate a topological plane $\mathcal{P}$ endowed with a pair of transverse line foliations $(F^s,F^u)$ together with a continuous action $\rho_{\Phi}:\pi_1(M)\to\mathrm{Homeo}(\mathcal{P})$ preserving $F^s$ and $F^u$. We call $(\mathcal{P},F^s,F^u)$ the \emph{bifoliated plane of $\Phi$} and $\rho_\Phi$ \emph{the action of $\pi_1(M)$ on $(\mathcal{P},F^s,F^u)$}. By a theorem of T. Barbot, the data $(\mathcal{P},F^s,F^u,\rho_{\Phi})$, considered up to conjugation, completely determine $\Phi$ up to orbital equivalence. The previous results define an injective correspondence assigning to each $3$-dimensional Anosov flow, up to orbital equivalence, a conjugacy class of group actions on the plane each preserving a pair of transverse foliations.

Conversely, given a topological plane $\mathcal{P}$, a group $G$ and a continuous action $\rho: G\rightarrow \text{Homeo}(\mathcal{P})$ preserving a pair of transverse line foliations $(\mathcal{F}^s,\mathcal{F}^u)$, it is natural to ask: 

\begin{ques}\label{q.origin}
   Under which conditions do the data $(\mathcal{P},\mathcal{F}^s,\mathcal{F}^u,\rho)$ arise from an Anosov flow on a closed $3$-manifold $M$?
\end{ques}

Motivated by the above question, the second author defined in \cite{markovianactions} the notion of \emph{strong Markovian action}, an action on a bifoliated plane preserving a ``Markov partition". 

\begin{defi}[Markovian family - Informal]\label{d.markovianactioninformal}
 A \emph{strong Markovian family of $\rho$ in $(\mathcal{P},\mathcal{F}^s,\mathcal{F}^u)$} is a $\rho$-invariant set of rectangles $\mathcal{R}=(R_i)_{i \in I}$ (i.e. topological disks that are trivially bifoliated by $\mathcal{F}^s$ and $\mathcal{F}^u$) satisfying the four following properties: 
\begin{enumerate}
\item $\mathcal{R}$ consists of finitely many $\rho$-orbits of rectangles.
\vspace{0.1cm}
\item For every two distinct rectangles $R_i,R_j\in\mathcal{R}$, if $\mathrm{Int}(R_i) \cap \mathrm{Int}(R_j) \neq \emptyset$, then, up to changing the roles of $R_i$ and $R_j$, $R_i \cap R_j$ is a non-trivial horizontal subrectangle of $R_i$ and a non-trivial vertical subrectangle of $R_j$.  
\vspace{0.1cm}
\item For any compact set $U\subset \mathcal{P}$, we have that $U$ can be covered by finitely many rectangles in $\mathcal{R}$. 
\vspace{0.1cm}
\item If $(R_i)_{i \in \mathbb{Z}}$ is a bi-infinite sequence of rectangles in $\mathcal{R}$ such that for every $i\in\mathbb{Z}$ the intersection $R_{i+1}\cap R_i$ is a non-trivial vertical subrectangle of $R_i$, then $\underset{i\in \mathbb{Z} }{\cap}R_i$ contains a unique point. Conversely, every point of $\mathcal{P}$ arises as the intersection of a bi-infinite sequence of rectangles satisfying the previous property.
\end{enumerate}
\end{defi}
For a more detailed definition of the notion of strong Markovian family, see Definition \ref{d.markovfamily}.

\begin{defi}\label{d.markovianaction}
    Consider $\rho$ a continuous and faithful action of a group $G$ on the plane $\mathcal{P}$ that preserves a pair of transverse foliations $(\mathcal{F}^s,\mathcal{F}^u)$. We will say that $\rho$ is a \emph{strong Markovian action} if: 
    \begin{enumerate}
        \item $G$ is a countable group with no torsion and
        \item $\rho$ preserves a strong Markovian family in $(\mathcal{P},\mathcal{F}^s,\mathcal{F}^u)$ 
    \end{enumerate}

\end{defi}
An important source of examples of strong Markovian actions is given by Anosov flows. Indeed, the action of the fundamental group on the bifoliated plane of any Anosov flow is a strong Markovian action (see \cite{monstre}). Conversely, in this paper, we show that strong Markovian actions are precisely the actions of the plane that arise from Anosov flows. More specifically, 
\begin{defi}\label{d.cornercondition}
    Consider $\mathcal{R}$ a strong Markovian family of $\rho$ in $(\mathcal{P},\mathcal{F}^s,\mathcal{F}^u)$. We will say that $\mathcal{R}$ satisfies the \emph{corner condition} if for every $R\in\mathcal{R}$ and every non-trivial element $g\in G$, the homeomorphism $\rho(g)$ fixes a point $x\in\partial R$ if and only if $x$ is a corner of $R$.
\end{defi}

\begin{theorem}\label{t.main}    
     Fix $\mathcal{P}$ a topological plane endowed with a pair of transverse foliations $(\mathcal{F}^s,\mathcal{F}^u)$. Let $\rho: G\rightarrow \text{Homeo}(\mathcal{P})$ be an orientation preserving strong Markovian action preserving a strong Markovian family $\mathcal{R}$ that satisfies the corner condition. 
     
    Then, there exists a closed orientable $3$-manifold $M$, a topological Anosov flow $\Phi=(\Phi_t)_{t\in\mathbb{R}}$ on $M$ and a reduced Markov partition $\mathcal{M}$ of $\Phi$ such that:
    \begin{enumerate}
        \item the bifoliated plane $(\mathcal{P}_{\Phi},\mathcal{F}^s_{\Phi},\mathcal{F}^u_{\Phi})$  of  $\Phi$ is isomorphic to $(\mathcal{P},\mathcal{F}^s,\mathcal{F}^u)$: there exists a homeomorphism $h:\mathcal{P}\rightarrow \mathcal{P}_{\Phi}$ such that $h(\mathcal{F}^{s,u})=\mathcal{F}^{s,u}_{\Phi}$,
        \item if $\rho_{\Phi}$ denotes the action of $\pi_1(M)$ on $(\mathcal{P}_{\Phi},\mathcal{F}^s_{\Phi},\mathcal{F}^u_{\Phi})$, then $\rho_\Phi$ is conjugated to $\rho$: there exists an isomorphism $\alpha:G\to\pi_1(M)$ such that for every $g\in G$ 
        $$h \circ \rho(g)=\rho_{\Phi}(\alpha(g))\circ h $$
        \item $h(\mathcal{R})$ is the projection on $\mathcal{P}_{\Phi}$ of the lift of $\mathcal{M}$ to the universal cover $\widetilde{M}$ of $M$.
    \end{enumerate}
\end{theorem}

For a definition of the notion of reduced Markov partition, see Definition \ref{d.markovpartition}.

Before proceeding, we make a few remarks concerning the main theorem. First, regarding the corner condition, although Item~(3) of Theorem~\ref{t.main} remains true for strong Markovian families that do not satisfy it, we impose this technical assumption to reduce the combinatorial complexity of our proofs (see Remark~\ref{r.cornercondition}). Moreover, as we shall prove in Proposition~\ref{p.cornerconditionexists}, this assumption is not restrictive: every strong Markovian action preserves a strong Markovian family that satisfies the corner condition. Hence, \emph{every} orientation preserving strong Markovian action on the plane arises from some topological Anosov flow. 

Next, the proof of Theorem~\ref{t.main} relies on purely topological arguments, which allow us to construct topological Anosov flows from bifoliated planes endowed with strong Markovian actions. Although, by a theorem of Mario Shannon (see \cite{Shannon}), every transitive topological Anosov flow is orbitally equivalent to a smooth Anosov flow, the corresponding result is not known in the non-transitive setting. This is why Theorem~\ref{t.main} is stated for topological, rather than smooth, Anosov flows.

We also note that Theorem \ref{t.main} can be generalized for strong Markovian actions of the plane that preserve a pair of singular transverse foliations; in this case, the associated flow is pseudo-Anosov. Since the proof of this more general result is more involved, we postpone it to a subsequent work.

Finally, results closely related to Theorem \ref{t.main} have previously appeared in the literature. The authors of \cite{reconstruct} proved that any Anosov-like action on a singular bifoliated plane preserving the orientations of the singular foliations (which are assumed orientable) arises from a pseudo-Anosov flow. Later, in \cite{reconstruct2} the previous result was generalized for all Anosov-like actions with dense orbits. 

Our result differs from these works in several respects. To begin, our theorem does not require either the preservation of the orientations of the invariant foliations or the existence of dense orbits. Moreover, we emphasize that our hypotheses are formulated in a different language from those of the previous works: rather than assuming topological properties of the action, such as the existence and uniqueness of fixed points and their dynamical properties, we assume the existence of a preserved strong Markovian family. This formulation allows us to capture the finite combinatorial structure underlying Anosov flows, as \cite{preprint} shows that an Anosov flow is completely determined, up to orbit equivalence, by the finite data encoded in a strong Markovian family. As a final remark, our proof shows that every strong Markovian family on a bifoliated plane arises from a topological Anosov flow together with one of its Markov partitions. This further supports the viewpoint developed in \cite{preprint} of studying Anosov flows through strong Markovian families, by showing that these families naturally encode Markov partitions of the associated flows.

\textit{Acknowledgments.} Both authors are supported by the ANR project Anodyne ANR-24-CE40-5065-01.

\section{Preliminaries}

\subsection{Bifoliations on the plane.}

Let $\mathcal{P}$ be a topological plane endowed with a pair of transverse line foliations $(\mathcal{F},\mathcal{G})$. Throughout this paper, we shall refer to $(\mathcal{P},\mathcal{F},\mathcal{G})$ as a \emph{bifoliated plane}.

\begin{defi}
    Two bifoliated planes $(\mathcal{P}_1,\mathcal{F}_1,\mathcal{F}_2)$ and $(\mathcal{P}_2,\mathcal{G}_1,\mathcal{G}_2)$ are called isomorphic if, up to exchanging the roles of $\mathcal{F}_1$ and $\mathcal{F}_2$, there exists a homeomorphism $h:\mathcal{P}_1\rightarrow \mathcal{P}_2$ such that $h(\mathcal{F}_1)=\mathcal{G}_1$ and $h(\mathcal{F}_2)=\mathcal{G}_2$. 
\end{defi}

Since $\mathcal{F}$ and $\mathcal{G}$ are transverse, each point of the bifoliated plane $(\mathcal{P},\mathcal{F},\mathcal{G})$ belongs in the interior of some trivially bifoliated region. Among the trivially bifoliated subsets of $(\mathcal{P},\mathcal{F},\mathcal{G})$, we will pay particular attention to \emph{rectangles}, which appear in the definition of a strong Markovian family.  

\begin{defi}\label{d.rectangle}
 A segment contained in a leaf of $\mathcal{F}$ (resp. $\mathcal{G}$) will be called an $\mathcal{F}$-\emph{segment} (resp. \emph{$\mathcal{G}$-segment}), or simply a \emph{segment of $\mathcal{F}$} (resp. \emph{$\mathcal{G}$}).
 
Consider a closed disk $R$ in $\mathcal{P}$ that is bounded by the union of two $\mathcal{F}$-segments and two $\mathcal{G}$-segments and that is trivially bifoliated by the restrictions of $\mathcal{F}$ and $\mathcal{G}$ to $R$. We will call $R$ a \emph{rectangle} in $(\mathcal{P},\mathcal{F},\mathcal{G})$. 
\end{defi}

Although there exist uncountably many bifoliated planes up to isomorphism, pairs of transverse foliations on the plane satisfy several common topological properties:
\begin{defi}
Let $\mathcal{F}$ be a foliation on $\mathcal{P}$.

\begin{itemize}
    \item For any $x\in\mathcal{P}$, we will denote by $\mathcal{F}(x)$ the leaf of $\mathcal{F}$ containing $x$.
    
    \item For any $x\in\mathcal{P}$, the closure (with respect to the leaf topology) of a connected component of $\mathcal{F}(x)-\{x\}$ will be called an \emph{$\mathcal{F}$-separatrix} of $x$.
\end{itemize}
\end{defi}

\needspace{6\baselineskip}
\begin{prop}\label{p.propertiesoffoliformarkovianactions}
Consider $(\mathcal{P},\mathcal{F},\mathcal{G})$ a bifoliated plane. We have that:
    \begin{enumerate}
        \item The foliations $\mathcal{F}$ and $\mathcal{G}$ do not admit compact leaves. 
        \item Each leaf of $\mathcal{F}$ or $\mathcal{G}$ is a closed and properly embedded line in $\mathcal{P}$. In particular, each $\mathcal{F}$- or $\mathcal{G}$-separatrix of a point $x\in \mathcal{P}$ is a closed and properly embedded half-line in $\mathcal{P}$. 
        \item Any leaf of $\mathcal{F}$ or $\mathcal{G}$ separates $\mathcal{P}$ into 2 simply connected components. Moreover, for every $x\in \mathcal{P}$ the set $$\mathcal{P}-(\mathcal{F}(x)\cup \mathcal{G}(x))$$ consists of 4 connected components. 
        \item Every leaf of  $\mathcal{F}$ intersects at most once any leaf of $\mathcal{G}$.
        \item Any leaf of  $\mathcal{F}$  (resp. $\mathcal{G}$) intersects a rectangle in $(\mathcal{P}, \mathcal{F}, \mathcal{G})$ at most along a unique $\mathcal{F}$-segment (resp. $\mathcal{G}$-segment).

    \end{enumerate}
\end{prop}
The above proposition is classical; a proof can be found, for instance, in Proposition~2.63 of \cite{monstre}. We conclude this section with the following definition, based on the above proposition. 

\begin{defi}
Let $(\mathcal{P},\mathcal{F},\mathcal{G})$ be a bifoliated plane.

\begin{itemize}
    \item For any $x\in \mathcal{P}$, the closure of a connected component of $\mathcal{P}-(\mathcal{F}(x)\cup \mathcal{G}(x))$ is called a \emph{quadrant} of $x$ in $(\mathcal{P},\mathcal{F},\mathcal{G})$. Moreover, a neighborhood of $x$ inside one of its quadrants $Q$ is called a \emph{germ of $Q$}.
    \item For any rectangle $R$ in $(\mathcal{P},\mathcal{F},\mathcal{G})$ and any $x\in R$, we call the $\mathcal{F}$-segment $\mathcal{F}(x)\cap R$ the \emph{$\mathcal{F}$-leaf} of $x$ in $R$. We similarly define the \emph{$\mathcal{G}$-leaf} of $x$ in $R$.

\end{itemize}
\end{defi}

\subsection{The bifoliated plane of an Anosov flow}

\begin{defi}\label{d.anosovflow}
  Consider $M$ a closed smooth $3$-manifold and $(X^t)_{t\in \mathbb{R}}$ a non-singular $C^0$-flow on $M$. We will say that $(X^t)_{t\in \mathbb{R}}$ is a \emph{topological Anosov} flow if, given any distance $d$ on $M$ induced by a Riemannian metric, we have that : 
  \begin{enumerate}
      \item There exists a pair of transverse codimension-one $C^{0,0}$-foliations $F^s,F^u$ (i.e., the holonomy and the leaves of $F^s$ and $F^u$ are $C^0$) such that $(X^t)_{t\in \mathbb{R}}$ preserves every leaf of $F^s$ and $F^u$.
      \item For every $x\in M,y\in F^s(x)$ (resp. $y\in F^u(x)$) there exists an increasing homeomorphism $h:\mathbb{R}\rightarrow \mathbb{R}$ such that $$d(X^t(x),X^{h(t)}(y))\underset{t\rightarrow +\infty}{\longrightarrow} 0~~ \big(\text{resp. }d(X^t(x),X^{h(t)}(y))\underset{t\rightarrow -\infty}{\longrightarrow}0\big) $$
      \item (Expansivity along the leaves) There exist $\eta, \epsilon>0$ such that for any $x\in M$, $y\in F^s(x)$, $z\in F^u(x)$ and any increasing  homeomorphism $h:\mathbb{R}\rightarrow \mathbb{R}$ with $h(0)=0$ $$ \forall t\in \mathbb{R}~ d(X^t(x),X^{h(t)}(y))<\eta \implies \exists |t_0|<\epsilon ~~ y=X^{t_0}(x)$$
      $$ \forall t\in \mathbb{R}~ d(X^t(x),X^{h(t)}(z))<\eta \implies \exists |t_0|<\epsilon ~~z=X^{t_0}(x)$$
  \end{enumerate}
We call $F^s$ and $F^u$ the \emph{stable} and \emph{unstable foliations} of $(X^t)_{t\in \mathbb{R}}$, respectively. 
\end{defi}

Consider $\Phi$ a topological Anosov flow on a smooth and closed $3$-manifold $M$. Denote by $F^s,F^u$ the stable and unstable foliations of $\Phi$, by $\widetilde{M}$ the universal cover of $M$ and by $\widetilde{\Phi},\widetilde{F^s},\widetilde{F^u}$ the lifts of the flow $\Phi$ and of the foliations $F^s,F^u$ on $\widetilde{M}$.

Thanks to the following two results, despite the dynamical complexity of $\Phi$ in $M$ (including infinitely many periodic orbits and orbits with complicated accumulation sets), $\widetilde{\Phi}$ is topologically conjugate to the constant-speed vertical flow on $\mathbb{R}^3$.

\begin{theorem}[C. Palmeira, \cite{Palmeira}]
    $\widetilde{M}$ is homeomorphic to $\mathbb{R}^3$.
\end{theorem}

\begin{thdef}\label{thdef.bifoliatedplane}
    Let $\mathcal{P}$ be the space of orbits of $\widetilde{\Phi}$ endowed with the quotient topology. We have that 
    \begin{enumerate}
        \item $\mathcal{P}$ is homeomorphic to $\mathbb{R}^2$.
        \item The projections of $\widetilde{F^s}$ and $\widetilde{F^u}$ onto $\mathcal{P}$ define a pair of transverse line foliations $\mathcal{F}^s$ and $\mathcal{F}^u$.  
        \item Any action by deck transformations of $\pi_1(M)$ on $\widetilde{M}\simeq\mathbb{R}^3$ projects on $\mathcal{P}$ to an action by homeomorphisms preserving $\mathcal{F}^s$ and $\mathcal{F}^u$.
    \end{enumerate}
   We call $\mathcal{P}$ the bifoliated plane or the orbit space of $\Phi$ and $\mathcal{F}^s,\mathcal{F}^u$ the stable and unstable foliations of $\mathcal{P}$. 
\end{thdef}

Theorem-Definition~\ref{thdef.bifoliatedplane} was proved independently for $C^1$ Anosov flows by T.~Barbot and S.~Fenley in \cite{Barbotthese} and \cite{Fe}, respectively. The proofs of Barbot and Fenley extend verbatim to the setting of topological Anosov flows (see for instance \cite{planeapproach} or \cite{monstre}).

Recall that $\pi_1(M)$ admits infinitely many actions on $\mathbb{R}^3$ by deck transformations, any two of which differ by an inner automorphism of $\pi_1(M)$. Consequently, Item (3) of the previous Theorem-Definition defines infinitely many actions of $\pi_1(M)$ on $\mathcal{P}$ by homeomorphisms, all equivalent up to an inner automorphism of $\pi_1(M)$. By abuse of language, throughout this paper we will refer to any of these actions as \textbf{the} action of $\pi_1(M)$ on the bifoliated plane of $\Phi$.

\subsection{Markov partitions of Anosov flows}
Following our previous notations, we now define the notion of a \emph{Markov partition} for the topological Anosov flow $\Phi$. As in Definition~\ref{d.rectangle}, any segment contained in a leaf of $F^s$ (resp. $F^u$) will be called an $F^s$-\emph{segment} (resp. an $F^u$-\emph{segment}).

\begin{defi}\label{d.rectangleM}
Consider $i:\mathbb{D}\rightarrow M$ an embedding of the closed disk in $M$ such that $R:=i(\mathbb{D})$ is transverse to $\Phi$ and also $i(\partial\mathbb{D})$ is the union of two $F^s$-segments and two $F^u$-segments. If $R$ is trivially bifoliated by the restrictions of $F^s$ and $F^u$ on $R$, then we will call $R$ \emph{a transverse rectangle of $\Phi$}. 

We call $i(\operatorname{Int}(\mathbb{D}))$ (resp. $i(\partial\mathbb{D})$) the \emph{interior} (resp. \emph{boundary}) of $R$, and denote it by $\operatorname{Int}(R)$ (resp. $\partial R$). Furthermore, we denote by $\partial^sR$ (resp. $\partial^uR$) the union of the two $F^s$-segments (resp. $F^u$-segments) contained in $\partial R$, and call them the \emph{stable boundary} (resp. \emph{unstable boundary}) of $R$.

Finally, any subrectangle of $R$, say $R'$, for which $\partial^s R'\subset \partial^s R$ will be called a \emph{vertical subrectangle} of $R$. Similarly, any subrectangle $R'$ of $R$ for which $\partial^u R'\subset \partial^u R$ will be called a \emph{horizontal subrectangle} of $R$.
\end{defi}

\begin{defi}\label{d.markovpartition}
A \emph{Markov partition} of $\Phi$ is a finite family of transverse rectangles of $\Phi$, say $R_1,...,R_n$, such that: 
\begin{enumerate}
\item The interiors of the rectangles are pairwise disjoint.

\vspace{0.1cm}
\item There exists $T>0$ such that every orbit segment of $\Phi$ of length $T$ intersects $\underset{i=1}{\overset{n}{\cup}}R_i$.
\item For every $i,j \in \llbracket 1, n \rrbracket $ the strictly positive (resp. negative) orbit by $\Phi$ of the stable (resp. unstable) boundary of $R_i$ does not intersect $\mathrm{Int}(R_j)$.
\end{enumerate}
Furthermore, we will call the family $R_1,...,R_n$ a \emph{reduced} Markov partition if in addition to the previous axioms it also satisfies the following: 

\begin{enumerate}\setcounter{enumi}{3}
    \item For every $i\neq j$, there does not exist a continuous function $\tau: R_i \rightarrow \mathbb{R}$ such that $\Phi^{\tau}(R_i)\subseteq R_j$. 
\end{enumerate}
\end{defi}

By Item~(2) of the above definition, every Markov partition $R_1,...,R_n$ of $\Phi$ admits a first return map $$f:\underset{i\in \llbracket 1,n\rrbracket}{\cup}R_i\rightarrow \underset{i\in \llbracket 1,n\rrbracket}{\cup}R_i$$
Moreover, Item (3) implies that for every $i,j\in \llbracket 1,n\rrbracket$, the closure of every connected component of $f(\mathrm{Int}(R_i))\cap \mathrm{Int}(R_j)$ (resp. $f^{-1}(\mathrm{Int}(R_i))\cap \mathrm{Int}(R_j)$) is a vertical (resp. horizontal) subrectangle of $R_j$. These two properties will be at the heart of the definition of strong Markovian families introduced later.
 
The following two results ensure the existence of Markov partitions and reduced Markov partitions for all Anosov flows.  
\begin{theorem} \label{t.existenceofmarkovpartitions}
Let $\Phi$ be any topological Anosov flow on a closed 3-manifold $M$ and $F^s,F^u$ its stable and unstable foliations. For any finite set of periodic orbits $\Gamma$ of $\Phi$ such that $\underset{\gamma\in\Gamma}{\cup}F^s(\gamma)$ and  $\underset{\gamma\in\Gamma}{\cup}F^u(\gamma)$ are dense in $M$, there exists a Markov partition of $\Phi$ formed by rectangles, whose stable and unstable boundaries are contained respectively in $\underset{\gamma\in\Gamma}{\cup}F^s(\gamma)$ and  $\underset{\gamma\in\Gamma}{\cup}F^u(\gamma)$.
\end{theorem}
The above theorem was originally shown by M. Ratner in the case of transitive Anosov flows (see \cite{Ratner}) and was later generalized by the second author in \cite{markovpseudoanosov}. The existence of infinitely many sets of periodic orbits satisfying the density hypothesis of the above theorem was also shown in \cite{markovpseudoanosov}. 

\begin{prop}\label{t.reducedmarkovpartitionexists}
Let $\Phi$ be a topological Anosov flow on a closed $3$-manifold $M$. Every Markov partition $\mathcal{M}$ of $\Phi$ contains a subcollection of rectangles that forms a reduced Markov partition of $\Phi$.
\end{prop}
The above proposition was established for transitive Anosov flows in Theorem 2.5 of \cite{preprint}. However, the proof does not use the transitivity hypothesis and therefore remains valid for any Anosov flow in dimension $3$.

\subsection{Strong Markovian families in the plane}
Consider $\mathcal{P}$ an oriented plane endowed with a pair $(\mathcal{F}^s,\mathcal{F}^u)$ of transverse foliations. Denote by $\rho:G \rightarrow \text{Homeo}(\mathcal{P})$ a faithful and orientation preserving $C^0$ action of a countable group $G$ on the plane preserving both $ \mathcal{F}^s$ and $\mathcal{F}^u$. 

In this section, we give a formal definition of strong Markovian families of $\rho$ and state several of their properties. Similarly to Definition \ref{d.rectangleM},

\begin{defi}\label{d.subrectangleplane}
Consider $R$ a rectangle in $(\mathcal{P},\mathcal{F}^s,\mathcal{F}^u)$. Denote by $\partial^sR$ (resp. $\partial^uR$) the union of the two $\mathcal{F}^s$-segments (resp. $\mathcal{F}^u$-segments) in $\partial R$. We call $\partial^sR$ (resp. $\partial^uR$) the  \emph{$\mathcal{F}^s$-boundary} (resp. $\mathcal{F}^u$-boundary) of $R$. Any connected component of $\partial^sR$  (resp. $\partial ^uR$) will be called an \emph{$\mathcal{F}^s$-boundary} (resp. \emph{$\mathcal{F}^u$-boundary}) \emph{component} of $R$.

Moreover, any subrectangle $R'$ of $R$ such that $\partial^sR'\subset \partial^sR $ will be called a \emph{vertical subrectangle} of $R$. If furthermore $R' \neq R$, then we will call $R'$ a \emph{non-trivial} vertical subrectangle of $R$. We define similarly \emph{horizontal subrectangles} and \emph{non-trivial horizontal subrectangles}.

\end{defi}

\begin{defi}\label{d.markovfamily}
A \emph{strong Markovian family of $\rho$} inside  $(\mathcal{P},\mathcal{F}^s,\mathcal{F}^u)$ is a $\rho$-invariant set of  rectangles $\mathcal{R}=(R_i)_{i \in I}$ covering $\mathcal{P}$ (i.e. $\cup_{i \in I} R_i = \mathcal{P}$) such that 
\begin{enumerate}
\item (Finiteness axiom) $\mathcal{R}$ consists of finitely many $\rho$-orbits of rectangles. 

\vspace{0.1cm}
\item (Markovian intersection axiom) For every two distinct rectangles $R_i,R_j\in\mathcal{R}$, if $\mathrm{Int}(R_i) \cap \mathrm{Int}(R_j) \neq \emptyset$, then, up to changing the roles of $R_i$ and $R_j$, $R_i \cap R_j$ is a non-trivial horizontal subrectangle of $R_i$ and a non-trivial vertical subrectangle of $R_j$.
\vspace{0.1cm}
\item (Strong finite return time axiom)  For any point $x\in \mathcal{P}$ and any sufficiently small germ $U$ of a quadrant of $x$ there exists $R, R_h, R_v\in \mathcal{R}$ such that $R_h\cap R$ is a non-trivial horizontal subrectangle of $R$, $R_v\cap R$ is a non-trivial vertical subrectangle of $R$ and $U\subset R\cap R_h\cap R_v$. 
\vspace{0.1cm}
\item(Expansivity axiom) If $(R_i)_{i\in \mathbb{N}}$ is a sequence of rectangles in $\mathcal{R}$ such that $R_{i+1}\cap R_i$ is a non-trivial vertical (resp. horizontal) subrectangle of $R_i$ for every $i\in \mathbb{N}$, then $\underset{i\geq 0 }{\cap}R_i$ is an $\mathcal{F}^u$-leaf (resp. $\mathcal{F}^s$-leaf) of $R_0$.
\end{enumerate}
\end{defi}

\begin{figure}
    \centering
    \includegraphics[scale=0.32]{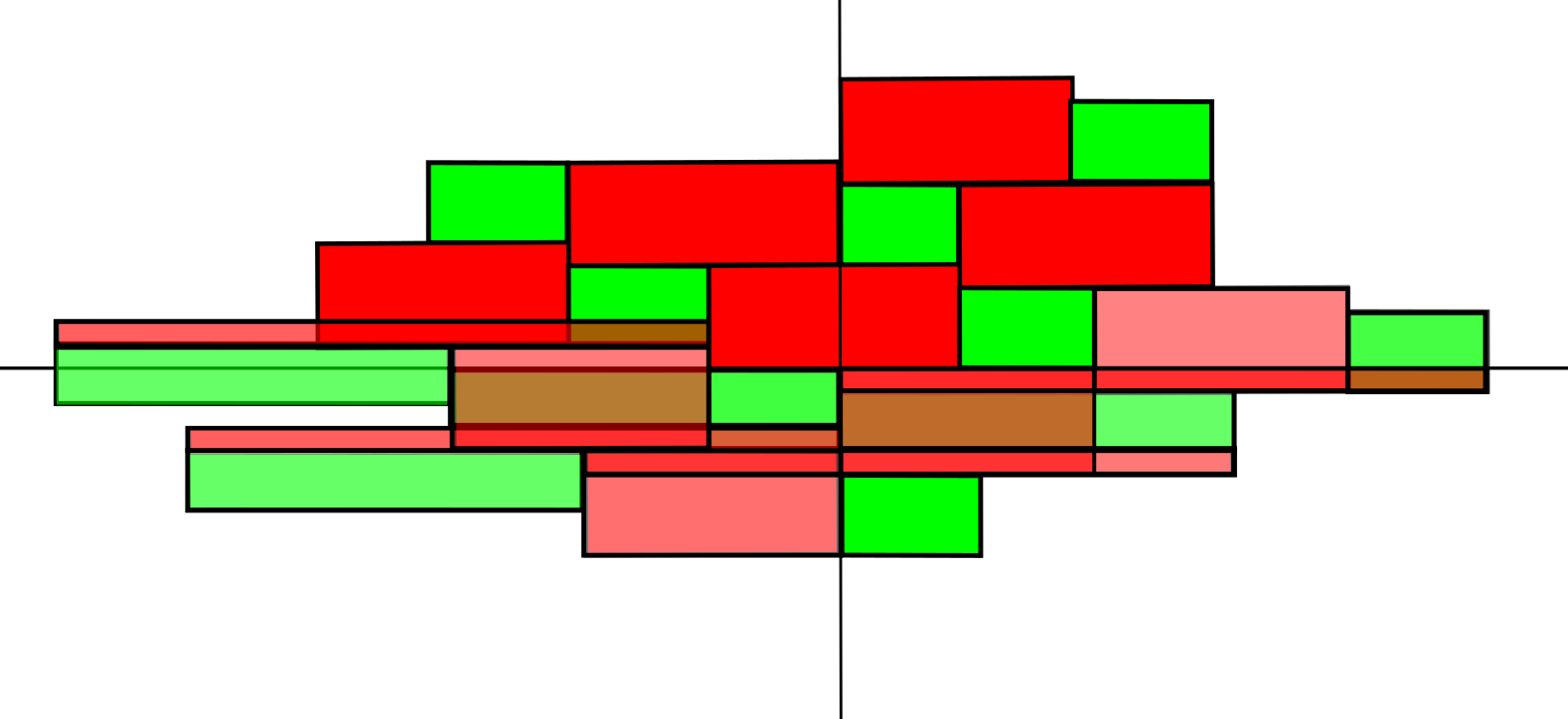}
    \caption{A strong Markovian family in $\mathcal{P}$}
    \label{f.markovianfamily}
\end{figure}

Strong Markovian families were originally introduced in \cite{monstre} as analogues of Markov partitions for group actions on the plane. Similarly to Markov partitions, these families of rectangles provide a combinatorial description of a strong Markovian action (see \cite{monstre}). Furthermore, they play an important role in the problem of classification of Anosov flows in dimension 3, as by Proposition 3.15 and Remark 3.30 of \cite{monstre},

\begin{prop}\label{p.anosovadmitsmarkovfamily}
    If $\Phi$ is a topological Anosov flow on a closed 3-manifold $M^3$, then the action of $\pi_1(M)$ on the bifoliated plane of $\Phi$ preserves a strong Markovian family.
\end{prop}

Fix for the rest of this section $\mathcal{R}$ a strong Markovian family of $\rho$. Even though each rectangle in $\mathcal{R}$ intersects infinitely many others, it is nevertheless possible to partially order the rectangles of a strong Markovian family, thanks to the following result. This allows us to define the notion of predecessor (resp. successor) of first generation of a rectangle $R\in \mathcal{R}$, which can be intuitively thought as one of the first rectangles in $\mathcal{R}$ lying in the ``past" (resp. ``future") of $R$ (compare with the notion of first return map for a Markov partition). This point of view will be important to keep in mind during the construction of the Markov partition associated with $\mathcal{R}$.

\begin{lemm}\label{l.existenceofpredecessors}
For any rectangle $R \in \mathcal{R}$ there exists a unique finite collection of rectangles $R_1,...,R_n \in \mathcal{R}$ intersecting $R$ along non-trivial vertical subrectangles and such that:
\begin{enumerate}
\item $R_1,...,R_n$ are maximal for the previous property: any $R' \in \mathcal{R}$ intersecting $R$ along a non-trivial vertical subrectangle satisfies $R' \cap R \subseteq R_i \cap R$ for some $i \in \llbracket 1, n \rrbracket$. 
\item $R_1,...,R_n$ have disjoint interiors. 
\item The $R_1,...,R_n$ cover $R$: $\overset{n}{\underset{i=1}{\cup}} R_i \cap R = R$. 
\end{enumerate}
\end{lemm}
\noindent Naturally, the analogue of the previous lemma for horizontal subrectangles is also true. 

Lemma \ref{l.existenceofpredecessors} was proven for strong Markovian families preserved by Anosov-like strong Markovian actions in \cite{monstre} (see Lemma 4.2). However, the proof does not rely on either the Anosov-like or the strong Markovian nature of the action, and therefore carries over verbatim to our setting.

\begin{defi}\label{d.successor}
For any $R\in \mathcal{R}$, we will say that $R'$ is a \emph{predecessor} (resp. \emph{successor}) \emph{of first generation} of $R$ if $R'\cap R$ is a non-trivial vertical (resp. horizontal) subrectangle of $R$ and $R'$ is maximal for this property in the sense of the previous lemma. 

We will say that $R'$ is a predecessor of \emph{$2$-nd generation} of $R$, if $R'$ is a predecessor first generation of a predecessor of first generation of $R$. We define similarly a predecessor (resp. successor) of \emph{$n$-th generation} for any $n \geq 3$. By convention, the unique predecessor (resp. successor) of $R$ of \emph{generation $0$} is $R$ itself. 

\end{defi}

\begin{lemm}\label{l.npredecessor}
Let $R, R' \in \mathcal{R}$ be such that $R' \cap R$ is a non-trivial vertical subrectangle of $R$. Then, there exists a unique $n \in \mathbb{N}$ such that $R'$ is a predecessor of $n$-th generation of $R$. Furthermore, $R$ is a successor of $n$-th generation of $R'$ for the same integer $n$. Finally, there exists a unique sequence in $\mathcal{R}$, say $R_0=R, R_1,\ldots,R_n=R'$, such that $R_{i+1}$ is a predecessor of first generation of $R_i$ for every $i\in \llbracket 0,n-1\rrbracket$. 

\end{lemm}
\begin{lemm}\label{l.npredecessorsdisjoint}
Take $R \in \mathcal{R}$ and $N\in \mathbb{N}^*$. The predecessors (resp. successors) of $R$ of generation $N$ have disjoint interiors. 

\end{lemm}

Once again, Lemmas \ref{l.npredecessor} and \ref{l.npredecessorsdisjoint} were proven for strong Markovian families preserved by Anosov-like strong Markovian actions in \cite{monstre} (see Remark 4.4, Lemmas 4.5 and 4.6). However, the proofs given in \cite{monstre} apply in the exact same way in our setting. 

\begin{figure}[h]
    \centering
    \includegraphics[scale=0.45]{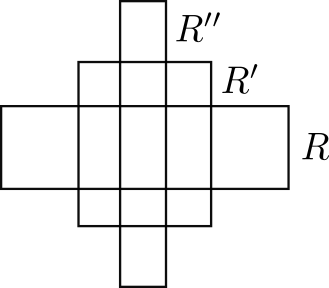}
    \caption{In the above figure, $R''$ is a predecessor of $R'$ of some generation and $R'$ is a predecessor of $R$ of some generation.}
    \label{f.predecessorsinchain}
\end{figure}

We finish this section by remarking that the relation ``being a predecessor of some generation" defines a partial order on $\mathcal{R}$: 

\begin{rema}\label{r.precedentsuivant}
Take $R, R', R'', R''' \in \mathcal{R}$ such that $R'$ is a predecessor of $n$-th generation of $R$, $R''$ is a predecessor of $k$-th generation of $R'$ and $R'''$ is a successor of $l$-th generation of $R'$ (see Figure \ref{f.predecessorsinchain}). Thanks to Lemmas \ref{l.npredecessor} and \ref{l.npredecessorsdisjoint}, we have that: 
\begin{enumerate}  
    \item $R''$ is a predecessor of $(n+k)$-generation of $R$.
    \item If $\mathrm{Int}(R)\cap \mathrm{Int}(R''')\neq \emptyset$ and $l<n$, then $R'''$ is a predecessor of $R$ of generation $n-l$.
    \item If $\mathrm{Int}(R)\cap \mathrm{Int}(R''')\neq \emptyset$ and $l=n$, then $R'''=R$. 
    \item If $\mathrm{Int}(R)\cap \mathrm{Int}(R''')\neq \emptyset$ and $l>n$, then $R'''$ is a successor of $R$ of generation $l-n$.
     \item For every $g\in G$ we have that $\rho(g)(R')$ is a predecessor of $n$-th generation of $\rho(g)(R)$.
\end{enumerate}

\end{rema}
\subsection{Strong Markovian actions on the plane}

Consider $\rho:G\rightarrow \text{Homeo}(\mathcal{P})$ an orientation preserving strong Markovian action (see Definition \ref{d.markovianaction}) on the bifoliated plane $(\mathcal{P},\mathcal{F}^s,\mathcal{F}^u)$ preserving a strong Markovian family $\mathcal{R}$. Denote by $e$ the trivial element in $G$. 

Although the definition of a strong Markovian action is combinatorial in nature, strong Markovian actions are known to exhibit hyperbolic behavior, as illustrated by Theorem \ref{t.markovianisanosovlike} and Lemma \ref{l.longrectangles} below. More specifically, $\rho$ admits a plethora of periodic points in $\mathcal{P}$, each of which is of saddle type, and for each $R\in\mathcal{R}$ the predecessors (resp. successors) of $R$ of generation $N$ become arbitrarily long in the direction of $\mathcal{F}^u$ (resp. $\mathcal{F}^s$) as $N\to+\infty$.

\begin{defi}
Consider $x\in \mathcal{P}$. We will say that $x$ is a \emph{periodic point for $\rho$} if there exists $g\in G-\{e\}$ such that $\rho(g)(x)=x$. Similarly, for any leaf $l\in \mathcal{F}^s\cup \mathcal{F}^u$ we will say that $l$ is a \emph{periodic leaf for $\rho$} if there exists $g\in G-\{e\}$ such that $\rho(g)(l)=l$. 
\end{defi}

\begin{defi}\label{d.topologicalexpansion}
    Let $g\in G$ and $l\in \mathcal{F}^s$. We will say that $\rho(g)$ acts as a \emph{topological contraction} (resp. \emph{expansion}) on $l$ if there exists $x_0\in l$ such that for every $x\in l$ the sequence $\rho(g^n)(x)$ converges to $x_0$ as $n\rightarrow +\infty$ (resp. as $n\rightarrow -\infty$). 
\end{defi}

\begin{theorem}[Theorem 1.5, Lemmas 3.3 and 3.6,   \cite{markovianactions}]\label{t.markovianisanosovlike}
Consider $\rho:G\rightarrow \text{Homeo}(\mathcal{P})$ an orientation preserving strong Markovian action on the bifoliated plane $(\mathcal{P},\mathcal{F}^s,\mathcal{F}^u)$ and $\mathcal{R}$ a strong Markovian family preserved by $\rho$. We have that: 

    \begin{enumerate}
         \item For every $R\in \mathcal{R}$ and every $g\in G$, if $\rho(g)(R)=R$, then $g=e$.

         \vspace{0.2cm}
         \item If a leaf $l\in \mathcal{F}^s$ (resp. $l\in \mathcal{F}^u$) contains an $\mathcal{F}^s$-boundary (resp. $\mathcal{F}^u$-boundary) component of some rectangle in $\mathcal{R}$, then $l$ is periodic for $\rho$. 

         \vspace{0.2cm}
        \item If $l\in \mathcal{F}^s \cup \mathcal{F}^u$ is periodic for $\rho$, then $l$ contains a unique periodic point for $\rho$. 

         \vspace{0.2cm}
        \item For any $g\in G-\{e\}$ and any $x\in \mathcal{P}$ fixed by $\rho(g)$, up to changing $g$ to $g^{-1}$, we have that $\rho(g)$ is 
        topologically expanding on $\mathcal{F}^u(x)$ and topologically contracting on $\mathcal{F}^s(x)$.
       \end{enumerate} 

\end{theorem} 

\begin{lemm}[Lemmas 3.8 and 3.12, \cite{markovianactions}] \label{l.longrectangles}
 Consider $\rho:G\rightarrow \text{Homeo}(\mathcal{P})$ an orientation preserving strong Markovian action on the bifoliated plane $(\mathcal{P},\mathcal{F}^s,\mathcal{F}^u)$ and $\mathcal{R}$ a strong Markovian family preserved by $\rho$. Fix $l$ a leaf in $\mathcal{F}^s$ and $D$ be the closure of a connected component of $\mathcal{P} - l$. 
   \begin{enumerate}
       \item If $l$ is periodic, let $p$ be the unique periodic point for $\rho$ contained in $l$. Then, we have that
       
       \begin{itemize}
        \item for every $x,y\in l$ belonging in the same $\mathcal{F}^s$-separatrix of $p$ in $l$ there exists $R\in \mathcal{R}$ containing the points $x,y$ and satisfying $\mathrm{Int}(R)\cap D\neq \emptyset$. 
        \item for any $x,y\in l$ belonging in different  $\mathcal{F}^s$-separatrices of $p$ in $l$ there exists $R\in \mathcal{R}$ containing the points $x,y$ and satisfying $\mathrm{Int}(R)\cap D\neq \emptyset$ if and only if there exists $R'\in \mathcal{R}$ containing $p$ and a neighborhood of $p$ inside $D$.

        \end{itemize}

        \vspace{0.2cm}
        \item If $l$ is not periodic, then we have that for every $x,y\in l$ there exists $R\in \mathcal{R}$ containing the points $x,y$ and satisfying $\mathrm{Int}(R)\cap D\neq \emptyset$.

   \end{enumerate} 
   
  \end{lemm}

The proof of our main result differs substantially depending on whether the action $\rho$ preserves orientations of the foliations $\mathcal{F}^s$ and $\mathcal{F}^u$. This motivates the following terminology:

\begin{defi}\label{d.orientablemarkovaction}
We say that $\rho$ is \emph{fully orientable} if it preserves every choice of orientation of $\mathcal{F}^s$ and $\mathcal{F}^u$.
\end{defi}

We conclude this section with the following result, which shows that the corner condition does not limit the scope of Theorem~\ref{t.main}.

\begin{defi}
We say that the strong Markovian family $\mathcal{R}$ satisfies the \emph{corner condition} if, for every $R\in \mathcal{R}$, the periodic points of $\rho$ lying in $\partial R$ are corners of $R$.

Furthermore, we say that a strong Markovian family $\mathcal{R}'$ is \emph{finer than $\mathcal{R}$} if every rectangle in $\mathcal{R}'$ is a subset of a rectangle in $\mathcal{R}$. 
\end{defi}
\begin{prop}\label{p.cornerconditionexists}
    Let $\rho:G\rightarrow \text{Homeo}(\mathcal{P})$ be a strong Markovian action on the bifoliated plane $(\mathcal{P},\mathcal{F}^s,\mathcal{F}^u)$ preserving a strong Markovian family $\mathcal{R}$. There exists a strong Markovian family $\widehat{\mathcal{R}}$ of $\rho$ that is finer than $\mathcal{R}$ and that satisfies the corner condition.
\end{prop}
\begin{proof}
    Consider $R\in \mathcal{R}$. If an $\mathcal{F}^s$-boundary component of $R$ contains a periodic point $p$ of $\rho$ in its interior, then cut $R$ along $\mathcal{F}^u(p)\cap R$ to obtain two rectangles (recall that $\mathcal{F}^u(p)\cap R$ is connected, thanks to Proposition \ref{p.propertiesoffoliformarkovianactions}). Since the $\mathcal{F}^s$-boundary of $R$ contains at most two periodic points of $\rho$ (see Theorem \ref{t.markovianisanosovlike}), after cutting $R$, if necessary, along the one or two $\mathcal{F}^u$-leaves through these points, we obtain one, two, or three rectangles whose $\mathcal{F}^s$-boundaries contain no periodic points in their interiors. Denote by $\mathcal{S}$ the set of rectangles constructed in this way. We will show that $\mathcal{S}$ is a strong Markovian family of $\rho$.

    First, since $\mathcal{R}$ covers the plane $\mathcal{P}$, it easy to see that the finer collection of rectangles $\mathcal{S}$ also covers $\mathcal{P}$. Next, the action $\rho$ preserves its set of periodic points, as well as the foliations $\mathcal{F}^s$ and $\mathcal{F}^u$. Therefore, for every $g\in G$ and every $R\in \mathcal{R}$, if $p$ is a periodic point of $\rho$ lying in the interior of $\partial^s R$, then $\rho(g)(p)$ is a periodic point of $\rho$ lying in the interior of $\partial^s \rho(g)(R)$. In particular, if during the construction of $\mathcal{S}$ the rectangle $R$ was cut along $\mathcal{F}^u(p)\cap R$, then $\rho(g)(R)$ was also cut along $$\rho(g)(\mathcal{F}^u(p)\cap R)=\mathcal{F}^u(\rho(g)(p))\cap \rho(g)(R)$$

We deduce that $\mathcal{S}$ is invariant by $\rho$. Moreover, if $R_1,\ldots,R_n$ form a set of representatives of the rectangle orbits in $\mathcal{R}$, then each $R_i$ gives rise to finitely many rectangles in $\mathcal{S}$ after being cut along finitely many of its $\mathcal{F}^u$-leaves. It is easy to see that the resulting finite collection of rectangles in $\mathcal{S}$ forms a set of representatives of every $\rho$-orbit of rectangles in $\mathcal{S}$.

Let us now show that $\mathcal{S}$ satisfies the Markovian intersection axiom (see Definition \ref{d.markovfamily}). Assume, by contradiction, that there exist $S_1,S_2\in \mathcal{S}$ such that $\mathrm{Int}(S_1)\cap \mathrm{Int}(S_2)\neq \emptyset$ and such that $S_1$ and $S_2$ do not intersect in a Markovian way. By construction, $S_1$ (resp. $S_2$) is a vertical subrectangle of some rectangle $R_1$ (resp. $R_2$) in $\mathcal{R}$. Using our construction and the fact that $\mathrm{Int}(S_1)\cap \mathrm{Int}(S_2)\neq \emptyset$, we get that $R_1\neq R_2$ and that $\mathrm{Int}(R_1)\cap \mathrm{Int}(R_2)\neq \emptyset$. By the Markovian intersection axiom, assume without any loss of generality that $R_1\cap R_2$ is a non-trivial horizontal subrectangle of $R_2$ and a non-trivial vertical subrectangle of $R_1$ (see Figure \ref{f.cutting}). We now distinguish three cases. 

\vspace{0.2cm}
\textbf{Case 1: The interior of $\partial^s R_1$ contains no periodic points of $\rho$}

In this case, $R_1=S_1$. Moreover, $R_1$ intersects every vertical subrectangle of $R_2$ in a Markovian way, which implies that, contrary to our original hypothesis, $S_1$ and $S_2$ also intersect in a Markovian way. 

\vspace{0.2cm}
\textbf{Case 2: The interior of $\partial^s R_1$ contains exactly one periodic point of $\rho$}

Denote the previous periodic point by $p$. Assume first that $p\notin R_2$. In this case, $R_1$ is cut into two rectangles exactly one of which intersects the interior $R_2$ and thus coincides with $S_1$. Similarly to Case 1, $S_1$ intersects every vertical subrectangle of $R_2$ in a Markovian way. Hence, $S_1$ and $S_2$ also intersect in a Markovian way, which is absurd. 

Assume next that $p\in R_2-\partial^s R_2$ (see Figure \ref{f.impossibleintersection}). We will show that this case is impossible. Indeed, consider $g\in \text{Stab}_{\rho}(p)$ such that $\rho(g)$ acts as a contraction on $\mathcal{F}^u(p)$ and as an expansion on $\mathcal{F}^s(p)$ (see Theorem \ref{t.markovianisanosovlike}). For $n$ sufficiently large, we have that $\rho(g^n)(R_2)\in \mathcal{R}$ is a rectangle that is very thin in the direction of $\mathcal{F}^u$ and very long in the direction of $\mathcal{F}^s$. It follows that $\rho(g^n)(R_2)$ intersects $R_1$ as in Figure \ref{f.impossibleintersection}, which contradicts the fact that $\mathcal{R}$ satisfies the Markovian intersection axiom. 

\begin{figure}[h!]
    \centering
    \includegraphics[width=0.5\linewidth]{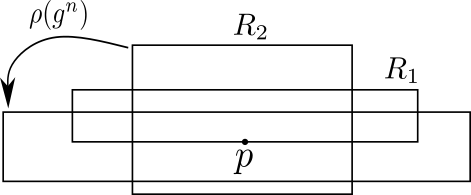}
    \caption{A periodic point in the boundary of $R_1$ can not belong in the interior of $R_2$.}
    \label{f.impossibleintersection}
\end{figure}

Finally, assume that $p\in \partial^s R_2$ (see Figure \ref{f.cutting}). In this case, $R_1$ is cut along $\mathcal{F}^u(p)$ into two rectangles exactly one of which coincides with $S_1$. Similarly, $R_2$ is cut along $\mathcal{F}^u(p)$ into two rectangles exactly one of which, say $R_2'$,  intersects $S_1$. By construction, $S_2$ either coincides with $R_2'$ or is a vertical subrectangle of $R_2'$. However, the rectangle $S_1$ intersects in a Markovian way every vertical subrectangle of $R_2'$, which proves once again that $S_1$ and $S_2$ intersect in a Markovian way and leads to an absurdity. 

\begin{figure}
    \centering
    \includegraphics[width=0.45\linewidth]{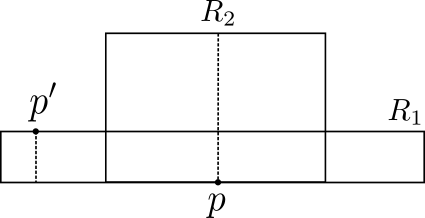}
    \caption{Cutting $R_1$ along two $\mathcal{F}^u$-leaves.}
    \label{f.cutting}
\end{figure}

\vspace{0.2cm}
\textbf{Case 3: The interior of $\partial^s R_1$ contains exactly two periodic points of $\rho$}\nopagebreak

    Denote by $p$ and $p'$ the previous two periodic points. Recall that by our previous arguments, $p,p'\notin R_2-\partial^sR_2$. Therefore, by exchanging the roles of $p$ and $p'$ if necessary, either $p,p'\notin R_2$, or $p\in \partial^s R_2$ and $p'\notin R_2$, or $p,p'\in \partial^sR_2$. The first two cases can be treated similarly to Case 2. Let us now prove that the third case is impossible. 

Indeed, if $p,p'\in \partial^sR_2$, then by Theorem \ref{t.markovianisanosovlike} and the fact that $R_1$ and $R_2$ intersect in a Markovian way, we have that each $\mathcal{F}^s$-boundary component of $R_2$ is included in an $\mathcal{F}^s$-boundary component of $R_1$. However, this contradicts the Markovian intersection axiom, according to which $R_1\cap R_2$ is a \emph{non-trivial} horizontal subrectangle of $R_2$ and a \emph{non-trivial} vertical subrectangle of $R_1$.

\vspace{0.2cm}
We have thus shown that $\mathcal{S}$ satisfies the Markovian intersection axiom. Moreover, our previous arguments show that, for any $S,S'\in \mathcal{S}$ with $\mathrm{Int}(S)\cap \mathrm{Int}(S')\neq \emptyset$, if $S$ is a vertical subrectangle of $R\in \mathcal{R}$ and $S'$ is a vertical subrectangle of $R'\in \mathcal{R}$, then $S\cap S'$ is a vertical subrectangle of $S$ if and only if $R\cap R'$ is a vertical subrectangle of $R$. Thanks to the previous fact, it is not difficult to check that $\mathcal{S}$ satisfies the strong finite return and the expansivity axioms.

We deduce that $\mathcal{S}$ is a strong Markovian family of $\rho$. By construction, the rectangles of $\mathcal{S}$ do not contain periodic points of $\rho$ in the interior of their $\mathcal{F}^s$-boundaries, although they may still contain periodic points in the interior of their $\mathcal{F}^u$-boundaries. By repeating our construction with the roles of $\mathcal{F}^s$ and $\mathcal{F}^u$ reversed, we obtain a strong Markovian family of $\rho$ satisfying the corner condition. 
\end{proof}

During the proof of Proposition \ref{p.cornerconditionexists}, we showed that, after possibly cutting the rectangles of a strong Markovian family finitely many times, we can always obtain a strong Markovian family satisfying the corner condition. Therefore, restricting our attention to strong Markovian families satisfying the corner condition entails no significant loss of generality.

\section{From a fully orientable strong Markovian action to a closed $3$-manifold}\label{s.fullyorientmanifold}

Let $(\mathcal{P},\mathcal{F}^s,\mathcal{F}^u)$ be a bifoliated plane endowed with a fully orientable strong Markovian action $\rho:G\rightarrow \text{Homeo}(\mathcal{P})$. By Proposition \ref{p.cornerconditionexists}, $\rho$ preserves a strong Markovian family $\mathcal{R}$ satisfying the corner condition. 

Our goal in this section is to construct the closed manifold $M$ appearing in the statement of Theorem~\ref{t.main}, under the assumption that our original action is fully orientable. To this end, in Subsection~\ref{ss.definitionwplus} we construct a simply connected $3$-manifold $W^+$ homeomorphic to $\mathbb{R}^3$ on which $G$ acts naturally. Next, in Subsection~\ref{ss.proper-discontinuity}, we show that the natural action of $G$ on $W^+$ is free and properly discontinuous. In Subsection~\ref{ss.big-brother}, we introduce the notion of a \emph{big brother} of a rectangle in the strong Markovian family $\mathcal{R}$, which will then be used in Subsection~\ref{ss.cocompactness} to prove that $G$ acts on $W^+$ cocompactly.

\subsection{The space of positively oriented $\mathcal{F}^s$-segments $W^+$}\label{ss.definitionwplus}

Fix for the rest of this section an orientation of $\mathcal{F}^s$ and $\mathcal{F}^u$. 

\begin{defi}
     For every point $x\in \mathcal{P}$, we define $\mathcal{F}^s_+(x)$ (resp. $\mathcal{F}^s_-(x)$), the \emph{positive} (resp. \emph{negative}) \emph{$\mathcal{F}^s$-separatrix} of $x$, as the set of all $y\in \mathcal{F}^s(x)$ such that either $y=x$ or the orientation of $\mathcal{F}^s(x)$ goes from $x$ to $y$ (resp. from $y$ to $x$). We similarly define the  \emph{positive} and \emph{negative} \emph{$\mathcal{F}^u$-separatrix} of $x$. 
\end{defi}

We define the \emph{space of positively oriented $\mathcal{F}^s$-segments} in $\mathcal{P}$ as follows: 
 $$W^+:=\{(x,y)\in \mathcal{P}^2|y\in \mathcal{F}^+(x)-\{x\} \}.$$

The previous space was originally introduced by T. Marty, who used it in an unpublished work in order to recover any Anosov flow with orientable stable and unstable foliations from its bifoliated plane. In this paper, we will make use of the space $W^+$ in order to generalize the previous result by T. Marty for any bifoliated plane that is endowed with a strong Markovian action. 

The topology of $\mathcal{P}$ naturally induces a topology on $\mathcal{P}^2$ and thus also on $W^+$. Even more, 

\begin{prop}
\label{p.wplusisasimplyconnectedmanifold}
The space $W^+$ is homeomorphic to $\mathbb{R}^3$. 
\end{prop}

\begin{proof}
Indeed, by a result of H. Whitney the leaves of any $C^0$ foliation on the plane are the orbits of a continuous flow. More precisely, Theorem 27A of \cite{Whitney} states that there exists a $C^0$ map 
$f:\mathcal{P}\times \mathbb{R}\rightarrow \mathcal{P}$  such that 
\begin{itemize}
    \item the map $x\rightarrow f(x,t)$ is a homeomorphism for every $t\in\mathbb{R}$,

    \vspace{0.1cm}
    
    \item$f(x,0)=x$ and  $f(f(x,s),t) =f(x, t+s)$ for every $x\in\mathcal{P}$ and $s,t\in\mathbb{R}$, 
    
    \vspace{0.1cm}
    
    \item  $f(x,\mathbb{R}^+)=\mathcal{F}^+(x)$ for every  $x\in\mathcal{P}$.
\end{itemize}
Consider now the map
$$F : \begin{array}{rcl}
W^+ & \rightarrow & \mathcal{P}\times \mathbb{R}_{>0} \\
(x,y) & \mapsto & (x,t(x,y))
\end{array}$$
where $t(x,y)$ is the unique positive real number such that $f(x,t(x,y))=y$. The map $F$ defines a homeomorphism from $W^+$ to $\mathcal{P}\times \mathbb{R}_{>0}\simeq \mathbb{R}^3$. 
\end{proof}

Observe that the orientation preserving action $\rho:G\to\mathrm{Homeo}(\mathcal{P})$ on $\mathcal{P}$ naturally induces an orientation-preserving action on $W^+$
$$\widetilde\rho : G  \to \mathrm{Homeo}(W^+)$$
by setting $\widetilde{\rho}(g)((x,y))=(\rho(g)(x),\rho(g)(y))$. Whenever convenient, we will also use the notation $\widetilde{\rho}(g).(x,y):=\widetilde{\rho}(g)((x,y))$ for the action of $g\in G$ on $W^+$.

The remainder of the section is devoted to proving that this action of free, properly discontinuous and cocompact. It will immediately follow that the quotient $M:=\widetilde{\rho}(G)\backslash W^+$ is a closed $3$-manifold, whose fundamental group is isomorphic to $G$.

\subsection{Freeness and proper discontinuity of $\widetilde{\rho}$}
\label{ss.proper-discontinuity}

Denote by $e$ the trivial element in $G$. Moreover, for any $x\in \mathcal{P}$ and $y\in \mathcal{F}_+^s(x)$ (resp. $y\in \mathcal{F}_+^u(x)$) denote by $[x,y]^s$ (resp. $[x,y]^u$) the unique $\mathcal{F}^s$-segment (resp. $\mathcal{F}^u$-segment) going from $x$ to $y$.

\begin{prop}\label{p.actionproperdisc}
    The action $\widetilde{\rho}$ is free and properly discontinuous. 
\end{prop}

\begin{proof}
Suppose that there exist $g\in G$ and $(x,y)\in W^+$ such that $\widetilde{\rho}(g).(x,y)= (x,y)$. By definition of $\widetilde{\rho}$, this means that $\rho(g)$ fixes both $x$ and $y$. In particular, $\rho(g)$ fixes two distinct points on the same $\mathcal{F}^s$-leaf. By Theorem~\ref{t.markovianisanosovlike}, this implies that $g=e$. Therefore, the action $\widetilde{\rho}$ is free.

\bigskip

 Now we turn to the proper discontinuity of $\widetilde{\rho}$. Consider a sequence $(g_n)_{n\in \mathbb{N}}$  in $G$ and a sequence $((x_n,y_n))_{n\in \mathbb{N}}$ in $W^+$ converging to $(x,y)\in W^+$ such that $$\widetilde{\rho}(g_n).(x_n,y_n)\underset{n\rightarrow +\infty}{\longrightarrow} (X,Y).$$ In order to prove that $\widetilde{\rho}$ acts properly discontinuously, it suffices to prove that the same element of $G$ appears infinitely many times in the sequence $(g_n)_{n\in\mathbb{N}}$. Assume, by contradiction, that this is not the case. After possibly passing to a subsequence, we may further assume that the $g_n$'s are pairwise distinct.
 
By definition, the $\mathcal{F}^s$-segments $[x_n,y_n]^s$ (resp. $\rho(g_n)([x_n,y_n]^s)$) converge for the Hausdorff topology to $[x,y]^s$ (resp. $[X,Y]^s$).

\vspace{0.4cm}
\textbf{Claim.} By changing our choice of sequence $((x_n,y_n))_{n\in \mathbb{N}}$ if necessary, we can assume that the $\mathcal{F}^s$-segments $[x,y]^s$ and $[X,Y]^s$ contain no periodic points of $\rho$ in their interiors. 

\begin{proof}[Proof of the claim]
  Recall that, by Theorem~\ref{t.markovianisanosovlike}, there is at most one periodic point of $\rho$ in each of the segments $[x,y]^s$ and $[X,Y]^s$. Assume now that $[x,y]^s$ contains a periodic point $p$ in its interior. In this case, since $[x_n,y_n]^s$ converges for the Hausdorff topology to $[x,y]^s$, there exists a sequence $(p_n)_{n\in \mathbb{N}}$ in $\mathcal{P}$ such that $p_n\in [x_n,y_n]^s$ and $p_n\underset{n\rightarrow +\infty}{\longrightarrow} p$. Moreover, since $\rho(g_n)([x_n,y_n]^s)$ converges for the Hausdorff topology to $[X,Y]^s$, up to extracting a subsequence, there exists $P\in [X,Y]^s$ such that 
$$\rho(g_n)(p_n)\underset{n\rightarrow +\infty}{\longrightarrow} P $$

Assume without any loss of generality that $P\neq X$ (the case where $P=X$ can be treated by a similar argument). Under this assumption, we have
$$x_n\underset{n\rightarrow +\infty}{\longrightarrow} x, ~~p_n\underset{n\rightarrow +\infty}{\longrightarrow} p  \text{ and } (x,p)\in W^+$$ 
$$\rho(g_n)(x_n)\underset{n\rightarrow +\infty}{\longrightarrow} X,~~~ \rho(g_n)(p_n)\underset{n\rightarrow +\infty}{\longrightarrow} P \text{ and } (X,P)\in W^+$$ 
We have thus proven that, by  replacing our initial choice of sequence $(y_n)_{n\in \mathbb{N}}$ and by possibly passing to a subsequence, we can assume that $[x,y]^s$ contains no periodic point of $\rho$ in its interior. The argument for the segment $[X,Y]^s$ is similar.
\end{proof}

Using the previous claim and Lemma~\ref{l.longrectangles}, we obtain that there exist $r,R\in\mathcal{R}$ such that 

\begin{enumerate}
    \item $[x,y]^s\subset r$ and $[X,Y]^s\subset R$.
    \vspace{0.1cm}
    \item $[X,Y]^s\subsetneq \mathcal{F}^s(X)\cap R$. In other words, $[X,Y]^s$ is strictly contained in the $\mathcal{F}^s$-leaf of $X$ in $R$. 
    \vspace{0.1cm}
    \item $R\cap \rho(g_n)([x_n,y_n]^s)\neq\emptyset$ and $\mathrm{Int}(R)\cap \rho(g_n)(\mathrm{Int}(r))\neq\emptyset$  for infinitely many $n\in \mathbb{N}$.
\end{enumerate}
Passing to a subsequence if necessary, assume that the previous properties hold for every $n\in \mathbb{N}$.

Thanks to Theorem \ref{t.markovianisanosovlike}, since the $g_n$'s are pairwise distinct, the rectangles $\rho(g_n)(r)$ are also pairwise distinct. Moreover, since $\mathrm{Int}(R)\cap \rho(g_n)(\mathrm{Int}(r))\neq \emptyset$, for every $n\in \mathbb{N}$ the rectangle $\rho(g_n)(r)$ is either a successor or a predecessor of some generation of $R$ (see Lemma \ref{l.npredecessor}). We now consider the following two cases.
  
\vspace{0.2cm}  
\textbf{Case 1.} There exist infinitely many $n\in \mathbb{N}$ for which $\rho(g_n)(r)$ is a predecessor of some generation of $R$. 

By eventually considering a subsequence, assume that $\rho(g_n)(r)$ is a predecessor of some generation of $R$ for every $n\in \mathbb{N}$. Using the facts that $\rho(g_n)(r)\cap R$ is a vertical subrectangle of $R$, that $\rho(g_n)([x_n,y_n]^s)\subset \rho(g_n)(r)$, and that the segments $\rho(g_n)([x_n,y_n]^s)$ converge for the Hausdorff topology to the segment $[X,Y]^s\subset R$ as $n\rightarrow +\infty$, we get that
$$\mathrm{Int}(\rho(g_n)(r))\cap \mathrm{Int}(\rho(g_m)(r))\neq \emptyset$$
for every sufficiently large $n,m\in\mathbb{N}$.

Passing to a further subsequence, assume that 
$\mathrm{Int}(\rho(g_n)(r))\cap \mathrm{Int}(\rho(g_m)(r))\neq \emptyset$ for every $n,m\in \mathbb{N}$. Thanks to the previous fact and to Lemma~\ref{l.npredecessor}, $\rho(g_n)(r)$ is either a predecessor or a successor of some generation of $\rho(g_m)(r)$ for every $n,m\in \mathbb{N}$. Recall now that, for every $N\in \mathbb{N}$, the predecessors of $R$ of generation smaller or equal to $N$ are finitely many (see Lemma~\ref{l.existenceofpredecessors}). Our previous arguments together with Remark~\ref{r.precedentsuivant} imply that, for every $m\in \mathbb{N}$, there exist infinitely many $n\in \mathbb{N}$ such that $\rho(g_n)(r)$ is a predecessor of some generation of $\rho(g_m)(r)$.

Using the previous fact, we can construct a sequence of positive integers $(n_k)_{k\in\mathbb{N}}$ such that $\rho(g_{n_{k+1}})(r)$ is a predecessor of some generation of $\rho(g_{n_k})(r)$ for every $k\in \mathbb{N}$. By the expansivity axiom, $$\bigcap_{k\in\mathbb{N}} \rho(g_{n_k})(r)$$
is an $\mathcal{F}^u$-segment. However, this is impossible, since $\rho(g_{n_k})([x_{n_k},y_{n_k}]^s)\subset \rho(g_{n_k})(r)$ for every $k\in \mathbb{N}$ and $\rho(g_{n_k})([x_{n_k},y_{n_k}]^s)$ converges to $[X,Y]^s$ as $k\to +\infty$.
  
\vspace{0.2cm}  
\textbf{Case 2.} There exist infinitely many $n\in \mathbb{N}$ for which $\rho(g_n)(r)$ is a successor of some generation of $R$.

By eventually considering a subsequence, assume that $\rho(g_n)(r)$ is a successor of some generation of $R$ for every $n\in \mathbb{N}$. By Lemma~\ref{l.existenceofpredecessors}, for every $N\in \mathbb{N}$, the successors of $R$ of generation smaller or equal to $N$ are finitely many; hence, for every $N\in \mathbb{N}$ there exists $m\in \mathbb{N}$ such that $\rho(g_m)(r)$ is a successor of $R$ of generation greater than $N$.

Consider now a point $z$ in the interior of $[x,y]^s$ and let $V\in \mathcal{R}$ be a predecessor of $r$ of generation $M\gg 0$ (see Figure \ref{f.properdiscontinuity}) such that

\begin{itemize}
\item $z$ belongs in $V$,
\item $[x,y]^s$ intersects both $\mathcal{F}^u$-boundary components of $V$,
\item $V\cap [x,y]^s$ is a small $\mathcal{F}^s$-segment contained in the interior of $[x,y]^s$.
\end{itemize}

\noindent Such a rectangle $V$ can be obtained by a repeated application of the strong finite return axiom and the expansivity axiom. 

\vspace{0.2cm}  
\textbf{Subcase 1.} There exist infinitely many $n\in \mathbb{N}$ for which $\rho(g_n)(V)\cap \mathrm{Int}(R)\neq \emptyset$.    

In this case, take $m\in \mathbb{N}$ sufficiently large so that $[x_m,y_m]^s$ intersects both $\mathcal{F}^u$-boundary components of $V$, $\rho(g_m)(V)\cap \mathrm{Int}(R)\neq \emptyset$ and $\rho(g_m)(r)$ is a successor of $R$ of generation greater than $M+1$.

\begin{figure}[h!]
    \centering
    \includegraphics[width=0.7\linewidth]{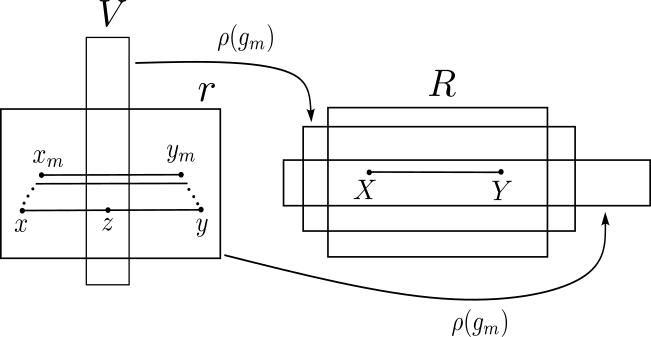}
    \caption{Under the assumptions of Subcase 1, $\rho(g_n)([x_n,y_n]^s)$ can not converge to $[X,Y]^s$.}
    \label{f.properdiscontinuity}
\end{figure}

Since $\rho(g_m)(r)$ is a successor of $R$ of generation greater than $M+1$ and $\rho(g_m)(V)$ is a predecessor of $\rho(g_m)(r)$ of generation $M$ that intersects $\mathrm{Int}(R)$, Remark~\ref{r.precedentsuivant} implies that $\rho(g_m)(V)$ is a successor of $R$ of some positive generation. Therefore, $\rho(g_m)(V\cap [x_m,y_m]^s)$ is an $\mathcal{F}^s$-segment intersecting both $\mathcal{F}^u$-boundary components of $R$. This leads to a contradiction, since $(\rho(g_n)([x_n,y_n]^s))_{n\in\mathbb{N}}$ converges for the Hausdorff topology to $[X,Y]^s\subsetneq \mathcal{F}^s(X)\cap R$ and thus, for $m$ sufficiently big, $\rho(g_m)([x_m,y_m]^s)$ can not intersect both $\mathcal{F}^u$-boundary components of $R$.

\vspace{0.2cm}  
\textbf{Subcase 2.} There exist finitely many $n\in \mathbb{N}$ for which $\rho(g_n)(V)\cap \mathrm{Int}(R)\neq \emptyset$.    

Suppose for the sake of simplicity that $\rho(g_n)(V)\cap \mathrm{Int}(R)= \emptyset$ for every $n\in \mathbb{N}$. Recall that by construction $\rho(g_n)(V)$ and $R$ are both predecessors of some generation of $\rho(g_n)(r)$. Without any loss of generality, up to taking a subsequence, we can assume that $\rho(g_n)(V)$ lies on the left of $R$ for every $n\in \mathbb{N}$, as in Figure \ref{f.properdiscontinuity2}. In this case, since $\rho(g_n)([x_n,y_n]^s)$ converges to $[X,Y]^s\subset R$, we get that $\rho(g_n)([x_n,y_n]^s\cap V)$ converges to $X$ and that $X\in \partial^uR$. 

\begin{figure}[h!]
    \centering
    \includegraphics[scale=0.5]{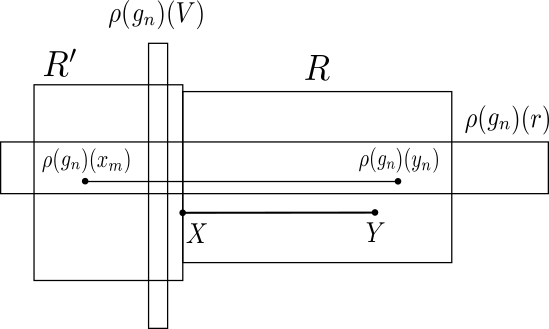}
    \caption{Under the assumptions of Subcase 2, $\rho(g_n)([x_n,y_n]^s)$ can not converge to $[X,Y]^s$.}
    \label{f.properdiscontinuity2}
\end{figure}

By the strong finite return axiom, there exists $R'\in \mathcal{R}$ such that $R'\cap \mathcal{F}^s_-(X)$ is a non-trivial $\mathcal{F}^s$-segment and for infinitely many $n\in\mathbb{N}$,
$$R'\cap \rho(g_n)([x_n,y_n]^s)\neq\emptyset,\quad
\mathrm{Int}(R')\cap \rho(g_n)(\mathrm{Int}(r))\neq\emptyset, \text{ and}\quad
\mathrm{Int}(R')\cap \rho(g_n)(\mathrm{Int}(V))\neq\emptyset$$

Passing once again to a subsequence, we may assume that the above properties hold for every $n\in\mathbb{N}$. Since the rectangles $\rho(g_n)(r)$ are successors of progressively larger generations of $R$ as $n\rightarrow+\infty$, the expansivity axiom implies that they become ``thinner'' in the $\mathcal{F}^u$-direction and ``longer'' in the $\mathcal{F}^s$-direction as $n\rightarrow+\infty$. Therefore, by the Markovian intersection axiom and Lemma~\ref{l.npredecessor}, for every sufficiently large $n$, the rectangle $\rho(g_n)(r)$ is a successor of some generation of $R'$.  

Similarly to the proof of Subcase~1, using the fact that $R'$ has only finitely many successors of generation at most $N$ (see Lemma~\ref{l.existenceofpredecessors}), we obtain that, for every $N\in\mathbb{N}$, there exists $m\in\mathbb{N}$ such that $\rho(g_m)(r)$ is a successor of $R'$ of generation greater than $N$.

Take $m\in \mathbb{N}$ sufficiently large so that $[x_m,y_m]^s$ intersects both $\mathcal{F}^u$-boundary components of $V$, $\rho(g_m)(V)\cap \mathrm{Int}(R')\neq \emptyset$ and $\rho(g_m)(r)$ is a successor of $R'$ of generation greater than $M+1$. Since $\rho(g_m)(V)$ is a predecessor of $\rho(g_m)(r)$ of generation $M$ that intersects $\mathrm{Int}(R')$, Remark~\ref{r.precedentsuivant} implies that $\rho(g_m)(V)$ is a successor of $R'$ of some positive generation. Therefore, $\rho(g_m)(V\cap [x_m,y_m]^s)$ is an $\mathcal{F}^s$-segment intersecting both $\mathcal{F}^u$-boundary components of $R'$. This leads to a contradiction, since $(\rho(g_n)([x_n,y_n]^s))_{n\in\mathbb{N}}$ converges for the Hausdorff topology to the segment $[X,Y]^s\subsetneq \mathcal{F}^s(X)\cap R$ and thus for $m$ sufficiently big $\rho(g_m)([x_m,y_m]^s)$ can not intersect both $\mathcal{F}^u$-boundary components of $R'$.

\end{proof}

\subsection{The big brother of a rectangle in $\mathcal{R}$}
\label{ss.big-brother}
Using the fact that $\mathcal{F}^s$ is orientable, for any $R\in \mathcal{R}$ we can write $\partial^u R$ as the union of two connected components, $\partial^u_- R$ and $\partial^u_+ R$, in such a way that every $\mathcal{F}^s$-segment in $R$ going from $\partial^u_- R$ to $\partial^u_+ R$ is positively oriented. We call $\partial^u_- R$ and $\partial^u_+ R$ the \emph{negative} and \emph{positive} $\mathcal{F}^u$-boundary components of $R$, respectively.

\begin{thdef}
\label{thdef.bigbrother}
    Consider any rectangle $R\in \mathcal{R}$. There exists a unique rectangle $\mathcal{B}(R)\in \mathcal{R}$ such that $R\cap \mathcal{B}(R)=\partial^u_+ R$ and which is minimal with respect to this property in the following sense:
    $$\forall R'\in \mathcal{R}\quad\quad R\cap R'=\partial^u_+ R\implies \mathcal{B}(R)\cap \mathcal{F}^u(\partial^u_+ R)\subseteq R'\cap \mathcal{F}^u(\partial^u_+ R)$$
We will call $\mathcal{B}(R)$ the \emph{big brother} of $R$. 
\end{thdef}

\begin{proof}
  We begin by proving the existence of a rectangle $R_0\in \mathcal{R}$ for which $R\cap R_0=\partial^u_+ R$. Take $D$ to be the closure of the connected component of $\mathcal{P}-\mathcal{F}^u(\partial^u_+ R)$ for which $R\cap D= \partial^u_+ R$. Thanks to Lemma \ref{l.longrectangles} and to the fact that $\mathcal{R}$ satisfies the corner condition, we get that there exists $R_0\in \mathcal{R}$ containing $\partial^u_+ R$ and such that $\mathrm{Int}(R_0)\cap D\neq \emptyset$. By the Markovian intersection axiom, if $\mathrm{Int}(R_0)\cap \mathrm{Int}(R)\neq \emptyset$, then the fact that $\partial^u_+ R\subset R_0$ would force $R_0\cap R$ to be a vertical subrectangle of $R$, contradicting our assumption that $\mathrm{Int}(R_0)\cap D\neq \emptyset$. It follows that $R\cap R_0=\partial^u_+ R$. 

\begin{figure}[h!]
    \centering
    \includegraphics[width=0.5\linewidth]{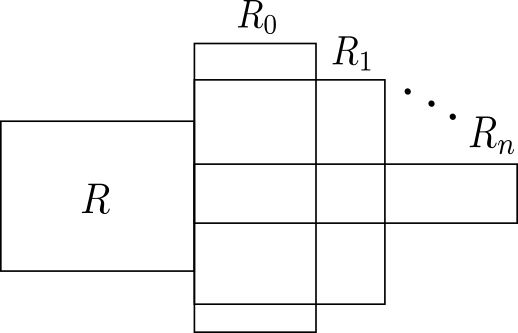}
    \caption{The non-existence of a minimal rectangle with property $(\star)$, contradicts the expansivity axiom.}
    \label{f.existencebigbrother}
\end{figure}

Next, for the sake of simplicity, and for the remainder of this proof, we say that a rectangle $r\in \mathcal{R}$ satisfies property $(\star)$ if $r\cap R=\partial^u_+ R$. Let us now prove that there exists a rectangle in $\mathcal{R}$ that is minimal for $(\star)$, in the sense of the above statement. 

Indeed, if $R_0$ is not minimal for $(\star)$, then there exists $R_1\in \mathcal{R}$ satisfying $(\star)$ and such that $ R_1\cap \mathcal{F}^u(\partial^u_+ R)\subsetneq R_0\cap \mathcal{F}^u(\partial^u_+ R)$ (see Figure \ref{f.existencebigbrother}). The previous inclusion together with the Markovian intersection axiom imply that $R_1\cap R_0$ is a horizontal subrectangle of $R_0$ and thus that $R_1$ is a successor of some positive generation of $R_0$ (see Lemma \ref{l.npredecessor}). If there does not exist a rectangle in $\mathcal{R}$ that is minimal for $(\star)$, then by repeatedly applying the previous argument, we obtain a sequence $R_0,...,R_n,...$ such that $R_{n+1}$ is a successor of some positive generation of $R_n$ and $\partial^u_+ R\subset R_n$ for every $n\in \mathbb{N}$. By the expansivity axiom, we have that $\underset{i\geq 1}{\cap}R_i$ is an $\mathcal{F}^s$-leaf of $R_0$ containing $\partial^u_+ R$, which is absurd. We deduce that there exists a rectangle in $\mathcal{R}$ that is minimal for $(\star)$.

Finally, assume that there exist two rectangles in $\mathcal{R}$, say $R'$ and $R''$, that are minimal for $(\star)$. This implies that $ R'\cap \mathcal{F}^u(\partial^u_+ R)= R''\cap \mathcal{F}^u(\partial^u_+ R)$. Thanks to the previous fact, since $R'$ and $R''$ intersect in a Markovian way, we have that either $R'\subseteq R''$ or $R''\subseteq R'$. This contradicts the Markovian intersection axiom (see Definition \ref{d.markovfamily}), according to which, up to exchanging the roles of $R'$ and $R''$, the set $R'\cap R''$ is a \emph{non-trivial} vertical subrectangle of $R'$ and a \emph{non-trivial} horizontal subrectangle of $R''$. We deduce that $R'=R''$, which finishes the proof of the desired result. 
\end{proof}

We remark at this point that the proof of Theorem-Definition~\ref{thdef.bigbrother} is the first place in this paper where we use the fact that $\mathcal{R}$ satisfies the corner condition. More precisely,

\begin{rema}\label{r.cornercondition}
Consider a strong Markovian family $\mathcal{R}'$ for $\rho$ that does not satisfy the corner condition, and fix $R\in \mathcal{R}'$ such that $\partial^u_+R$ contains a periodic point in its interior. Then the rectangle $R$ need not admit a big brother. In this case, however, one can show that there exists a unique pair of rectangles $R^{\mathcal{B}1},R^{\mathcal{B}2}$ such that
$$
R\cap (R^{\mathcal{B}1}\cup R^{\mathcal{B}2})=\partial^u_+R
$$
and such that $R^{\mathcal{B}1}$ and $R^{\mathcal{B}2}$ are minimal with respect to this property in the following sense:
$$
\forall R_1,R_2\in \mathcal{R}'\quad
R\cap (R_1\cup R_2)=\partial^u_+R
\implies
(R^{\mathcal{B}1}\cup R^{\mathcal{B}2})\cap \mathcal{F}^u(\partial^u_+R)
\subseteq
(R_1\cup R_2)\cap \mathcal{F}^u(\partial^u_+R).
$$
\end{rema}

The proof of the previous result is an adaptation of our proof of Theorem-Definition~\ref{thdef.bigbrother} and can be used to generalize Theorem~\ref{t.main} for strong Markovian families that do not satisfy the corner condition. However, such a generalization would considerably increase the combinatorial complexity of the arguments in the remainder of this paper. We will therefore restrict our attention to strong Markovian families satisfying the corner condition.

\begin{defi}
    A \emph{birectangle} is a pair $(R,\mathcal{B}(R))\in\mathcal{R}^2$, where $R\in\mathcal{R}$ and $\mathcal{B}(R)$ denotes the big brother of $R$ in the sense of Theorem-Definition~\ref{thdef.bigbrother}. We denote by $\mathcal{R}^{\mathcal{B}}$ the set of all birectangles.
\end{defi}

Since $\rho$ preserves the orientation of $\mathcal{F}^s$, it is easy to see, thanks to the uniqueness in Theorem-Definition~\ref{thdef.bigbrother}, that

\begin{rema}
\label{r.invariantbigbrothers}
    The diagonal action of $\rho$ on $\mathcal{R}^2$ leaves 
$\mathcal{R}^{\mathcal{B}}$ invariant. 
\end{rema}

Furthermore, $\mathcal{R}^{\mathcal{B}}$ inherits from $\mathcal{R}$ a Markovian-type behavior. More specifically, we have the following result, which constitutes the key argument in the proof of the cocompactness of the action of $G$ on $W^+$:

\begin{prop}
\label{p.markovianbirectangles}
    Fix $(x,y)\in W^+$. Let $L:=\mathcal{F}^s(x)=\mathcal{F}^s(y)$ and $D$ be the closure of a connected component of $\mathcal{P}-L$. We consider the following family of segments (see Figure \ref{f.case1proofprop38}): 
    \begin{align*}
        \mathcal{I}(x,y)&:=\{(R\cup \mathcal{B}(R))\cap L| ~R\in \mathcal{R}, ~x\in R-\partial^u_+R, ~ [x,y]^s\subset R\cup \mathcal{B}(R) \text{ and } R\cap \mathrm{Int}(D)\neq \emptyset\}
        \end{align*}
    
    We have that: 
    \begin{enumerate}
        \item $\mathcal{I}(x,y)  \neq \emptyset $. 
        \item The inclusion $\subseteq$ defines a total order on  $\mathcal{I}(x,y)$. In other words, for each $I, J\in \mathcal{I}(x,y)$ either $I\subseteq J$ or $J\subseteq I$.
        \item There exists a unique $I_{min}(x,y)\in \mathcal{I}(x,y)$ such that 
        $$ \forall I\in \mathcal{I}(x,y) ~~ I_{min}(x,y)\subseteq I$$
    \end{enumerate}
\end{prop}

\begin{proof}[Proof of Item (1)]
   
   If $[x,y]^s$ contains no periodic point of $\rho$ in its interior, then by Lemma \ref{l.longrectangles} there exists $r\in \mathcal{R}$ such that $r\cap \mathrm{Int}(D)\neq \emptyset$ and $[x,y]^s\subset r$. Clearly, the segment $(r\cup \mathcal{B}(r))\cap L $ belongs in $\mathcal{I}(x,y)$. 
   
   Assume now that there exists $p$ a periodic point of $\rho$ in the interior of $[x,y]^s$ (see Figure \ref{f.existencebirectangle}). Once again by Lemma \ref{l.longrectangles}, there exists $r'\in \mathcal{R}$ such that $r'\cap \mathrm{Int}(D)\neq \emptyset$ and $[x,p]^s\subset r'$. Notice that $r'\cup \mathcal{B}(r')$ contains a neighborhood of $p$ in $L$. Consider now $g\in \text{Stab}_{\rho}(p)$ such that $\rho(g)$ acts as an expansion on $\mathcal{F}^s(p)$ (see Theorem \ref{t.markovianisanosovlike}). By Remark \ref{r.invariantbigbrothers}, for any $N\in \mathbb{N}$ $$\mathcal{B}(\rho(g^N)(r'))=\rho(g^N)(\mathcal{B}(r')).$$ Hence, for $N$ sufficiently big, we have that $[x,y]^s\subset \rho(g^N)(r')\cup \mathcal{B}(\rho(g^N)(r'))$. We deduce that the segment $\big(\rho(g^N)(r')\cup \mathcal{B}(\rho(g^N)(r'))\big)\cap L$ belongs in $\mathcal{I}(x,y)$. 

\begin{figure}[h!]
    \centering
    \includegraphics[scale=0.65]{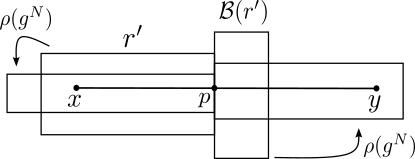}
    \caption{The set $\mathcal{I}(x,y)$ is non-empty.}
    \label{f.existencebirectangle}
\end{figure}
\end{proof}

\begin{proof}[Proof of Item (2)]
   
   Consider $I, J\in \mathcal{I}(x,y)$ with $I\neq J$. Recall that by definition, $[x,y]^s\subset I\cap J$. Fix $R,S\in \mathcal{R}$ such that $(R\cup \mathcal{B}(R))\cap L=I$, $(S\cup \mathcal{B}(S))\cap L=J$, $R\cap \mathrm{Int}(D)\neq \emptyset$, $S\cap \mathrm{Int}(D)\neq \emptyset$ and $x\in (R-\partial_+^uR) \cap (S-\partial_+^uS)$. 
   
   The above properties imply that $\mathrm{Int}(R)\cap \mathrm{Int}(S)\neq \emptyset$. Thanks to the previous fact and the Markovian intersection property, after possibly exchanging the roles of $R$ and $S$, we may assume without loss of generality that $R\cap S$ is a non-trivial vertical subrectangle of $R$. We now proceed by considering two cases.

\vspace{0.2cm}
\textbf{Case 1.} $\mathcal{F}^u(\partial^u_+R)\neq \mathcal{F}^u(\partial^u_+S)$

In this case, as illustrated in Figure~\ref{f.case1proofprop38}, we have that
$$\mathrm{Int}(\mathcal{B}(S))\cap \mathrm{Int}(R)\neq \emptyset.$$
By the Markovian intersection axiom and the fact that $S$ and $\mathcal{B}(S)$ have disjoint interiors, it follows that $R\cap \mathcal{B}(S)$ is also a non-trivial vertical subrectangle of $R$. Consequently, thanks to Proposition \ref{p.propertiesoffoliformarkovianactions}, $$ J= (S\cup \mathcal{B}(S))\cap L\subset R\cap L\subset (R\cup \mathcal{B}(R))\cap L=I$$
which proves the desired result. 

\begin{figure}[h!]
    \centering
    \includegraphics[scale=0.7]{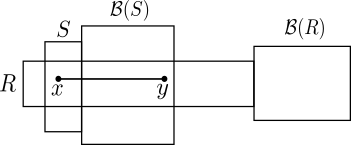}
    \caption{The configuration of $R$, $S$, and their big brothers in the case where $\mathcal{F}^u(\partial^u_+R)\neq \mathcal{F}^u(\partial^u_+S)$.}
    \label{f.case1proofprop38}
\end{figure}

\vspace{0.2cm}
\textbf{Case 2.} $\mathcal{F}^u(\partial^u_+R)= \mathcal{F}^u(\partial^u_+S)$

 Since $R\cap S$ is a vertical subrectangle of $R$, in this case we have that $\partial^u_+R\subseteq \partial^u_+S$ (see Figure \ref{f.case2proofprop38}). Thanks to the previous fact and to the definition of a big brother, $\partial^u_-\mathcal{B}(R)\subseteq \partial^u_-\mathcal{B}(S)$. Therefore, by the Markovian intersection property, either $\mathcal{B}(R)=\mathcal{B}(S)$ or  $\mathcal{B}(R)\cap \mathcal{B}(S)$ is a non-trivial vertical subrectangle of $\mathcal{B}(R)$. In both of the previous cases we have that $$ S\cap L\subset R\cap L \text{ and } \mathcal{B}(S)\cap L\subseteq \mathcal{B}(R)\cap L$$ which implies that $J\subset I$. 

\begin{figure}[h!]
    \centering
    \includegraphics[scale=0.65]{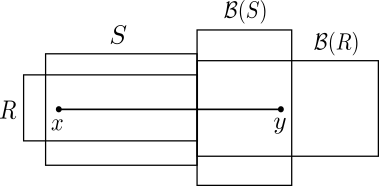}
    \caption{The configuration of $R$, $S$, and their big brothers in the case where $\mathcal{F}^u(\partial^u_+R)= \mathcal{F}^u(\partial^u_+S)$.}
    \label{f.case2proofprop38}
\end{figure}

\end{proof}

\begin{proof}[Proof of Item (3)]

Assume by contradiction that $\mathcal{I}(x,y)$ does not contain a minimal element. Since $\mathcal{I}(x,y)\neq \emptyset$, this implies that there exists $(I_n)_{n\in \mathbb{N}}$ a sequence of segments in $\mathcal{I}(x,y)$
such that $I_{n+1}\subsetneq I_n$ for every $n\in \mathbb{N}$.

By definition, $[x,y]^s\subset I_n$ for every $n\in \mathbb{N}$. Consider now $R_n\in \mathcal{R}$ such that $(R_n\cup \mathcal{B}(R_n))\cap L=I_n$, $R_n\cap \mathrm{Int}(D)\neq \emptyset$ and $x\in R_n-\partial^u_+R_n$ for every $n\in \mathbb{N}$. Thanks to the previous properties, we have that $\mathrm{Int}(R_n)\cap \mathrm{Int}(R_m)\neq \emptyset$ for every $n,m\in \mathbb{N}$. Moreover, using the fact that $I_{n+1}\subsetneq I_n$ and by repeating the argument used in the proof of Item~(2), we get that for every $n,m\in \mathbb{N}$ with $m>n$, exactly one of the following holds:

\begin{itemize}
    \item $R_n\cap R_m$ and $\mathcal{B}(R_m)\cap R_n$ are both non-trivial vertical subrectangles of $R_n$ 
    \item $R_n\cap R_m$ is a non-trivial vertical subrectangle of $R_n$, $\mathcal{F}^u(\partial^u_+R_n)= \mathcal{F}^u(\partial^u_+R_m)$ and $\mathcal{B}(R_m)= \mathcal{B}(R_n)$
    \item $R_n\cap R_m$ is a non-trivial vertical subrectangle of $R_n$, $\mathcal{F}^u(\partial^u_+R_n)= \mathcal{F}^u(\partial^u_+R_m)$ and $\mathcal{B}(R_m)\cap \mathcal{B}(R_n)$ is a non-trivial vertical subrectangle of $\mathcal{B}(R_n)$
\end{itemize}

According to which of the above scenarios occurs infinitely many times, we distinguish three cases.

\vspace{0.2cm}
\textbf{Case 1.} Passing to a subsequence if necessary, $R_{n+1}\cap R_n$ and $\mathcal{B}(R_{n+1})\cap R_n$ are non-trivial vertical subrectangles of $R_n$ for every $n\in \mathbb{N}$. 

In this case, it is easy to see that $I_{n+1}\subset R_n$ and thus that $[x,y]^s\subset R_n$ for every $n\in \mathbb{N}$. However, as $R_{n+1}\cap R_n$ is a non-trivial vertical subrectangle of $R_n$ for every $n\in \mathbb{N}$, by the expansivity axiom, $\underset{n\geq 0}{\cap}R_n$ is an $\mathcal{F}^u$-segment of $R_0$ containing $[x,y]^s$, which is absurd. 

\vspace{0.2cm}
\textbf{Case 2.} Passing to a subsequence if necessary, $R_{n+1}\cap R_n$ is a non-trivial vertical subrectangle of $R_n$, $\mathcal{F}^u(\partial^u_+R_{n+1})= \mathcal{F}^u(\partial^u_+R_n)$ and $\mathcal{B}(R_{n+1})= \mathcal{B}(R_n)$ for every $n\in \mathbb{N}$. 

In this case, by the expansivity axiom, $\underset{n\geq 0}{\cap}R_n$ is an $\mathcal{F}^u$-segment of $R_0$. Using the previous fact together with
$$[x,y]^s\subset \underset{n\geq 0}{\cap}(R_n\cup \mathcal{B}(R_n))
=\underset{n\geq 0}{\cap}R_n\cup \mathcal{B}(R_0),$$
we deduce that $[x,y]^s\subset \mathcal{B}(R_0)$. This contradicts the assumption that $x\in R_0-\partial^u_+R_0$.

\vspace{0.2cm}
\textbf{Case 3.} Passing to a subsequence if necessary, $R_{n+1}\cap R_n$ is a non-trivial vertical subrectangle of $R_n$, $\mathcal{F}^u(\partial^u_+R_{n+1})= \mathcal{F}^u(\partial^u_+R_n)$ and $\mathcal{B}(R_{n+1})\cap \mathcal{B}(R_n)$ is a non-trivial vertical subrectangle of $\mathcal{B}(R_n)$ for every $n\in \mathbb{N}$. 

Similarly to the two previous cases, by the expansivity axiom, $\underset{n\geq 0}{\cap}R_n$ is an $\mathcal{F}^u$-segment of $R_0$ and $\underset{n\geq 0}{\cap}\mathcal{B}(R_n)$ is an $\mathcal{F}^u$-segment of $\mathcal{B}(R_0)$. This implies that for $n$ sufficiently big $[x,y]^s\not \subset R_n\cup \mathcal{B}(R_n)$, which is absurd. 

\end{proof}
We conclude this section with the following result, which is a direct consequence of the proof of Proposition~\ref{p.markovianbirectangles}.

\begin{prop}
\label{p.comparebirectangles}
    Fix $(x,y)\in W^+$. Let $L:=\mathcal{F}^s(x)=\mathcal{F}^s(y)$, $D$ be the closure of a connected component of $\mathcal{P}-L$ and $\mathcal{I}(x,y)$ be the family of segments defined in Proposition \ref{p.markovianbirectangles}. 

    Consider $I,J\in \mathcal{I}(x,y)$ and $R,S\in \mathcal{R}$ such that $I=(R\cup \mathcal{B}(R))\cap L$ and $J=(S\cup \mathcal{B}(S))\cap L$. If $R\cap S$ is a non-trivial vertical subrectangle of $R$, then $J\subsetneq I$. 
\end{prop}

\subsection{Cocompactness of the action $\tilde\rho$}
\label{ss.cocompactness}

\begin{defi}
    Consider $I$ an $\mathcal{F}^{s}$-segment (resp. $\mathcal{F}^u$-segment) and $R\in \mathcal{R}$. We will say that $I$ \emph{crosses} $R$ if $R\cap I$ is an $\mathcal{F}^{s}$-leaf (resp. $\mathcal{F}^u$-leaf) of $R$. 
\end{defi}

\begin{prop}\label{p.actionboundedfund}
The action $\widetilde{\rho}$ of $G$ on $W^+$ admits a bounded fundamental domain. 
\end{prop}

\begin{proof}
 Consider a finite collection $R_1,\dots,R_n$ of rectangles consisting of a representative of each $\rho$-orbit of rectangles in the Markovian family $\mathcal{R}$. For every $i\in \llbracket 1, n\rrbracket$ we define 
\begin{align*}
\mathfrak{D}(R_i):=\{(x,&y) \in W^+ \mid x,y\in R_i\cup \mathcal{B}(R_i) \text{ and there exists $V$ a predecessor of $1$-st generation }\\ & \text{of $R_i$ such that } [x,y]^s \text{ crosses either } V \text{ or } \mathcal{B}(V)\}.
\end{align*}
  

 Let $\mathfrak{D}:= \underset{i\in \llbracket 1, n\rrbracket}{\cup} \mathfrak{D}(R_i)$. Using the fact that $R_i$ admits finitely many predecessors of first generation (see Lemma \ref{l.existenceofpredecessors}), it is easy to see that the set $\mathfrak{D}$ is compact in $W^+$. We claim that $\mathfrak{D}$ contains a fundamental domain for $\widetilde{\rho}$ or in other words that for every $(x,y)\in W^+$ there exists $g\in G$ such that $(\rho(g)(x),\rho(g)(y))\in \mathfrak{D}$. 

Indeed, fix $(x,y)\in W^+$, $D$ a connected component of $\mathcal{P}-\mathcal{F}^s(x)$ and consider 
\begin{align*}
       \mathcal{I}(x,y)&:=\{(R\cup \mathcal{B}(R))\cap \mathcal{F}^s(x)| ~R\in \mathcal{R}, ~x\in R-\partial^u_+R, ~ [x,y]^s\subset R\cup \mathcal{B}(R) \text{ and } R\cap \mathrm{Int}(D)\neq \emptyset\}
\end{align*}
By Proposition~\ref{p.markovianbirectangles}, $\mathcal{I}(x,y)$ is non-empty, is totally ordered by the inclusion relation and admits a unique minimal element, denoted by $I_{\min}(x,y)$.  

Fix $R\in \mathcal{R}$ such that $x\in R-\partial^u_+R$, $I_{\min}(x,y)= (R\cup \mathcal{B}(R))\cap \mathcal{F}^s(x)$, and $R\cap \mathrm{Int}(D)\neq \emptyset$. By definition of $R_1,...,R_n$, there exists $i\in \llbracket 1, n\rrbracket$ and $g\in G$ such that $\rho(g)(R)=R_i$. Moreover, by Remark \ref{r.invariantbigbrothers}, $[\rho(g)(x),\rho(g)(y)]^s\subset R_i\cup \mathcal{B}(R_i)$. We will show that $(\rho(g)(x),\rho(g)(y))\in \mathfrak{D}(R_i)$. 

Assume for the sake of simplicity that $g=e$ and thus that $R=R_i$. Denote by $V$ the unique predecessor of first generation of $R_i$ that contains $x$ and intersects $[x,y]^s$ non-trivially (see Lemma \ref{l.existenceofpredecessors}). If $[x,y]^s$ crosses $V$ or $\mathcal{B}(V)$, then, by the definition of $\mathfrak{D}(R_i)$, we have that $(x,y)\in \mathfrak{D}(R_i)$. If not, then by our choice of $V$, we have $[x,y]^s\subset V\cup \mathcal{B}(V)$. Using this fact, together with the facts that $x\in V-\partial^u_+V$ and $V\cap \mathrm{Int}(D)\neq \emptyset$, we get that $(V\cup \mathcal{B}(V))\cap \mathcal{F}^s(x)\in \mathcal{I}(x,y)$. However, by Proposition \ref{p.comparebirectangles} and the fact that $V$ is a predecessor of first generation of $R_i$, $$(V\cup \mathcal{B}(V))\cap \mathcal{F}^s(x)\subsetneq (R_i\cup \mathcal{B}(R_i))\cap \mathcal{F}^s(x)=I_{\min}(x,y)$$ which contradicts the minimality of $I_{\min}(x,y)$ and leads to an absurdity. 

\end{proof}
During the previous proof, we constructed a compact set $\mathfrak{D}$ containing a fundamental domain of $\widetilde{\rho}$. Although the precise definition of $\mathfrak{D}$ will not be important in the sequel of this paper, we will repeatedly use the following fact: 

\begin{obse}
\label{obs.fundamentaldomain}
    Fix $R_1,...,R_n$ a set of representatives of every $\rho$-orbit of rectangles inside $\mathcal{R}$. There exists a compact set $\mathfrak{D}=\mathfrak{D}(R_1,\ldots,R_n)\subset W^+$ such that, for every $(x,y)\in W^+$, if $D$ is a connected component of $\mathcal{P}-\mathcal{F}^s(x)$, $(R,\mathcal{B}(R))$ is the ``smallest'' birectangle containing $[x,y]^s$ and intersecting $D$ non-trivially (in the sense of Proposition~\ref{p.markovianbirectangles}), and $g\in G$ satisfies $\rho(g)(R)\in\{R_1,\ldots,R_n\}$, then 
$$  \widetilde{\rho}(g).(x,y)\in \mathfrak{D}$$ 
\end{obse}

\section{From a fully orientable strong Markovian action to an Anosov flow}\label{s.fullyorientflow}

In the previous sections, we constructed, thanks to the action $\rho$ and the foliation $\mathcal{F}^s$, a simply connected 3-manifold $W^+$ homeomorphic to $\mathbb{R}^3$ endowed with an orientation-preserving action $\widetilde{\rho}$ of $G$ by homeomorphisms. We also proved that the action $\widetilde{\rho}$ on $W^+$ is properly discontinuous and admits a bounded fundamental domain; hence, the quotient of $W^+$ by $\widetilde{\rho}$ defines a closed orientable three manifold, which we will denote by $M$. This implies in particular that $G\cong \pi_1(M)$ is a 3-manifold group.

Denote by $\pi:W^+\rightarrow \mathcal{P}$ the projection map defined by $$ \forall (x,y) \in W^+ ~~ \pi((x,y ))=x$$ Recall that $W^+$ is endowed with the subspace topology induced by the product topology on $\mathcal{P}^2$; hence, $\pi$ is continuous. Moreover, by repeating the argument in the proof of Proposition~\ref{p.wplusisasimplyconnectedmanifold}, $\pi$ defines a trivial line bundle over $\mathcal{P}$ whose fibers are precisely the sets $\pi^{-1}(x)$, where $x\in\mathcal{P}$. 

Consider $\mathcal{F}$ the orientable one dimensional foliation in $W^+$ defined by $$\mathcal{F}:=\{\pi^{-1}(x)|x\in \mathcal{P}\}=\big\{\{(x,y)|y\in \mathcal{F}^s_+(x)-\{x\}\}|x\in \mathcal{P}\big\}$$ Since $\rho$ is fully orientable, we have that $\widetilde{\rho}$ preserves $\mathcal{F}$; hence, $\mathcal{F}$ descends to an orientable one dimensional foliation on $M$. Thanks to Theorem 27A of \cite{Whitney}, the previous foliation can be parametrized by a non-singular flow $\Phi$. 

Denote by $\widetilde{\Phi}$ the lift of $\Phi$ on $W^+$. By reversing the orientation of the flow $\Phi$ if necessary, assume that for every $( x,y )\in W^+$ and every $t>0$, if $\widetilde{\Phi}^t((x,y) )= ( x, y(t))$, then $[x,y]^s\subset [x,y(t)]^s$. In other words, by identifying $W^+$ with the space of positively oriented $\mathcal{F}^s$-segments in $\mathcal{P}$, we assume that the flow $\widetilde{\Phi}$ makes the $\mathcal{F}^s$-segments bigger in the future and smaller in the past.

\begin{prop}
\label{p.pseudoanosov}
    The flow $\Phi$ is a topological Anosov flow.
\end{prop} 

\begin{proof}
    
We will verify that $\Phi$ satisfies the conditions of Definition~\ref{d.anosovflow}. First, notice that by definition, the flow $\Phi$ is not singular.

\vspace{0.3cm}
\textit{Existence of two transverse codimension-one foliations invariant under $\Phi$}

Using the fact that $W^+\overset{\pi}{\longrightarrow} \mathcal{P}$ defines a trivial line bundle over $\mathcal{P}$, the fibers of which are identified with the leaves of $\mathcal{F}$, we obtain that  $\widetilde{F^s}:=\pi^{-1}(\mathcal{F}^s)$ and $\widetilde{F^u}:=\pi^{-1}(\mathcal{F}^u)$ define two transverse codimension one foliations in $W^+$. Moreover, as the action of $\rho$ preserves $\mathcal{F}^s$ and $\mathcal{F}^u$, the action of $\widetilde{\rho}$ preserves both $\widetilde{F^s}$ and $\widetilde{F^u}$; hence $\widetilde{F^s}$ and $\widetilde{F^u}$ project to two transverse codimension one foliations in $M$. Denote the previous foliations by $F^s$ and $F^u$ respectively. Finally, by construction, the leaves of $\widetilde{F^s}$ and $\widetilde{F^u}$ are $\mathcal{F}$-saturated. This implies that $F^s$ and $F^u$ are invariant by $\Phi$ and thus that $\Phi$ satisfies Item (1) of Definition \ref{d.anosovflow}. 

\vspace{0.3cm}
\textit{Contraction along $F^s$ and $F^u$}

Using our choice of orientation of $\Phi$, we will now prove that $F^s$ and $F^u$ define respectively a stable and an unstable foliation for $\Phi$. 

Endow $M$ with a smooth structure and a distance $d$ given by some Riemannian metric. Consider $\widetilde{d}$ the lift of $d$ on $W^+$. The action $\widetilde{\rho}$ corresponds to the action of $\pi_1(M)$ on $W^+$ by deck transformations; hence, $\widetilde{\rho}$ preserves $\widetilde{d}$. 

Fix $z_1,z_2\in M$ such that $z_2\in F^s(z_1)$. We would like to find $h\in \text{Homeo}_+(\mathbb{R})$ an increasing homeomorphism such that 
    $d(\Phi^t(z_1),\Phi^{h(t)}(z_2))\underset{t\rightarrow +\infty}{\longrightarrow}0$. Consider $\widetilde{z_1}$ (resp. $\widetilde{z_2}$) a lift of $z_1$ (resp. $z_2$) on $W^+$ such that $\widetilde{z_2}\in \widetilde{F^s}(z_1)$. It suffices to find $h\in \text{Homeo}_+(\mathbb{R})$ an increasing homeomorphism such that 
   \begin{equation}\label{eq.goal}
        \widetilde{d}(\widetilde{\Phi}^{t}(\widetilde{z_1}),\widetilde{\Phi}^{h(t)}(\widetilde{z_2}))\underset{t\rightarrow +\infty}{\longrightarrow}0
    \end{equation} 

Consider first the case where $\widetilde{z_1}$ and $\widetilde{z_2}$ belong to the same $\widetilde{\Phi}$-orbit. In this case, the existence of a homeomorphism $h$ satisfying (\ref{eq.goal}) is obvious. Assume from now on that $\widetilde{z_1}$ and $\widetilde{z_2}$ belong to different $\widetilde{\Phi}$-orbits. By the definition of $\widetilde{F^s}$, up to exchanging the roles of $\widetilde{z_1}$ and $\widetilde{z_2}$, we can write $$\widetilde{z_1}=(x_1,y_1)\in W^+ \text{ and } \widetilde{z_2}=(x_2,y_2)\in W^+$$ where $x_1,x_2,y_1,y_2\in \mathcal{P}$ and $x_2\in\mathcal{F}^s_+(x_1)-\{x_1\}$. 

Fix $D$ the closure of a connected component of $\mathcal{P}-\mathcal{F}^s(x_1)=\mathcal{P}-\mathcal{F}^s(x_2)$. For every $t\in \mathbb{R}$, let $y_1(t)$ and $y_2(t)$ denote the unique points of $\mathcal{F}^s(x_1)=\mathcal{F}^s(x_2)$ for which $$\widetilde{\Phi}^{t}(( x_1,y_1))= ( x_1, y_1(t))$$
$$\widetilde{\Phi}^{t}(( x_2,y_2))= ( x_2, y_2(t))$$

Consider $h:\mathbb{R}\to\mathbb{R}$ an increasing homeomorphism such that $y_1(t)=y_2(h(t))$ for all sufficiently large $t>0$. In what follows, we show that $$\lim_{t\rightarrow +\infty}\widetilde{d}\big(\widetilde{\Phi}^{t}(( x_1,y_1)),\widetilde{\Phi}^{h(t)}(( x_2,y_2))\big)=\lim_{t\rightarrow +\infty}\widetilde{d}\big(( x_1,y_1(t)),( x_2,y_1(t))\big)=0$$

Let $R_1,...,R_N$ be a set of representatives of every $\rho$-orbit of rectangles in $\mathcal{R}$ and $\mathfrak{D}=\mathfrak{D}(R_1,...,R_N)$ be the compact set in $W^+$ given by Observation \ref{obs.fundamentaldomain}. For every $t\in \mathbb{R}$ we define the family of segments
\vspace{0.2cm}
$$\mathcal{I}(x_1, y_1(t)):=\{(R\cup \mathcal{B}(R))\cap \mathcal{F}^s(x_1)|~ R\in \mathcal{R}, ~x_1\in R-\partial^u_+R, ~[x_1,y_1(t)]^s\subset R\cup \mathcal{B}(R) \text{ and } R\cap \mathrm{Int}(D)\neq \emptyset\}$$

By Proposition~\ref{p.markovianbirectangles}, for every $t\in \mathbb{R}$, the set $\mathcal{I}(x_1,y_1(t))$ is non-empty, is totally ordered by the inclusion relation and admits a unique minimal element, denoted by $I_{\min}(x_1,y_1(t))$. Consider $R(t)\in \mathcal{R}$ such that $x_1\in R(t)-\partial^u_+R(t)$, $R(t)\cap \mathrm{Int}(D)\neq \emptyset$ and $$I_{\min}(x_1,y_1(t))=\big(R(t)\cup \mathcal{B}(R(t))\big)\cap \mathcal{F}^s(x_1)$$ Consider also $g(t)\in G$ such that $\rho(g(t))(R(t))\in \{R_1,...,R_N\}$.  Thanks to Observation \ref{obs.fundamentaldomain}, $\widetilde{\rho}(g(t)).(x_1,y_1(t))\in \mathfrak{D}$ for every $t\in \mathbb{R}$. Moreover, since the action $\widetilde{\rho}$ preserves $\widetilde{d}$, for every sufficiently large $t>0$ we have that 
\begin{equation*}
   \begin{aligned}
    \widetilde{d}\big(\widetilde{\Phi}^{t}(( x_1,y_1)),\widetilde{\Phi}^{h(t)}(( x_2,y_2))\big)&=\widetilde{d}\big(( x_1,y_1(t)), ( x_2,y_1(t))\big)\\&=\widetilde{d}\big(\widetilde{\rho}(g(t)).( x_1,y_1(t)), ~\widetilde{\rho}(g(t)).( x_2,y_1(t))\big)
\end{aligned} 
\end{equation*}

Suppose now, by contradiction, that there exists $\epsilon>0$ and a sequence  $(t_n)_{n\in\mathbb{N}}$ with $t_n\underset{n\rightarrow +\infty}{\longrightarrow}+\infty$ such that \begin{equation}\label{eq.absurdhypo}
\widetilde{d}\big(\widetilde{\rho}(g(t_n)).( x_1,y_1(t_n)), ~\widetilde{\rho}(g(t_n)).( x_2,y_1(t_n))\big)>\epsilon
\end{equation}

Recall that, thanks to our choice of orientation of $\widetilde{\Phi}$, we have that $[x_1,y_1(t)]^s\subsetneq [x_1,y_1(t')]^s$ for every $t'> t$. Thanks to the previous fact and the fact that $x_2\in\mathcal{F}^s_+(x_1)$, after possibly passing to a subsequence, we may assume that $x_2\in[x_1,y_1(t_n)]^s$ for every $n\in\mathbb{N}$.

Next, since $\mathfrak{D}$ is compact and $\widetilde{\rho}(g(t_n)).(x_1,y_1(t_n))\in \mathfrak{D}$ for every $n\in\mathbb{N}$, by passing to a further subsequence if necessary, we may assume that there exists $(X,Y)\in\mathfrak{D}$ such that
$$\widetilde{\rho}(g(t_n)).(x_1,y_1(t_n))\underset{n\to\infty}{\longrightarrow} (X,Y)$$ The previous fact implies that the $\mathcal{F}^s$-segments $[\rho(g(t_n))(x_1),\rho(g(t_n))(y_1(t_n))]^s$ converge for the Hausdorff topology to $[X,Y]^s$. 

Finally, using the fact that $x_2\in [x_1,y_1(t_n)]^s$ for every $n\in\mathbb{N}$ and, after possibly extracting another subsequence, we obtain $Z\in [X,Y]^s$ such that
$$\rho(g(t_n))(x_2)\underset{n\to\infty}{\longrightarrow} Z$$
Notice that $Z\neq X$, thanks to the Inequality \eqref{eq.absurdhypo}.

Recall now that by definition, $x_1\in R(t)-\partial^u_+R(t)$ and $R(t)\cap \mathrm{Int}(D)\neq \emptyset$ for every $t\in \mathbb{R}$. It follows that $\mathrm{Int}(R(t'))\cap \mathrm{Int}(R(t))\neq \emptyset$ for every $t'>t$ and thus, by Lemma \ref{l.npredecessor}, $R(t')$ is either a predecessor or a successor of some generation of $R(t)$. If $t'$ is sufficiently bigger than $t$, then thanks to our choice of orientation of $\widetilde{\Phi}$ and to our definition of $R(t)$, $$\big(R(t)\cup \mathcal{B}(R(t))\big)\cap \mathcal{F}^s(x_1)\subsetneq [x_1,y_1(t')]^s \subseteq \big(R(t')\cup \mathcal{B}(R(t'))\big)\cap \mathcal{F}^s(x_1)$$ By Proposition \ref{p.comparebirectangles} (applied for $\mathcal{I}(x_1, y_1(t))$), we deduce that for any $t'$  sufficiently bigger than $t$, $R(t')$ is a successor of some positive  generation of $R(t)$. Thanks to the previous fact and the inclusion $\rho(g(t_n))(R(t_n))\in \{R_1,\ldots,R_N\}$ for every $n\in\mathbb{N}$, after passing once again to a subsequence, we may assume that there exists $i\in\llbracket 1,N\rrbracket$ such that $\rho(g(t_n))(R(t_n))=R_i$ and that $R(t_{n+1})$ is a successor of some positive generation of $R(t_n)$ for every $n\in\mathbb{N}$. By Item (1) of Theorem \ref{t.markovianisanosovlike}, we deduce that $\rho(g(t_m))\neq \rho(g(t_n))$ for every $n\neq m$. 

Our previous arguments imply that $\widetilde{\rho}(g(t_n)).(x_1,x_2) \underset{n\rightarrow +\infty}{\longrightarrow} (X,Z)\in W^+$ and that the sequence $(g(t_n))_{n\in\mathbb{N}}$ consists of pairwise distinct elements of $G$. This contradicts the proper discontinuity of $\widetilde{\rho}$ and leads to an absurdity.

\vspace{0.3cm}
\textit{Expansiveness along $F^s$ and $F^u$}

Thanks to our previous arguments, $\Phi$ is non-singular and satisfies Items (1) and (2) of Definition \ref{d.anosovflow}. To prove that $\Phi$ is a topological Anosov flow, it suffices to prove that $\Phi$ satisfies Item (3) of Definition \ref{d.anosovflow}. In order to do that, we will show that there exist $ \epsilon>0$ and $T>0$ such that for any $\widetilde{z}\in W^+$, $\widetilde{w}\in \widetilde{F^s}(\widetilde{z})$ and any increasing  homeomorphism $f:\mathbb{R}\rightarrow \mathbb{R}$ with $f(0)=0$ 
\begin{equation}\label{eq.statementstable}
    \forall t\in \mathbb{R}~~ \widetilde{d}(\widetilde{\Phi}^t(\widetilde{z}),\widetilde{\Phi}^{f(t)}(\widetilde{w}))<\epsilon \implies \exists |t_0|<T ~~ \widetilde{w}=\widetilde{\Phi}^{t_0}(\widetilde{z})
\end{equation} 

The analogue of the above statement for $\widetilde{F^u}$ follows from a similar argument. We therefore restrict our attention to proving \eqref{eq.statementstable}. The ensuing proof follows the same strategy as the one used to show that $(\Phi^t)_{t>0}$ acts as a contraction in the direction of $F^s$.

Fix $\widetilde{z}\in W^+$ and $\widetilde{w}\in \widetilde{F^s}(z)$. Consider first the case where $\widetilde{z}$ and $\widetilde{w}$ belong to the same $\widetilde{\Phi}$-orbit. In this case, since $\widetilde{\Phi}$ is non-singular and its orbits coincide with the fibers of the trivial line bundle $W^+\overset{\pi}{\longrightarrow}\mathcal{P}$, and since the action of $\widetilde{\rho}$ on $W^+$ admits a bounded fundamental domain and preserves both $\widetilde{d}$ and $\widetilde{\Phi}$, it is not difficult to see that there exist constants $\epsilon>0$ and $T>0$ such that
$$ \forall t\in \mathbb{R}~~ \text{ if }|t|>T, \text{ then } ~~ \forall \widetilde{x}\in W^+~~ \widetilde{d}(\widetilde{x}, \widetilde{\Phi}^t(\widetilde{x}))>\epsilon$$ The proof of (\ref{eq.statementstable}) when $\widetilde{z}$ and $\widetilde{w}$ belong to the same $\widetilde{\Phi}$-orbit is an immediate consequence of the above fact. It remains to show that

\textbf{Claim.} There exists $\epsilon'>0$ such that, for every $f\in\operatorname{Homeo}_+(\mathbb{R})$ and every pair of points $\widetilde{z}\in W^+$ and $\widetilde{w}\in\widetilde{F^s}(\widetilde{z})$ belonging to distinct $\widetilde{\Phi}$-orbits, we have $$\max_{t\in \mathbb{R}}\big(\widetilde{d}(\widetilde{\Phi}^{t}(\widetilde{z}),\widetilde{\Phi}^{f(t)}(\widetilde{w}))\big)>\epsilon'$$

Assume that the previous claim does not hold. Then, there exist $(f_n)_{n\in \mathbb{N}}$ a sequence of increasing homeomorphisms of $\mathbb{R}$, $(\epsilon'_n)_{n\in \mathbb{N}}$ a sequence of strictly positive real numbers with $\epsilon'_n\underset{n\rightarrow +\infty}{\longrightarrow}0$, and $(\widetilde{z_n})_{n\in \mathbb{N}}, (\widetilde{w_n})_{n\in \mathbb{N}}$ two sequences of points in $W^+$ such that $\widetilde{w_n}\in \widetilde{F^s}(\widetilde{z_n})$, $\widetilde{z_n}\notin \underset{t\in \mathbb{R}}{\bigcup}\widetilde{\Phi}^t(\widetilde{w_n})$ and $$\forall n\in \mathbb{N} ~ \forall t\in \mathbb{R}~~ \widetilde{d}\big(\widetilde{\Phi}^t(\widetilde{z_n}), \widetilde{\Phi}^{f_n(t)}(\widetilde{w_n}))\big)<\epsilon'_n$$ 

As in our previous arguments, up to exchanging the roles of $\widetilde{z_n}$ and $\widetilde{w_n}$ and passing to a subsequence, we can write $$\widetilde{z_n}=(x^1_n,y^1_n)\in W^+ \text{ and } \widetilde{w_n}=(x^2_n,y^2_n)\in W^+$$ where $ x_n^1,y_n^1,  x_n^2,y_n^2\in \mathcal{P}$ and $x_n^2\in \mathcal{F}_+^s(x_n^1)-\{x_n^1\}$.

Fix $D_n$ the closure of a connected component of $\mathcal{P}-\mathcal{F}^s(x^1_n)=\mathcal{P}-\mathcal{F}^s(x^2_n)$. For every $t\in \mathbb{R}$, let $y_n^1(t)$ and $y^2_n(t)$ denote the unique points of $\mathcal{F}^s(x^1_n)=\mathcal{F}^s(x^2_n)$ for which $$\widetilde{\Phi}^{t}(( x^1_n,y^1_n))= ( x^1_n, y^1_n(t))$$
$$\widetilde{\Phi}^{t}(( x^2_n,y^2_n))= ( x^2_n, y^2_n(t))$$

Let $R_1,...,R_N$ be a set of representatives of every $\rho$-orbit of rectangles in $\mathcal{R}$ and $\mathfrak{D}=\mathfrak{D}(R_1,...,R_N)$ be   the compact set in $W^+$ given by Observation \ref{obs.fundamentaldomain}. For every $t\in \mathbb{R}$ and every $n\in \mathbb{N}$ we define the family of segments
\vspace{0.2cm}
$$\mathcal{I}(x^1_n, y^1_n(t)):=\{(R\cup \mathcal{B}(R))\cap \mathcal{F}^s(x^1_n)|~ R\in \mathcal{R}, ~x^1_n\in R-\partial^u_+R, ~[x^1_n,y^1_n(t)]^s\subset R\cup \mathcal{B}(R) \text{ and } R\cap \mathrm{Int}(D_n)\neq \emptyset\}$$

By Proposition~\ref{p.markovianbirectangles}, for every $t\in \mathbb{R}$ and $n\in \mathbb{N}$, the set $\mathcal{I}(x^1_n,y^1_n(t))$ is non-empty, is totally ordered by the inclusion relation and admits a unique minimal element, denoted by $I_{\min}(x^1_n,y^1_n(t))$. Consider $R_n(t)\in \mathcal{R}$ such that $x_n^1\in R_n(t)-\partial_+^uR_n(t)$, $ R_n(t)\cap \mathrm{Int}(D_n)\neq \emptyset$ and $$I_{\min}(x^1_n,y^1_n(t))=\big(R_n(t)\cup \mathcal{B}(R_n(t))\big)\cap \mathcal{F}^s(x^1_n)$$ Consider also $g_n(t)\in G$ such that $\rho(g_n(t))(R_n(t))\in \{R_1,...,R_N\}$. Thanks to Observation \ref{obs.fundamentaldomain}, we have that $\widetilde{\rho}(g_n(t)).(x^1_n,y^1_n(t))\in \mathfrak{D}$. Moreover, since the action $\widetilde{\rho}$ preserves $\widetilde{d}$, 
\begin{equation*}
   \begin{aligned}
    \widetilde{d}\big(\widetilde{\Phi}^{t}(( x^1_n,y^1_n)),\widetilde{\Phi}^{f(t)}(( x^2_n,y^2_n))\big)&=\widetilde{d}\big(( x^1_n,y^1_n(t)), ( x^2_n,y^2_n(f_n(t)))\big)\\&=\widetilde{d}\big(\widetilde{\rho}(g_n(t)).( x^1_n,y^1_n(t)), ~\widetilde{\rho}(g_n(t)).( x^2_n,y^2_n(f_n(t)))\big)<\epsilon'_n
\end{aligned} 
\end{equation*}

Recall now that, thanks to our choice of orientation of $\widetilde{\Phi}$, we have that $[x^1_n,y^1_n(t)]^s\subsetneq [x^1_n,y^1_n(t')]^s$ for every $t'>t$. Using the previous fact together with the assumption that $x^2_n\in \mathcal{F}^s_+(x^1_n)-\{x^1_n\}$, we obtain $t_n\in \mathbb{R}$ such that $x^2_n\notin [x^1_n,y^1_n(t_n)]^s$ for every $n\in\mathbb{N}$. 

Next, since $\mathfrak{D}$ is compact and $\widetilde{\rho}(g(t_n)).(x^1_n,y^1_n(t_n))$ (resp. $\widetilde{\rho}(g(t_n)).(x^2_n,y^2_n(f_n(t_n)))$) belongs in $ \mathfrak{D}$ (resp. in the $\epsilon'_n$-neighborhood of $\mathfrak{D}$) for every $n\in\mathbb{N}$, by passing to a further subsequence if necessary, we may assume that there exists $(X,Y)\in \mathfrak{D}$ such that
$$\widetilde{\rho}(g(t_n)).(x^1_n,y^1_n(t_n))\underset{n\to\infty}{\longrightarrow} (X,Y)$$
$$\widetilde{\rho}(g(t_n)).(x^2_n,y^2_n(f_n(t_n)))\underset{n\to\infty}{\longrightarrow} (X,Y)$$

In particular, the sequences $(\rho(g(t_n))(x^1_n))_{n\in\mathbb{N}}$ and  $(\rho(g(t_n))(x^2_n))_{n\in\mathbb{N}}$ converge to the same point $X\in \mathcal{P}$ when $n\rightarrow +\infty$. However, since $x_n^2\in \mathcal{F}^s_+(x_n^1)$ and $x^2_n\notin [x^1_n,y^1_n(t_n)]^s$ for every $n\in\mathbb{N}$, the sequence $(\rho(g(t_n))(x^2_n))_{n\in\mathbb{N}}$ can not converge to a point of $[X,Y]^s$, except maybe $Y$. This is absurd. 

\end{proof}

Summarizing the results proven in this section, we have that 

\begin{itemize}
    \item $\Phi$ is a topological Anosov flow. 
    \item The orbits of the lift of $\Phi$ on the universal cover of $M$, namely $W^+$, define a trivial line bundle over $\mathcal{P}$. Therefore, $\mathcal{P}$ coincides with the orbit space of $\Phi$.
    \item The lifts of $F^s$ and $F^u$, the stable and unstable foliations of $\Phi$, define on $W^+$ two foliations by planes, whose projections on $\mathcal{P}$ coincide with $\mathcal{F}^s$ and $\mathcal{F}^u$ respectively. We deduce that $(\mathcal{P},\mathcal{F}^s,\mathcal{F}^u)$ coincides with the bifoliated plane of $\Phi$.
    \item The action of $\pi_1(M)$ on $W^+$ by deck transformations coincides with $\widetilde{\rho}$. Therefore, thanks to our construction of $\widetilde{\rho}$, the action of $\pi_1(M)$ on the bifoliated plane of $\Phi$ coincides with $\rho$.
\end{itemize}

As a consequence of the previous results, to complete the proof of Theorem~\ref{t.main} in the case of fully orientable strong Markovian actions, it suffices to show that $\mathcal{R}$ is the projection on $\mathcal{P}$ of the lift on $W^+$ of a reduced Markov partition of $\Phi$. The next section is devoted to the proof of this fact.

\section{From a fully orientable strong Markovian action to a Markov partition}\label{s.fullyorientmarkovpartition}

The main goal of this section consists in proving the following proposition, which finishes the proof of Theorem \ref{t.main} in the case of fully orientable strong Markovian actions. Recall that $\pi:W^+\rightarrow \mathcal{P}$ denotes the projection map defined by $\pi((x,y))=x$ for every $(x,y)\in W^+$. 

\begin{prop}\label{p.fullyorientablemarkov}
    There exists $\widehat{\mathcal{R}}$ a $\widetilde{\rho}$-invariant collection of rectangles in $W^+$ such that 
    
    \begin{itemize}
        \item $\widehat{\mathcal{R}}$ is a lift of $\mathcal{R}$ on $W^+$ for which each rectangle of $\mathcal{R}$ admits a unique lift in $\widehat{\mathcal{R}}$.
        \item The projection of $\widehat{\mathcal{R}}$ on $M$ is a reduced Markov partition of the topological Anosov flow $\Phi$.
    \end{itemize}
\end{prop}

The proof of the above proposition will be split in two parts: 

\begin{enumerate}[label=\Alph*)]
    \item First, we construct a family of rectangles in $W^+$ that differs from the lift to $W^+$ of a reduced Markov partition of $\Phi$ only in that the interiors of the rectangles need not be disjoint. More specifically, we will show that:

\begin{prop}\label{p.fullyorientablemarkovstepone}
    For every $R\in \mathcal{R}$ define $$\widetilde{R}=\{( x, y)\in W^+|x\in R, y\in \partial^u_+\mathcal{B}(R)\}$$ and let $\widetilde{\mathcal{R}}=\{\widetilde{R}|R\in\mathcal{R}\}$. We have that $\widetilde{\mathcal{R}}$ a $\widetilde{\rho}$-invariant collection of rectangles in $W^+$ such that:
    
    \begin{enumerate}[label=(\arabic*)]
        \item $\widetilde{\mathcal{R}}$ is a lift of $\mathcal{R}$ on $W^+$ for which each rectangle of $\mathcal{R}$ admits a unique lift in $\widetilde{\mathcal{R}}$.
        \item Let $\widetilde{x}\in W^+$ and $x=\pi(\widetilde{x})\in\mathcal{P}$. For every $\mathcal{F}^s$-separatrix $\mathcal{F}^s_{0}(x)$ of $x$ and every $\mathcal{F}^u$-separatrix $\mathcal{F}^u_{0}(x)$ of $x$, there exists a rectangle in $\widetilde{\mathcal{R}}$ that intersects the positive orbit of $\widetilde{x}$ and has non-trivial intersection with both $\pi^{-1}(\mathcal{F}^s_{0}(x))$ and $\pi^{-1}(\mathcal{F}^u_{0}(x))$.
        
        \item The projection of $\widetilde{\mathcal{R}}$ on $M$ is a collection of rectangles satisfying all the axioms in the definition of a reduced Markov partition of $\Phi$ (see Definition~\ref{d.markovpartition}), except that the interiors of the rectangles need not be disjoint.

        \item For any $R,S\in \mathcal{R}$ with $\mathrm{Int}(R)\cap \mathrm{Int}(S)\neq \emptyset$, if $\widetilde{R}\cap \widetilde{S}\neq \emptyset$, then $\partial^u_+R\cap \partial^u_+S\neq \emptyset$, $\mathcal{B}(R)=\mathcal{B}(S)$ and $\pi(\widetilde{R}\cap\widetilde{S})=R\cap S$.
    \end{enumerate} 
\end{prop}
   \item  Once the proof of Proposition~\ref{p.fullyorientablemarkovstepone} is complete, we will show that it implies Proposition~\ref{p.fullyorientablemarkov}.
\end{enumerate}
\begin{proof}[Proof of Item (1) of Proposition \ref{p.fullyorientablemarkovstepone}]  

Consider $R\in \mathcal{R}$. Using the definition of the big brother of $R$ together with Item (4) of Proposition~\ref{p.propertiesoffoliformarkovianactions}, it is not difficult to see that, for every $x\in R$, there exists a unique point $y\in\mathcal{P}$ such that $(x,y)\in\widetilde{R}$. Consequently, $\widetilde{R}$ is a rectangle in $(W^+,\widetilde{F}^s,\widetilde{F}^u)$ that is topologically transverse to $\widetilde{\Phi}$. Next, by the definition of $\widetilde{R}$ and $\widetilde{\mathcal{R}}$, we have that $\widetilde{R}$ is the unique rectangle of $\widetilde{\mathcal{R}}$ projecting onto $R$ under $\pi$. Finally, by Remark~\ref{r.invariantbigbrothers}, 
$\widetilde{\rho(g)(R)}=\widetilde{\rho}(g)(\widetilde{R})$
for every $g\in G$; hence, $\widetilde{\mathcal{R}}$ is $\widetilde{\rho}$-invariant.
\end{proof}

\begin{proof}[Proof of Item (2) of Proposition \ref{p.fullyorientablemarkovstepone}] 
Consider $\widetilde{x}=(x,y)\in W^+$, $\mathcal{F}^s_{0}(x)$ an $\mathcal{F}^s$-separatrix of $x$, $\mathcal{F}^u_{0}(x)$ an $\mathcal{F}^u$-separatrix of $x$ and $D$ the closure of the unique  connected component of $\mathcal{P}-\mathcal{F}^s(x)$ containing $\mathcal{F}^u_{0}(x)$. 

Take $x',y'$ so that $[x,y]^s$ is contained in the interior of $[x',y']^s$. By Proposition \ref{p.markovianbirectangles}, there exists $r\in \mathcal{R}$ such that $[x',y']^s\subset r\cup \mathcal{B}(r)$ and $r\cap \mathrm{Int}(D)\neq \emptyset$. By definition of $D$ and since $x\in [x',y']^s-\{x',y'\}$, we deduce that either $r$ or $\mathcal{B}(r)$ contains $x$ and intersects both $\mathcal{F}^s_{0}(x)$ and $\mathcal{F}^u_{0}(x)$ non-trivially.

Assume without any loss of generality that $r$ satisfies the previous properties. This implies that $\widetilde{r}$, the unique lift of $r$ on $\widetilde{\mathcal{R}}$, intersects non-trivially both $\pi^{-1}(\mathcal{F}^s_{0}(x))$ and $\pi^{-1}(\mathcal{F}^u_{0}(x))$. Recall now that by identifying $W^+$ with the space of positively oriented $\mathcal{F}^s$-segments in $\mathcal{P}$, we assumed that the flow $\widetilde{\Phi}$ makes the $\mathcal{F}^s$-segments bigger in the future. Thanks to this fact and to the fact that $[x,y]^s\subset r\cup \mathcal{B}(r)$, we get that  the positive orbit of $\widetilde{x}$ by $\widetilde{\Phi}$ intersects $\widetilde{r}$, which proves the desired result. 
\end{proof}

\begin{proof}[Proof of Item (3) of Proposition \ref{p.fullyorientablemarkovstepone}] \nopagebreak
    
    Let $\mathcal{R}_M$ be the projection of $\widetilde{\mathcal{R}}$ on $M$. Thanks to Item (1) of this proposition and to the fact that $\mathcal{R}$ consists of finitely many $\rho$-orbits of rectangles, we get that $\widetilde{\mathcal{R}}$ consists of finitely many  $\widetilde{\rho}$-orbits of rectangles; hence, $\mathcal{R}_M$ is a finite collection of rectangles in $M$ that are topologically transverse to $\Phi$. We will now show that $\mathcal{R}_M$ satisfies Items (2)-(4) of Definition \ref{d.markovpartition}. 

    First, assume by contradiction that $\mathcal{R}_M$ does not satisfy Item (4) of Definition \ref{d.markovpartition}. In that case, there exist two distinct rectangles $\widetilde{R_i},\widetilde{R_j}\in \widetilde{\mathcal{R}}$ and a continuous map $\tau:\widetilde{R_i}\rightarrow \mathbb{R}$ such that $\widetilde{\Phi}^\tau(\widetilde{R_i})\subseteq \widetilde{R_j}$. Let $R_i=\pi(\widetilde{R_i})\in \mathcal{R}$ and $R_j=\pi(\widetilde{R_j})\in \mathcal{R}$. Since every rectangle in $\mathcal{R}$ lifts to a unique rectangle in $\widetilde{\mathcal{R}}$,  we have that $R_i\neq R_j$. Recall now that $\pi$ defines a trivial line bundle over $\mathcal{P}$, whose fibers coincide with the orbits of $\widetilde{\Phi}$. Thanks to the previous fact, we get that $R_i\subseteq R_j$, which is impossible for any two distinct rectangles in $\mathcal{R}$, thanks to the Markovian intersection axiom.

      Next, assume by contradiction that $\mathcal{R}_M$ does not satisfy Item (3) of Definition \ref{d.markovpartition}. In that case, there exist two distinct rectangles $\widetilde{R},\widetilde{S}\in\widetilde{\mathcal{R}}$ for which one of the following holds:
      
      \begin{itemize}
          \item There exist $\widetilde{x}\in\partial^s \widetilde{R}$ and $t_0>0$ such that $\widetilde{\Phi}^{t_0}(\widetilde{x})$ belongs to the interior of $\widetilde{S}$
          \item There exist $\widetilde{x}\in\partial^u \widetilde{R}$ and $t_0<0$ such that $\widetilde{\Phi}^{t_0}(\widetilde{x})$ belongs to the interior of $\widetilde{S}$
      \end{itemize}
      
      Assume without any loss of generality that we are in the situation of the first of the above two cases (the second case can be treated by a similar argument). 

    \begin{figure}[h!]
    \centering
    \includegraphics[scale=0.65]{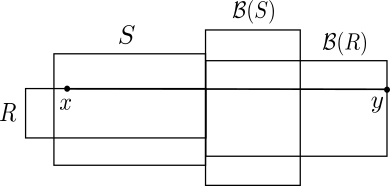}
    \caption{If the positive $\widetilde{\Phi}$-orbit of $(x,y)$ intersects the interior of $\widetilde{S}$, then $(x,y)$ can not lie in $\partial^s\widetilde{R}$.}
    \label{f.proofprop52}
    \end{figure}

      Consider $R=\pi(\widetilde{R})$, $S=\pi(\widetilde{S})$, $x=\pi(\widetilde{x})$, and let $y\in\mathcal{P}$ denote the unique point of intersection of $\mathcal{F}^s_+(x)$ with $\partial^u_+\mathcal{B}(R)$ (see Figure \ref{f.proofprop52}). Notice that by the definition of $\widetilde{R}$, we have $\widetilde{x}=(x,y)$. Notice also that, thanks to the fact that  $\pi(\bigcup_{t\in\mathbb{R}}\widetilde{\Phi}^t(\widetilde{x}))=x$, we have $x\in\partial^sR\cap\operatorname{Int}(S)$. On the one hand, since $\partial^sR\cap\operatorname{Int}(S)\neq\emptyset$, the Markovian intersection axiom implies that $R\cap S$ is a vertical subrectangle of $R$. Hence, by Proposition~\ref{p.comparebirectangles},
\begin{equation}\label{eq.comparisonrs}
    (S\cup\mathcal{B}(S))\cap\mathcal{F}^s(x)\subseteq(R\cup\mathcal{B}(R))\cap\mathcal{F}^s(x).
\end{equation}
On the other hand, since the strictly positive orbit of $\widetilde{x}$ intersects the interior of $\widetilde{S}$, the definition of $\widetilde{\mathcal{R}}$ together with our choice of orientation of $\widetilde{\Phi}$ implies that $[x,y]^s\subset S\cup\mathcal{B}(S)$ and that $y\notin \partial^u_+\mathcal{B}(S)$. The previous facts contradict  (\ref{eq.comparisonrs}), given that $ y\in\partial^u_+\mathcal{B}(R)$.

   Finally, assume by contradiction that $\mathcal{R}_M$ does not satisfy Item (2) of Definition \ref{d.markovpartition}. In this case, there exist $(\widetilde{x_n})_{n\in \mathbb{N}}$ a bounded sequence of points in $W^+$ and $(t_n)_{n\in \mathbb{N}}$ a sequence of positive real numbers going to infinity such that the orbit segment $\underset{t\in[0,t_n]}{\cup}\widetilde{\Phi}^t(\widetilde{x_n})$ does not intersect a rectangle in $\widetilde{\mathcal{R}}$ for every $n\in \mathbb{N}$. By  extracting a subsequence if necessary, assume without any loss of generality that $\widetilde{x_n}\underset{n\rightarrow \infty}{\longrightarrow}\widetilde{X} \in W^+$. 

    Let $x_n= \pi(\widetilde{x_n})$ and $X=\pi(\widetilde{X})$. By the continuity of $\pi$, we have that $(x_n)_{n\in \mathbb{N}}$ converges to $X$. Consider $V$ the closure of a connected component of $\mathcal{P}-(\mathcal{F}^s(X)\cup\mathcal{F}^u(X))$ and $U$ a small neighborhood of $X$ inside $V$ such that infinitely many $x_n$ are contained in $U$. By possibly extracting another subsequence, assume that all the $x_n$ belong in $U$. Denote by $\mathcal{F}_{0}^s(X)$ and $\mathcal{F}_{0}^u(X)$ the unique $\mathcal{F}^s$- and $\mathcal{F}^u$-separatrices of $X$ bounding $V$. Thanks to Item (3) of this proposition, there exists $\widetilde{R}\in  \widetilde{\mathcal{R}}$ intersecting the strictly positive orbit of $\widetilde{X}$ and having a non-trivial intersection with both $\pi^{-1}(\mathcal{F}_{0}^s(X))$ and $\pi^{-1}(\mathcal{F}_{0}^u(X))$. Take $T>0$ such that $\widetilde{\Phi}^{T}(\widetilde{X})\in \widetilde{R}$. For $n$ sufficiently big there exists $T_{n}>0$ close to $T$ such that $\widetilde{\Phi}^{T_n}(\widetilde{x_n})\in \widetilde{R}$. This contradicts the fact that the future orbits 
    of $\widetilde{x_n}$ take progressively more and more time before intersecting a rectangle in $\widetilde{\mathcal{R}}$.
 
\end{proof}

\begin{proof}[Proof of Item (4) of Proposition \ref{p.fullyorientablemarkovstepone}]       
    
    Consider $R,S\in\mathcal{R}$ such that $\mathrm{Int}(R)\cap \mathrm{Int}(S)\neq \emptyset$ and $\widetilde{R}\cap \widetilde{S}\neq\emptyset$. By Lemma~\ref{l.npredecessor}, we can assume without any loss of generality  that $R$ is a predecessor of $S$ of some generation. Let us show that $\partial^u_+S\subset \partial^u_+R$ and $\mathcal{B}(R)=\mathcal{B}(S)$. 

\begin{figure}[h!]
\centering

\begin{minipage}[t]{0.48\textwidth}
    \centering
    \includegraphics[width=\linewidth]{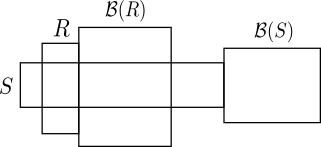}

\end{minipage}
\hfill
\begin{minipage}[t]{0.48\textwidth}
    \centering
    \includegraphics[width=\linewidth]{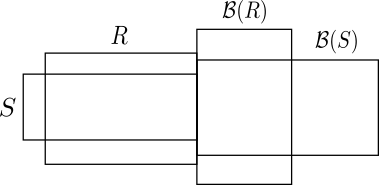}

\end{minipage}
\caption{In both of the above cases, we have that $\widetilde{R}\cap \widetilde{S}=\emptyset$.}
    \label{f.nointersect}
\end{figure}

 Indeed, if $R$ does not intersect $\partial_+^uS$, then $\operatorname{Int}(\mathcal{B}(R))\cap\operatorname{Int}(S)\neq\emptyset$ (see the left panel of Figure \ref{f.nointersect}). By Lemma \ref{l.npredecessor} and the fact that $R$ and $\mathcal{B}(R)$ have disjoint interiors, it follows that $\mathcal{B}(R)$ is also a predecessor of $S$ of some generation. In particular, thanks to Proposition \ref{p.propertiesoffoliformarkovianactions}, $\partial_+^u\mathcal{B}(R)\cap \partial_+^u\mathcal{B}(S)=\emptyset$. By our definition of $\widetilde{\mathcal{R}}$, the previous fact implies that the rectangles $\widetilde{R}$ and $\widetilde{S}$ are disjoint, which contradicts our initial hypothesis. We deduce that $R$ intersects $\partial_+^uS$ and that, since $R$ is a predecessor of $S$ of some generation, $\partial^u_+S\subset \partial^u_+R$. Using the definition of big brother, $\partial_-^u\mathcal{B}(S)\subseteq \partial_-^u\mathcal{B}(R)$; hence, thanks to the Markovian intersection property and to Lemma \ref{l.npredecessor}, either $\mathcal{B}(R)=\mathcal{B}(S)$ or $\mathcal{B}(R)$ is a predecessor of some non-zero generation of $\mathcal{B}(S)$. Assume that $\mathcal{B}(R)$ is a predecessor of some non-zero generation of $\mathcal{B}(S)$  (see the right panel of Figure \ref{f.nointersect}). Once again, by the Markovian intersection property, since $\partial_-^u\mathcal{B}(S)\subseteq \partial_-^u\mathcal{B}(R)$, we have that $\partial_+^u\mathcal{B}(R)\cap \partial_+^u\mathcal{B}(S)=\emptyset$. Therefore, by our definition of $\widetilde{\mathcal{R}}$, the rectangles $\widetilde{R}$ and $\widetilde{S}$ are disjoint, which contradicts our initial hypothesis. It follows that $\mathcal{B}(R)=\mathcal{B}(S)$. 

 Finally, using the definition of $\widetilde{\mathcal{R}}$, $$ \widetilde{R}\cap \widetilde{S}=\{(x,y)\in W^+| x\in R\cap S \text{ and } y\in \partial_+^u\mathcal{B}(R)\}$$
Since $R$ is a predecessor of $S$ of some generation, the previous set is a horizontal subrectangle of $\widetilde{R}$ and a vertical subrectangle of $\widetilde{S}$ projecting to $R\cap S$ under $\pi$. 
    
\end{proof}

Following the notations of Proposition \ref{p.fullyorientablemarkovstepone}, the projection $\mathcal{R}_M$ of $\widetilde{\mathcal{R}}$ on $M$ is almost a Markov partition of $\Phi$, except from the fact that the rectangles in $\mathcal{R}_M$ do not necessarily have disjoint interiors. In order to prove Proposition \ref{p.fullyorientablemarkov}, we will need to show that it is possible to separate the rectangles in $\widetilde{\mathcal{R}}$ by pushing them along the flow $\widetilde{\Phi}$. In general, such a separation is rather complex for arbitrary collections of rectangles in $W^+$ satisfying Items (1)-(3) of Proposition \ref{p.fullyorientablemarkovstepone}. However in our case, this will be possible, because of the particular nature of the intersections of the rectangles in $\widetilde{\mathcal{R}}$ given by Item (4) of Proposition \ref{p.fullyorientablemarkovstepone}. 

\begin{proof}[Proof of Proposition \ref{p.fullyorientablemarkov}]
  Take $\widetilde{\mathcal{R}}$ the collection of rectangles given by Proposition \ref{p.fullyorientablemarkovstepone}. Following the notations of the previous proposition, throughout this proof, for every $R\in\mathcal{R}$, we denote by $\widetilde{R}$ the unique lift of $R$ in $\widetilde{\mathcal{R}}$ and, conversely, for every $\widetilde{S}\in\widetilde{\mathcal{R}}$, we write $S=\pi(\widetilde{S})$.

  If the rectangles in $\widetilde{\mathcal{R}}$ have pairwise disjoint interiors, then by Proposition \ref{p.fullyorientablemarkovstepone}, $\widetilde{\mathcal{R}}$ satisfies all the properties required in Proposition~\ref{p.fullyorientablemarkov}, thereby completing its proof. Assume from now on that there exist rectangles in $\widetilde{\mathcal{R}}$, whose interiors intersect. For every $\widetilde{R}$ we define the \emph{cluster} of $\widetilde{R}$ as follows
\begin{align*}
      \mathfrak{Cl}(\widetilde{R}):=\{\widetilde{S}\in \widetilde{\mathcal{R}}| &\exists\widetilde{R_0}, \widetilde{R_1},...,\widetilde{R_s}\in \widetilde{\mathcal{R}} \text{ such that } \widetilde{R_0}=\widetilde{R}, \widetilde{R_s}=\widetilde{S}, \\& \text{ and } \mathrm{Int}(\widetilde{R_i})\cap \mathrm{Int}(\widetilde{R_{i+1}})\neq \emptyset  \text{ for every } i\in \llbracket 0, s-1\rrbracket \}
\end{align*}

Naturally, $\widetilde{R}\in \mathfrak{Cl}(\widetilde{R})$ for every $\widetilde{R}\in \widetilde{\mathcal{R}}$. Moreover, since $\widetilde{\mathcal{R}}$ is $\widetilde{\rho}$-invariant and consists of finitely many $\widetilde{\rho}$-orbits of rectangles, $\{\mathfrak{Cl}(\widetilde{R})\mid\widetilde{R}\in\widetilde{\mathcal{R}}\}$ defines a $\widetilde{\rho}$-invariant partition of $\widetilde{\mathcal{R}}$, which is finite up to the action $\widetilde{\rho}$. The particular nature of the intersections in $\widetilde{\mathcal{R}}$ (see Item (4) of Proposition \ref{p.fullyorientablemarkovstepone}) imposes a specific structure to every cluster of a rectangle in $\widetilde{\mathcal{R}}$. More specifically, we have that

\textbf{Claim 1:} \textit{Fix $\widetilde{R}\in \widetilde{\mathcal{R}}$. There does not exist a sequence $(\widetilde{R_i})_{i\in\mathbb{N}}$ of rectangles in $\mathfrak{Cl}(\widetilde{R})$ such that $R_{i+1}$ is a successor (resp. predecessor) of some positive generation of $R_i$ for every $i\in \mathbb{N}$.}


\textbf{Claim 2:} \textit{Fix $\widetilde{R}\in\widetilde{\mathcal{R}}$. Consider $\widetilde{R}_1,\widetilde{R}_2\in\mathfrak{Cl}(\widetilde{R})$ such that both $R_1$ and $R_2$ are predecessors of $R$ of some generation. Then, up to exchanging the roles of $R_1$ and $R_2$, we have that $R_1$ is a predecessor of some generation of $R_2$. Also, $\mathrm{Int}(\widetilde{R}_1)\cap\mathrm{Int}(\widetilde{R}_2)\cap \mathrm{Int}(\widetilde{R})\neq \emptyset$.}

\textbf{Claim 3:} \textit{Fix $\widetilde{R}\in \widetilde{\mathcal{R}}$. There exists a unique rectangle $\widetilde{R_{min}}$ in $\mathfrak{Cl}(\widetilde{R})$ such that for every $\widetilde{r}\in \mathfrak{Cl}(\widetilde{R})$ we have that $ R_{min}$ is a predecessor of some generation of $r$.}

\textbf{Claim 4:} \textit{Fix $\widetilde{R}\in \widetilde{\mathcal{R}}$. The set $\mathfrak{Cl}(\widetilde{R})$ is finite.}

\textbf{Claim 5:} \textit{Fix $\widetilde{R}\in \widetilde{\mathcal{R}}$ and $\widetilde{R_{min}}$ the rectangle in $\mathfrak{Cl}(\widetilde{R})$ given by Claim 3. There exists $\epsilon>0$ such that for every $t\in [-\epsilon,0)$ and every $\widetilde{S}\in\widetilde{\mathcal{R}}$ that satisfies $\mathrm{Int}(R_{min})\cap\mathrm{Int}(S)\neq\emptyset$    $$\widetilde{\Phi}^t(\widetilde{R_{min}})\cap  \widetilde{S}=\emptyset$$}

\begin{proof}[Proof of Claim 1]
    Consider $\widetilde{R}\in \widetilde{\mathcal{R}}$ and assume that there exists a sequence $(\widetilde{R_i})_{i\in\mathbb{N}}$ of rectangles in $\mathfrak{Cl}(\widetilde{R})$ such that $R_{i+1}$ is a successor (resp. predecessor) of some positive generation of $R_i$ for every $i\in \mathbb{N}$.  
    
    Notice first that, thanks to the definition of $ \mathfrak{Cl}(\widetilde{R})$ and to Item (4) of Proposition \ref{p.fullyorientablemarkovstepone}, $\mathcal{B}(R')=\mathcal{B}(R)$ for every $\widetilde{R'}\in \mathfrak{Cl}(\widetilde{R})$. In particular, $\mathcal{B}(R_i)=\mathcal{B}(R_{i+1})$ for every $i\in \mathbb{N}$. Thanks to the previous fact and to the finiteness axiom, there exist $i,j\in \mathbb{N}$ and $g\in G-\{e\}$ such that $\rho(g)(R_i)=R_j$. However, by Remark \ref{r.invariantbigbrothers}, we have that $\rho(g)(\mathcal{B}(R_i))=\mathcal{B}(R_j)=\mathcal{B}(R_i)$, which together with Item (1) of Theorem \ref{t.markovianisanosovlike} implies that $g=e$ and leads to an absurdity.

\end{proof} 

\begin{proof}[Proof of Claim 2]
   Consider $\widetilde{R}, \widetilde{R_1}, \widetilde{R_2}$ as in the statement of Claim 2. By the same argument used in the proof of Claim 1, 
   \begin{equation}\label{eq.equalitybigbrothers}
       \mathcal{B}(R_1)=\mathcal{B}(R_2)=\mathcal{B}(R)
   \end{equation} Therefore, by the definition of big brother, $\partial ^u_+R_1\cup \partial ^u_+R_2\subset \partial^u_-\mathcal{B}(R)\subset\mathcal{F}^u(\partial^u_+R)$. Moreover, since $R_1$ and $R_2$ are predecessors of some generation of $R$, we get that  $\partial^u_+R\subset \partial ^u_+R_1 \cap \partial ^u_+R_2$ and thus $\mathrm{Int}(R_1)\cap \mathrm{Int}(R_2)\neq \emptyset$. By Lemma \ref{l.npredecessor}, after possibly changing the roles of $R_1$ and $R_2$, we deduce that $R_1$ is a predecessor of some generation of $R_2$. Finally, the non-emptiness of $\mathrm{Int}(\widetilde{R})\cap\mathrm{Int}(\widetilde{R}_1)\cap \mathrm{Int}(\widetilde{R}_2)$ follows from the non-emptiness of $\mathrm{Int}(R)\cap \mathrm{Int}(R_1)\cap \mathrm{Int}(R_2)$, together with \eqref{eq.equalitybigbrothers} and the definition of $\widetilde{\mathcal{R}}$ (see Proposition~\ref{p.fullyorientablemarkovstepone}).
   
 \end{proof}
\begin{proof}[Proof of Claim 3]
   Consider $\widetilde{R}\in \widetilde{\mathcal{R}}$ and denote by $ \mathfrak{Cl}(R)$ the projection of $\mathfrak{Cl}(\widetilde{R})$ on $\mathcal{P}$. We define $R_{min}$ as the predecessor of $R$ of the biggest generation inside $\mathfrak{Cl}(R)$. Such a rectangle exists and is uniquely defined, thanks to Claims 1 and 2. 

   Consider now $\widetilde{r}\in \mathfrak{Cl}(\widetilde{R})$. By the definition of $\mathfrak{Cl}(\widetilde{R})$, there exists a finite sequence $\widetilde{R_0},...,\widetilde{R_s}$ of rectangles in $\mathfrak{Cl}(\widetilde{R})$ such that $\widetilde{R_0}=\widetilde{R_{min}}$, $\widetilde{R_s}=\widetilde{r}$ and $\mathrm{Int}(\widetilde{R_i})\cap \mathrm{Int}(\widetilde{R_{i+1}})\neq \emptyset$ for every $i\in \llbracket 0, s-1\rrbracket$. We will now prove that $R_{min}$ is a predecessor of some generation of $r$ by induction on $s$.
   
    The statement is clear when $s=0$. Assume now that $R_{min}$ is a predecessor of some generation of $R_{s-1}$. By the previous fact, Claim~2 and Item~(4) of Proposition~\ref{p.fullyorientablemarkovstepone}, we have $$\partial^u_+R_{s-1}\subseteq \partial^u_+R_{min}$$ Moreover, by definition, $\mathrm{Int}(\widetilde{r})\cap \mathrm{Int}(\widetilde{R_{s-1}})\neq\emptyset$, and thus, after applying $\pi$, $\mathrm{Int}(r)\cap \mathrm{Int}(R_{s-1})\neq\emptyset$. We deduce, thanks to Lemma~\ref{l.npredecessor} and to Item~(4) of Proposition~\ref{p.fullyorientablemarkovstepone}, that $r$ is either a predecessor or a successor of some generation of $R_{s-1}$ and that  $$\partial^u_+R_{s-1}\subseteq \partial^u_+r \text{ or } \partial^u_+r\subseteq \partial^u_+R_{s-1}$$ In both of the above cases, using the fact that $\partial^u_+R_{s-1}\subseteq \partial^u_+R_{min}$,  we obtain that $\mathrm{Int}(r)\cap \mathrm{Int}(R_{min})\neq\emptyset$; hence, $r$ is either a successor or a predecessor of some generation of $R_{min}$. The first case implies that $R_{min}$ is a predecessor of $r$ of some generation, whereas the second case implies that $r=R_{min}$ as $R_{min}$ is the predecessor of $R$ of the biggest generation in $\mathfrak{Cl}(R)$.

\end{proof}

\begin{proof}[Proof of Claim 4]
     Consider $\widetilde{R}\in \widetilde{\mathcal{R}}$ and assume that $ \mathfrak{Cl}(\widetilde{R})$ consists of infinitely many rectangles in $\widetilde{\mathcal{R}}$. Denote by $ \mathfrak{Cl}(R)$ the projection of $\mathfrak{Cl}(\widetilde{R})$ on $\mathcal{P}$ and let $\widetilde{R_{min}}$ be the rectangle in $\mathfrak{Cl}(\widetilde{R})$ given by Claim 3. Using the fact that any rectangle in $\mathcal{R}$ admits finitely many successors of first generation, together with the definition of $\widetilde{R_{min}}$, we get that there exists $R_1$, a successor of first generation of $R_{min}$, such that infinitely many successors of some generation of $R_1$ belong in $ \mathfrak{Cl}(R)$. By repeating the previous argument, one can construct an infinite sequence $R_0=R_{min},R_1, R_2,...,R_i,...$ containing infinitely many rectangles in $ \mathfrak{Cl}(R)$ and such that, for every $i\in \mathbb{N}$, the rectangle $R_{i+1}$ is a successor of first generation of $R_i$. Passing to a subsequence of $(R_i)_{i\in \mathbb{N}}$, we obtain a sequence $S_0=R_{min},S_1, S_2,...,S_i,...$ of rectangles in $\mathfrak{Cl}(R)$ such that $S_{i+1}$ is a successor of some positive generation of $S_i$ for every $i\in \mathbb{N}$. This contradicts Claim~1 and completes the proof of Claim~4.
\end{proof}

\begin{proof}[Proof of Claim 5]
Consider $\widetilde{R}\in \widetilde{\mathcal{R}}$ and let $\widetilde{R_{min}}$ be the rectangle in $\mathfrak{Cl}(\widetilde{R})$ given by Claim 3. Assume by contradiction that there exist $(\widetilde{S_n})_{n\in\mathbb{N}}$ a sequence of rectangles in $\widetilde{\mathcal{R}}$ with $\mathrm{Int}(R_{min})\cap\mathrm{Int}(S_n)\neq\emptyset$, $(\widetilde{x_n})_{n\in\mathbb{N}}$ a sequence of points in $\widetilde{R_{min}}$ and $(\epsilon_n)_{n\in\mathbb{N}}$ a sequence of strictly positive real numbers going to zero such that 
$$\widetilde{\Phi}^{-\epsilon_n}(\widetilde{x_n})\in \widetilde{S_n}$$

After passing to a subsequence if necessary, assume that the sequence $(\widetilde{x_n})_{n\in\mathbb{N}}$ converges to a point $\widetilde{X}\in \widetilde{R_{min}}$. Moreover, since $\widetilde{\mathcal{R}}$ consists of finitely many $\widetilde{\rho}$-orbits of rectangles and $\widetilde{\rho}$ is properly discontinuous, any compact neighborhood of $\widetilde{R_{min}}$ intersects only finitely many rectangles in $\widetilde{\mathcal{R}}$. We deduce that after passing to a further subsequence, we can assume that $\widetilde{S_n}=\widetilde{S_0}$ for every $n\in \mathbb{N}$. 

By passing to the limit as $n\rightarrow+\infty$, we obtain that $\widetilde{X}\in \widetilde{S_0}\cap \widetilde{R_{min}}$. Therefore,  since $\mathrm{Int}(R_{min})\cap\mathrm{Int}(S_0)\neq\emptyset$, Item (4) of Proposition \ref{p.fullyorientablemarkovstepone} implies that $$\pi(\widetilde{S_0}\cap \widetilde{R_{min}})=R_{min}\cap S_0$$ In other words, given that each orbit of $\widetilde{\Phi}$ corresponds to the preimage by $\pi$ of a point in $\mathcal{P}$, every orbit of $\widetilde{\Phi}$ that intersects both $\widetilde{S}$ and $\widetilde{R_{min}}$, intersects the two rectangles along a unique point belonging to $\widetilde{S_0}\cap \widetilde{R_{min}}$. This contradicts the existence of $x_0\in \widetilde{R_{min}}$ and $\epsilon_0<0$ such that $\widetilde{\Phi}^{-\epsilon_0}(\widetilde{x_0})\in \widetilde{S_0}$. 
\end{proof}

Recall that clusters in $\widetilde{\mathcal{R}}$ form a $\widetilde{\rho}$-invariant partition of $\widetilde{\mathcal{R}}$ that is finite up the action of $\widetilde{\rho}$. Consider $\mathfrak{Cl}(\widetilde{\mathcal{R}_0})$,  $\mathfrak{Cl}(\widetilde{\mathcal{R}_1})$, ...,  $\mathfrak{Cl}(\widetilde{\mathcal{R}_m})$ a representative of each $\widetilde{\rho}$-orbit of clusters in $\widetilde{\mathcal{R}}$. We define $N:=|\mathfrak{Cl}(\widetilde{\mathcal{R}_1})|+ ...+ |\mathfrak{Cl}(\widetilde{\mathcal{R}_m})|-m\geq 0$ as the \emph{total number of intersections of $\widetilde{\mathcal{R}}$}. We will now push the rectangles of $\widetilde{\mathcal{R}}$ along the flow $\widetilde{\Phi}$, in order to produce a new collection of rectangles satisfying Items (1)-(4) of Proposition \ref{p.fullyorientablemarkovstepone} and having a total number of intersections strictly smaller than $N$. 

Fix $\widetilde{R}\in \widetilde{\mathcal{R}}$ such that $ \{\widetilde{R}\} \subsetneq \mathfrak{Cl}(\widetilde{R})$. Let $\widetilde{R_{min}}$ be the rectangle in $ \mathfrak{Cl}(\widetilde{R})$ given by Claim 3 and $\epsilon>0$ be the constant given by Claim 5. Denote by $\widetilde{\mathcal{R}_{new}}$ the collection of rectangles in $W^+$ obtained from $\widetilde{\mathcal{R}}$ after replacing for every $g\in \pi_1(M)$ the rectangle $\widetilde{\rho}(g)(\widetilde{R_{min}})$ by the rectangle $\widetilde{\Phi}^{-\epsilon}(\widetilde{\rho}(g)(\widetilde{R_{min}}))$. 

\vspace{0.15cm}
\textbf{Claim 6:} \textit{$\widetilde{\mathcal{R}_{new}}$ is invariant by $\widetilde{\rho}$ and satisfies Items (1)-(4) of Proposition \ref{p.fullyorientablemarkovstepone}}
\begin{proof}[Proof of Claim 6]
    
First notice that, since $\widetilde{\Phi}$ is invariant by $\widetilde{\rho}$, the set $\{\widetilde{\Phi}^{-\epsilon}(\widetilde{\rho}(g)(\widetilde{R_{min}}))|g\in G\}$ forms a $\widetilde{\rho}$-orbit of rectangles in $W^+$. It follows that $\widetilde{\mathcal{R}_{new}}$ is invariant by $\widetilde{\rho}$. Next, using the fact that $\widetilde{\mathcal{R}_{new}}$ was obtained by pushing the rectangles of $\widetilde{\mathcal{R}}$ along $\widetilde{\Phi}$ for a uniformly bounded time, it is easy to verify that $\widetilde{\mathcal{R}_{new}}$ satisfies Items (1) and (2) of Proposition~\ref{p.fullyorientablemarkovstepone}. 

\vspace{0.15cm}
\textit{$\widetilde{\mathcal{R}_{new}}$ satisfies Item (3) of Proposition \ref{p.fullyorientablemarkovstepone}.}

By our previous arguments, $\widetilde{\mathcal{R}_{new}}$ is invariant by $\widetilde{\rho}$ and lifts $\mathcal{R}$ on $W^+$. It follows that $\widetilde{\mathcal{R}_{new}}$ consists of finitely many $\widetilde{\rho}$-orbits of rectangles and that its projection on $M$ consists of finitely many rectangles. Next, the fact that the projection of $\widetilde{\mathcal{R}_{new}}$ on $M$ satisfies Items (2) and (4) of Definition \ref{d.markovpartition}, follows directly from the construction of $\widetilde{\mathcal{R}_{new}}$ and the fact that $\widetilde{\mathcal{R}}$ satisfies Item (3) of Proposition \ref{p.fullyorientablemarkovstepone}. 

Finally, assume by contradiction that the projection of $\widetilde{\mathcal{R}_{new}}$ on $M$ does not satisfy Item (3) of Definition \ref{d.markovpartition}. In that case, there exist two distinct rectangles $\widetilde{S_1},\widetilde{S_2}\in\widetilde{\mathcal{R}_{new}}$ for which one of the following holds:
      
      \begin{itemize}
          \item There exist $\widetilde{x}\in\partial^s \widetilde{S_1}$ and $t_0>0$ such that $\widetilde{\Phi}^{t_0}(\widetilde{x})$ belongs to the interior of $\widetilde{S_2}$.
          \item There exist $\widetilde{x}\in\partial^u \widetilde{S_1}$ and $t_0<0$ such that $\widetilde{\Phi}^{t_0}(\widetilde{x})$ belongs to the interior of $\widetilde{S_2}$.
      \end{itemize}
      
      Suppose without any loss of generality that we are in the situation of the first of the above two cases (the second case can be treated by a similar argument). Notice that $\widetilde{S_1},\widetilde{S_2}$ can not both belong to $\widetilde{\mathcal{R}}$, as this would contradict Item (3) of Proposition \ref{p.fullyorientablemarkovstepone}. Similarly, if $\widetilde{S_1}=\widetilde{\Phi}^{-\epsilon}(\widetilde{\rho}(g_1)(\widetilde{R_{min}}))$ and $\widetilde{S_2}=\widetilde{\Phi}^{-\epsilon}(\widetilde{\rho}(g_2)(\widetilde{R_{min}}))$ for some $g_1,g_2\in G$, then the strictly positive orbit of $\widetilde{\Phi}^{\epsilon}(\widetilde{x})\in \partial^s(\widetilde{\rho}(g_1)(\widetilde{R_{min}}))$ would intersect the interior of $\widetilde{\rho}(g_2)(\widetilde{R_{min}})$, where both $\widetilde{\rho}(g_1)(\widetilde{R_{min}})$ and $\widetilde{\rho}(g_2)(\widetilde{R_{min}})$ belong to $\widetilde{\mathcal{R}}$. Once again, this contradicts Item (3) of Proposition \ref{p.fullyorientablemarkovstepone}. One can similarly treat the case where $\widetilde{S_1}\in \widetilde{\mathcal{R}}$ and $\widetilde{S_2}=\widetilde{\Phi}^{-\epsilon}(\widetilde{\rho}(g_2)(\widetilde{R_{min}}))$ for some $g_2\in G$.

      It remains to consider the case where $\widetilde{S_1}=\widetilde{\Phi}^{-\epsilon}(\widetilde{\rho}(g_1)(\widetilde{R_{min}}))$ and $\widetilde{S_2}\in \widetilde{\mathcal{R}}$ for some $g_1\in G$. In this case, recall that, by our choice of $\epsilon$ (see Claim 5), for every $t\in[0,\epsilon)$ the point $\widetilde{\Phi}^{t}(\widetilde{x})$ does not belong in the  interior of any rectangle in $\widetilde{\mathcal{R}}$. If $\widetilde{\Phi}^{\epsilon}(\widetilde{x})\notin \widetilde{S_2}$, then, similarly to before, one obtains a contradiction with Item (3) of Proposition \ref{p.fullyorientablemarkovstepone}. It follows that $\widetilde{\Phi}^{\epsilon}(\widetilde{x})\in \widetilde{S_2}$ and thus that $\widetilde{\rho}(g_1)(\widetilde{R_{min}})\cap \mathrm{Int}(\widetilde{S_2})\neq \emptyset$. By applying $\pi$, we get that $\rho(g_1)(R_{min})\cap \mathrm{Int}(S_2)\neq \emptyset$; hence, $\mathrm{Int}(\rho(g_1)(R_{min}))\cap \mathrm{Int}(S_2)\neq \emptyset$ and by Item (4) of Proposition \ref{p.fullyorientablemarkovstepone}, $\widetilde{S_2}\in \mathfrak{Cl}(\widetilde{\rho}(g_1)(\widetilde{R_{min}}))$. However, by the definition of $\widetilde{R_{min}}$, we have that $\rho(g_1)(R_{min})$ is a predecessor of $S_2$ of some generation; hence, no point in the $\mathcal{F}^s$-boundary of $\rho(g_1)(R_{min})$ can intersect the interior of $S_2$, which is incompatible with the existence of $\widetilde{x}$ and yields the desired result.

\vspace{0.15cm}      
\textit{$\widetilde{\mathcal{R}_{new}}$ satisfies Item (4) of Proposition \ref{p.fullyorientablemarkovstepone}.}

Consider two distinct rectangles $\widetilde{S_1},\widetilde{S_2}\in \widetilde{\mathcal{R}_{new}}$ such that $\widetilde{S_1}\cap \widetilde{S_2}\neq \emptyset$ and $\mathrm{Int}(\pi(\widetilde{S_1}))\cap \mathrm{Int}(\pi(\widetilde{S_2}))\neq \emptyset$. If both $\widetilde{S_1}$ and $\widetilde{S_2}$ belong in $\widetilde{\mathcal{R}}$, then desired result is an immediate consequence of Item (4) of Proposition \ref{p.fullyorientablemarkovstepone}. Next, assume that $\widetilde{S_1}=\widetilde{\Phi}^{-\epsilon}(\widetilde{\rho}(g)(\widetilde{R_{min}}))$ for some $g\in G$ and that $\widetilde{S_2}\in \widetilde{\mathcal{R}}$. In this case, $\widetilde{S_1}$ and $\widetilde{S_2}$ can not verify the previous conditions, thanks to our choice of $\epsilon$. 

Finally, assume that $\widetilde{S_1}=\widetilde{\Phi}^{-\epsilon}(\widetilde{\rho}(g_1)(\widetilde{R_{min}}))$ and $\widetilde{S_2}=\widetilde{\Phi}^{-\epsilon}(\widetilde{\rho}(g_2)(\widetilde{R_{min}}))$ for some $g_1,g_2\in G$ with $g_1\neq g_2$. Thanks to our original hypotheses, the interior of $\pi(\widetilde{S_1})=\rho(g_1)(R_{min})\in \mathcal{R}$ intersects the interior of $\pi(\widetilde{S_2})=\rho(g_2)(R_{min})\in \mathcal{R}$, while $\widetilde{\rho}(g_1)(\widetilde{R_{min}})\in \widetilde{\mathcal{R}}$ intersects $\widetilde{\rho}(g_2)(\widetilde{R_{min}})\in \widetilde{\mathcal{R}}$. By Item (4) Proposition \ref{p.fullyorientablemarkovstepone}, this implies that 
$\widetilde{\rho}(g_1)(\widetilde{R_{min}})\in \mathfrak{Cl}(\widetilde{\rho}(g_2)(\widetilde{R_{min}}))$. Since clusters define a $\widetilde{\rho}-$invariant partition of $\widetilde{\mathcal{R}}$, it follows that $\mathfrak{Cl}(\widetilde{\rho}(g_2)(\widetilde{R_{min}}))$ is invariant by $\widetilde{\rho}(g_1g_2^{-1})$. This is impossible, as $G$ is torsion-free, $g_1\neq g_2$ and $\mathfrak{Cl}(\widetilde{S_2})$ is finite.
\end{proof}

An important consequence of the last part of our proof of Claim 6 is that any two rectangles in $\widetilde{\mathcal{R}_{new}}$ that intersect along their interiors, necessarily belong in $\widetilde{\mathcal{R}}$. Using the previous fact, one can easily check that Claims~1, 2, and~4 remain true for $\widetilde{\mathcal{R}_{new}}$ and that the total number of intersections of $\widetilde{\mathcal{R}_{new}}$ is strictly smaller than $N$. Concerning Claims~3 and 5, their validity for $\widetilde{\mathcal{R}_{new}}$ follows from the same argument as the one used for $\widetilde{\mathcal{R}}$. 

Thanks to our previous arguments, we can construct from $\widetilde{\mathcal{R}}$ another $\widetilde{\rho}$-invariant collection of rectangles in $W^+$ whose total number of intersections is strictly smaller than that of $\widetilde{\mathcal{R}}$, while satisfying Items (1)-(4) of Proposition~\ref{p.fullyorientablemarkovstepone} as well as Claims~1-5. By a finite repetition of this process, we can produce a $\widetilde{\rho}$-invariant collection of rectangles $\widetilde{\mathcal{R}_{final}}$ satisfying Items (1)-(4) of Proposition \ref{p.fullyorientablemarkovstepone} and such that total number of intersections of $\widetilde{\mathcal{R}_{final}}$ is equal to zero. In other words, any two rectangles in $\widetilde{\mathcal{R}_{final}}$ have disjoint interiors. By the first part of the proof of this proposition, we deduce that the projection $\widetilde{\mathcal{R}_{final}}$ in $M$ defines a reduced Markov partition of $\Phi$, which proves the desired result.

\end{proof}

\section{The non-fully orientable case}

In Sections~\ref{s.fullyorientmanifold},~\ref{s.fullyorientflow} and~\ref{s.fullyorientmarkovpartition}, we proved Theorem~\ref{t.main} for fully orientable strong Markovian actions on the plane. Our goal in this section consists in showing how one can remove the full orientability hypothesis.

Let $( \mathcal{P}, \mathcal{F}^s, \mathcal{F}^u)$ be a bifoliated plane endowed with an orientation preserving non-fully orientable strong Markovian action $\rho$. By Proposition \ref{p.cornerconditionexists}, $\rho$ preserves a strong Markovian family $\mathcal{R}$ satisfying the corner condition. Fix for the rest of this section an orientation of $\mathcal{F}^s$ and $\mathcal{F}^u$. 

Denote by $G_+$ the index two subgroup of $G$, whose action on $\mathcal{P}$ via $\rho$ preserves the orientations of both $\mathcal{F}^s$ and $\mathcal{F}^u$. Denote also by $\rho_+: G_+\rightarrow \text{Homeo}( \mathcal{P}) $ the restriction of $\rho$ on $G_+$. The action $\rho_+$ is a fully orientable strong Markovian action. Similarly to the previous sections, we define $W^+:=\{( x,y) \in \mathcal{P}^2\mid y\in \mathcal{F}^s_+( x) -\{x\}\}$ and $\pi: W^+\rightarrow \mathcal{P}$ by $\pi( ( x,y))=x$ for every $( x,y)\in W^+$.

By Proposition~\ref{p.wplusisasimplyconnectedmanifold} and our discussion in Section~\ref{s.fullyorientflow}, $W^+$ is homeomorphic to $\mathbb{R}^3$ and $\pi$ defines a trivial line bundle over $\mathcal{P}$ whose fiber over each $x\in \mathcal{P}$ is $\{( x,y)|y\in \mathcal{F}_+^s( x)-\{x\}\}$. Moreover, by Propositions \ref{p.actionproperdisc} and \ref{p.actionboundedfund}, the diagonal action of $G_+$ on $W^+$, say $\widetilde{\rho_+}:G_+\rightarrow \text{Homeo}( W^+)$, is properly discontinuous and admits a bounded fundamental domain. Denote by $M_+$ the quotient of $W^+$ by $\widetilde{\rho_+}$. The set of fibers  of $\pi$ defines an $1$-dimensional foliation on $W^+$ that is invariant by $\widetilde{\rho_+}$ and whose projection on $M_+$ endowed with any non-singular parametrization defines a topological Anosov flow $\Phi_+$ ( see Proposition \ref{p.pseudoanosov}). 

Recall that by our discussion at the end of Section \ref{s.fullyorientflow}, 

\begin{itemize} 
    \item $\pi_1( M_+)\cong G_+$.
    \item $( \mathcal{P},\mathcal{F}^s,\mathcal{F}^u)$ is the bifoliated plane of $\Phi_+$. 
    \item The action of $\pi_1( M_+)$ on $W^+$ by deck transformations coincides with $\widetilde{\rho_+}$. 
    \item $\rho_+$ corresponds to the action of $\pi_1( M_+)$ on the bifoliated plane of $\Phi_+$.
\end{itemize}

In view of Theorem \ref{t.main}, we would now like to construct an Anosov flow on a closed 3-manifold, whose fundamental group is exactly $G$. In order to do so, we will prove that the action of $\widetilde{\rho_+}$ can be extended to a properly discontinuous action of $G$ on $W^+$.

\needspace{5\baselineskip}
\subsection{From a non-fully orientable strong Markovian action to an Anosov flow}

\begin{prop}\label{p.homeoconstructalmostorient}
   Let $h\in G-G_+$. There exists $f_h$ a homeomorphism of $M_+$ and a lift $\widetilde{f_h}$ of $f_h$ on $W^+$ such that: 
    \begin{enumerate}
        \item $f_h$ defines a self-orbit equivalence of $\Phi_+$ 
        \item $f_h\circ f_h=id$ 
        \item $f_h$ has no fixed points in $M_+$ 
        \item $\rho( h)\circ\pi=\pi \circ \widetilde{f_h}$
    \end{enumerate}
\end{prop} 

\begin{proof}
    Denote by $\alpha$ the inner automorphism of $G$ defined by $\alpha( g)=hgh^{-1}$ for every $g\in G$. Note that $\alpha( G_+)=G_+$. The map $\rho( h)$ defines a homeomorphism of $\mathcal{P}$ such that $\rho( h)( \mathcal{F}^{s,u})=\mathcal{F}^{s,u}$ and also such that for every $g\in G$ and every $x\in \mathcal{P}$ $$\rho( h)( \rho( g)( x))=\rho( \alpha( g))( \rho( h)( x))$$

    By Proposition 1.36 of \cite{hdrbarbot}, there exists $f_{0}:M_+\rightarrow M_+$ a homeomorphism defining self-orbital equivalence of $\Phi_+$ and $\widetilde{f_0}$ a lift of $f_0$ on $W^+$ such that 
    \begin{equation}\label{eq.projectionf0}
        \rho( h)\circ\pi=\pi \circ \widetilde{f_0}
    \end{equation}

    Even though $\widetilde{f_0}$ projects to $\rho( h)$ on $\mathcal{P}$, the group generated by $\widetilde{f_0}$ and $\widetilde{\rho_+}( G_+)$ is not necessarily isomorphic to $G$; in particular, $\widetilde{f_0}\circ \widetilde{f_0}$ may not even belong to the group $\widetilde{\rho_+}( G_+)$. In order to address this problem, we perturb $\widetilde{f_0}$ by flowing it for a suitable time along $\widetilde{\Phi}_+$, the lift of $\Phi_+$ to $W^+$. More specifically, we will define a continuous function $T:W^+\rightarrow \mathbb{R}$ such that the map $\widetilde{f_h}:=\widetilde{\Phi}_+^{T}\circ \widetilde{f_0}$ satisfies
    $$
        \widetilde{f_h}\circ \widetilde{f_h}=\widetilde{\rho_+}( h^2)  
    $$
  Once we construct and study the properties of $\widetilde{f_h}$, we will then show that its projection on $M_+$ defines a homeomoprhism $f_h$ satisfying Items (1)-(4) of this proposition.

\vspace{0.5cm}
    \textit{On the existence of $T$}

    Let us first show that for every $x\in W^+$ there exists a unique $T( x)\in \mathbb{R}$ such that 
    \begin{equation}\label{eq.defiT}
        \left( \widetilde{\Phi}_+^{T( x)}\circ \widetilde{f_0}\circ \widetilde{\Phi}_+^{T( x)}\circ \widetilde{f_0}\right)( x)=\widetilde{\rho_+}( h^2)( x)
    \end{equation} 
    
    Consider $j:\mathbb{R}\rightarrow W^+$ the embedding defined by $j( t)=\widetilde{\Phi}_+^{t}\left( \widetilde{\rho_+}( h^2)( x)\right)$ for every $t\in \mathbb{R}$. By definition, the homeomorphism  $\widetilde{f_0}$ preserves the positive and negative orbits of $\widetilde{\Phi}_+$
    and also verifies $$\pi\circ ( \widetilde{f_0}\circ \widetilde{f_0})=\rho( h^2)\circ \pi$$ It follows that for every $t\in \mathbb{R}$ the points $\left( \widetilde{\Phi}_+^{t}\circ \widetilde{f_0}\circ \widetilde{\Phi}_+^{t}\circ \widetilde{f_0}\right)( x)$ and $\widetilde{\rho_+}( h^2)( x)$ belong in the same $\widetilde{\Phi}_+$-orbit and also that the map $i:\mathbb{R}\rightarrow \mathbb{R}$ defined by $i( t)=j^{-1}\left( \widetilde{\Phi}_+^{t}\circ \widetilde{f_0}\circ \widetilde{\Phi}_+^{t}\circ \widetilde{f_0}( x)\right)$ is an increasing homeomorphism of $\mathbb{R}$. We deduce that there exists a unique $T( x)\in \mathbb{R}$ such that $\left( \widetilde{\Phi}_+^{T( x)}\circ \widetilde{f_0}\circ \widetilde{\Phi}_+^{T( x)}\circ \widetilde{f_0}\right)( x)=\widetilde{\rho_+}( h^2)( x)$. 

    Next, let us prove that for every $x\in W^+$ we have that \begin{equation}\label{eq.firstinvariance}
        T \left(( \widetilde{\Phi}_+^{T( x)}\circ \widetilde{f_0})( x)\right)=T( x)    \end{equation}

    By our previous arguments, $T_0=T\left((\widetilde{ \Phi}_+^{T( x)}\circ \widetilde{f_0})( x)\right)$ is the unique real number for which $$\left( \widetilde{\Phi}_+^{T_0}\circ \widetilde{f_0}\circ \widetilde{\Phi}_+^{T_0}\circ \widetilde{f_0}\right)\left(( \widetilde{\Phi}_+^{T( x)}\circ \widetilde{f_0})( x)\right)=\widetilde{\rho_+}( h^2)\left(( \widetilde{ \Phi}_+^{T( x)}\circ \widetilde{f_0})( x)\right)$$ Therefore, in order to prove that $T_0=T( x)$ it suffices to prove that 
    \begin{equation}\label{eq.sufficient}
        \left( \widetilde{\Phi}_+^{T( x)}\circ \widetilde{f_0}\circ \widetilde{\Phi}_+^{T( x)}\circ \widetilde{f_0}\right)\left((\widetilde{\Phi}_+^{T( x)}\circ \widetilde{f_0})( x)\right)=\widetilde{\rho_+}( h^2)\left( (\widetilde{\Phi}_+^{T( x)}\circ \widetilde{f_0})( x)\right)
    \end{equation}

    First, notice that Equation~\eqref{eq.projectionf0}, together with the facts that $\widetilde{\rho_+}$ coincides with the action of $\pi_1(M_+)\cong G_+$ on $W^+$ by deck transformations and that $\widetilde{f_0}$ is the lift of a homeomorphism of $M_+$, implies that
    \begin{equation}\label{eq.almostcommutef0}
        \widetilde{f_0}\circ \widetilde{\rho_+}( g)=\widetilde{\rho_+}( \alpha( g)) \circ \widetilde{f_0}
    \end{equation}

Moreover, since $\widetilde{\rho_+}$ preserves $\widetilde{\Phi}_+$,

    \begin{align*}
       \left( \widetilde{\Phi}_+^{T( x)}\circ \widetilde{f_0}\circ \widetilde{\Phi}_+^{T( x)}\circ \widetilde{f_0}\right)\left((\widetilde{ \Phi}_+^{T( x)}\circ \widetilde{f_0})( x)\right)&=  \left( \widetilde{\Phi}_+^{T( x)}\circ \widetilde{f_0}\right)\left( ( \widetilde{\Phi}_+^{T( x)}\circ \widetilde{f_0}\circ \widetilde{\Phi}_+^{T( x)}\circ \widetilde{f_0})( x)\right)\\
       &=\left( \widetilde{\Phi}_+^{T( x)}\circ \widetilde{f_0}\right)\big( \widetilde{\rho_+}( h^2)( x)\big)\\&=\left( \widetilde{\Phi}_+^{T( x)}\circ\widetilde{\rho_+}( \alpha( h^2))\right)( \widetilde{f_0}( x)) \\&=\left( \widetilde{\Phi}_+^{T( x)}\circ \widetilde{\rho_+}( h^2)\right)( \widetilde{f_0}( x))\\&=\widetilde{\rho_+}( h^2)\left((\widetilde{ \Phi}_+^{T( x)}\circ \widetilde{f_0})( x)\right)
    \end{align*}
     We have thus proven Equation (\ref{eq.sufficient}), which in turn proves Equation (\ref{eq.firstinvariance}). Finally, thanks to Equations (\ref{eq.defiT}) and (\ref{eq.firstinvariance}) we have that for every $x\in W^+$, if $y=( \widetilde{\Phi}_+^{T( x)}\circ \widetilde{f_0})( x)$, then 
$$\left( \widetilde{\Phi}_+^{T( y)}\circ \widetilde{f_0}\circ \widetilde{\Phi}_+^{T( x)}\circ \widetilde{f_0}\right)( x)=\widetilde{\rho_+}( h^2)( x)$$
    which proves the desired result, namely that $\widetilde{f_h}\circ \widetilde{f_h}=\rho( h^2)$.

\vspace{0.5cm}
    \textit{The map $\widetilde{f_h}$ is a homeomorphism of $W^+$}

    Indeed, since $\widetilde{f_h}\circ \widetilde{f_h}$ is a homeomorphism of $W^+$, the map $\widetilde{f_h}$ is clearly bijective. Moreover, since $\widetilde{f_h}\circ \widetilde{f_h}=\widetilde{\rho_+}( h^{2})$ we have that $( \widetilde{f_h})^{-1}=\widetilde{\rho_+}( h^{-2})\circ \widetilde{f_h}$. We deduce that in order to show that $\widetilde{f_h}$ is a homeomorphism it suffices to show that  $\widetilde{f_h}$ is continuous. Consider $( x_n)_{n\in\mathbb{N}}$ a sequence of points in $W^+$ converging to $x_{\infty}$. By  Equation (\ref{eq.defiT}) we have that for every $n\in \mathbb{N}$ 
    \begin{equation}\label{eq.sequence}
        \left( \widetilde{\Phi}_+^{T( x_n)}\circ \widetilde{f_0}\circ \widetilde{\Phi}_+^{T( x_n)}\circ \widetilde{f_0}\right)( x_n)=\widetilde{\rho_+}( h^2)( x_n)
    \end{equation}
    
    Using the fact that the orbits of $\widetilde{\Phi}_+$ define a trivial line bundle over $\mathcal{P}$ and that $\widetilde{\Phi}_+$ is is the lift on $W^+$ of a non-singular flow of the closed manifold $M_+$, we get that for any compact $K\subset W^+$ there exists $t_0( K)>0$ such that for any $t>t_0$ and any $y\in W^+$ close to $x_{\infty}$ we have that $$\left( \widetilde{\Phi}_+^{t}\circ \widetilde{f_0}\circ \widetilde{\Phi}_+^{t}\circ \widetilde{f_0}\right)( y)\notin K$$ 
    
    Thanks to the previous fact, to Equation (\ref{eq.sequence}) and to the fact that $\widetilde{\rho_+}( h^2)(x_n)$ converges to $\widetilde{\rho_+}( h^2)(x_{\infty})$ as $n\rightarrow +\infty$, we deduce that the sequence $( T( x_n))_{n\in\mathbb{N}}$ is bounded. By possibly extracting a subsequence assume that  $( T( x_n))_{n\in\mathbb{N}}$ converges to $T_{\infty}$. By Equations (\ref{eq.defiT}) and (\ref{eq.sequence}) we get that
    $$ \left( \widetilde{\Phi}_+^{T_{\infty}}\circ \widetilde{f_0}\circ \widetilde{\Phi}_+^{T_\infty}\circ \widetilde{f_0}\right)( x_\infty)=\widetilde{\rho_+}( h^2)( x_\infty)$$ 
     $$\left( \widetilde{\Phi}_+^{T( x_\infty)}\circ \widetilde{f_0}\circ \widetilde{\Phi}_+^{T( x_\infty)}\circ \widetilde{f_0}\right)( x_\infty)=\widetilde{\rho_+}( h^2)( x_\infty)$$

    Recall now that $T(x_\infty)$ is the unique real number $t$ for which $\left( \widetilde{\Phi}_+^{t}\circ \widetilde{f_0}\circ \widetilde{\Phi}_+^{t}\circ \widetilde{f_0}\right)( x_\infty)=\widetilde{\rho_+}( h^2)( x_\infty)$. We deduce that $T_\infty=T( x_\infty)$; hence, $T$ and $\widetilde{f_h}$ are both continuous, which proves the desired result.  

    \vspace{0.5cm}
    \textit{The homeomorphism $\widetilde{f_h}$ projects to a homeomorphism $f_h$ of $M_+$}

    In order to prove the previous statement, let us first prove that the function $T:W^+\rightarrow \mathbb{R}$ is invariant under the action of $\widetilde{\rho_+}$. Consider $g\in G_+$ and $x\in W^+$. Since $T( \widetilde{\rho_+}( g)( x))$ is the unique real number $t$ for which 
$$\left( \widetilde{\Phi}_+^{t}\circ \widetilde{f_0}\circ \widetilde{\Phi}_+^{t}\circ \widetilde{f_0}\right)( \widetilde{\rho_+}( g)( x))=\widetilde{\rho_+}( h^2)( \widetilde{\rho_+}( g)( x))$$
in order to prove that $T( \widetilde{\rho_+}( g)( x))=T( x)$, it suffices to show that 
\begin{equation}\label{eq.sufficient2}
        \left( \widetilde{\Phi}_+^{T( x)}\circ \widetilde{f_0}\circ \widetilde{\Phi}_+^{T( x)}\circ \widetilde{f_0}\right)( \widetilde{\rho_+}( g)( x))=\widetilde{\rho_+}( h^2)( \widetilde{\rho_+}( g)( x))
    \end{equation}

    Notice first that by Equation (\ref{eq.defiT}), we have that
$$\left( \widetilde{\Phi}_+^{T( x)}\circ \widetilde{f_0}\circ \widetilde{\Phi}_+^{T( x)}\circ \widetilde{f_0}\right)( x)=\widetilde{\rho_+}( h^2)( x)$$ Next, using  Equation (\ref{eq.almostcommutef0}) and the fact that $\widetilde{\rho_+}$ preserves $\widetilde{\Phi}_+$, we get that 

\begin{align*}
    \widetilde{\rho_+}( h^2)( \widetilde{\rho_+}( g)( x))&=\widetilde{\rho_+}( h^2gh^{-2})( \widetilde{\rho_+}( h^2)( x))\\&=\left( \widetilde{\rho_+}( h^2gh^{-2})\circ\widetilde{\Phi}_+^{T( x)}\circ \widetilde{f_0}\circ \widetilde{\Phi}_+^{T( x)}\circ \widetilde{f_0}\right)( x)\\
     & =\left( \widetilde{\Phi}_+^{T( x)}\circ\widetilde{\rho_+}( h^2gh^{-2}) \circ \widetilde{f_0}\circ \widetilde{\Phi}_+^{T( x)}\circ \widetilde{f_0}\right)( x)\\
    &=\left( \widetilde{\Phi}_+^{T( x)}\circ\widetilde{f_0}\circ \widetilde{\rho_+}( hgh^{-1}) \circ \widetilde{\Phi}_+^{T( x)}\circ \widetilde{f_0}\right)( x)\\
    &=\left( \widetilde{\Phi}_+^{T( x)}\circ\widetilde{f_0}\circ\widetilde{\Phi}_+^{T( x)}\circ \widetilde{\rho_+}( hgh^{-1})\circ \widetilde{f_0}\right)( x)\\
   &=\left( \widetilde{\Phi}_+^{T( x)}\circ\widetilde{f_0}\circ\widetilde{\Phi}_+^{T( x)}\circ \widetilde{f_0}\right)( \widetilde{\rho_+}( g)( x))
\end{align*}

We have thus proven that Equation (\ref{eq.sufficient2}) holds, which in turn proves that $T$ is invariant by the action of $\widetilde{\rho_+}$.  We are now ready to prove that $\widetilde{f_h}$ projects to a homeomorphism of $M_+$. In order to do so, it suffices to prove that for any $g\in G_+\cong \pi_1( M_+)$ 
\begin{equation}\label{eq.almostcommutefh}
    \widetilde{\rho_+}( g)\circ \widetilde{f_h}=\widetilde{f_h}\circ \widetilde{\rho_+}( h^{-1}gh)
\end{equation}
The above equation follows from the invariance of $T$ by $\widetilde{\rho_+}$ and Equation (\ref{eq.almostcommutef0}). More specifically, for every $x\in W^+$
\begin{align*}
    ( \widetilde{\rho_+}( g)\circ \widetilde{f_h})( x)&= \left( \widetilde{\rho_+}( g)\circ \widetilde{\Phi}^{T( x)}\circ \widetilde{f_0}\right)( x)\\
    &=\left( \widetilde{\Phi}^{T( x)}\circ\widetilde{\rho_+}( g)\circ \widetilde{f_0}\right)( x)\\
    &=\left( \widetilde{\Phi}^{T( x)}\circ\widetilde{f_0}\circ \widetilde{\rho_+}( h^{-1}gh)\right)( x)
    \\
    &=\left( \widetilde{\Phi}^{T( \widetilde{\rho_+}( h^{-1}gh)( x))}\circ\widetilde{f_0}\right)\big( \widetilde{\rho_+}( h^{-1}gh)( x)\big)\\
    &= \left( \widetilde{f_h}\circ \widetilde{\rho_+}( h^{-1}gh)\right)( x)
\end{align*}
which proves the desired result. We deduce that $\widetilde{f_h}$ projects to a homeomorphism $f_h$ of $M_+$. 

\vspace{0.5cm}
\textit{The homeomorphism $f_h$ satisfies Items (1)-(4) of this proposition}

Indeed, since $\widetilde{f_0}$ preserves the positive and negative orbits of $\widetilde{\Phi}_+$, the definition of $\widetilde{f_h}$ implies that $\widetilde{f_h}$ also preserves the positive and negative orbits of $\widetilde{\Phi}_+$. We deduce that $f_h$ is a self-orbit equivalence of $\Phi_+$. Next, by construction $\widetilde{f_h}\circ \widetilde{f_h}=\widetilde{\rho_+}( h^2)$, where $h^2\in G_+\cong \pi_1( M_+)$. It follows that $f_h\circ f_h=id_{M_+}$. 

Concerning Item (3) of this proposition, assume that $f_h$ admits a fixed point in $M_+$. In this case, there exists $\widetilde{F_h}$ a lift of $f_h$ on $W^+$ that also admits a fixed point. By a classical result of covering space theory, one can find $g\in G_+$ such that $\widetilde{F_h}=\widetilde{\rho_+}( g)\circ \widetilde{f_h}$. Using Equation (\ref{eq.almostcommutefh}), we get that 

\begin{align*}
    \widetilde{F_h}\circ \widetilde{F_h}&= \widetilde{\rho_+}( g)\circ \widetilde{f_h}\circ \widetilde{\rho_+}( g)\circ \widetilde{f_h}\\&=\widetilde{\rho_+}( g)\circ \widetilde{f_h}\circ \widetilde{f_h}\circ \widetilde{\rho_+}( h^{-1}gh)\\
    &=\widetilde{\rho_+}( g)\circ \widetilde{\rho_+}( h^2)\circ \widetilde{\rho_+}( h^{-1}gh)\\
    &=\widetilde{\rho_+}( ghgh)
\end{align*}
By the above fact, since $\widetilde{F_h}$ admits a fixed point in $W^+$, then $\widetilde{\rho_+}( ghgh)$ also admits a fixed point in $W^+$. Moreover, since $\widetilde{\rho_+}$ coincides with the action of $G_+$ by deck transformations on $W^+$, we get  that $\widetilde{\rho_+}( ghgh)=id_{W^+}$ and that $ghgh$ is trivial in $G_+$. This is impossible, as $g\in G_+$, $h\in G-G_+$ and $G$ has no torsion.

Finally, let us prove that $\widetilde{f_h}$ satisfies Item (4) of this proposition. By Equation (\ref{eq.projectionf0}), we have that for every $x\in W^+$ 
\begin{align*}
    ( \pi \circ \widetilde{f_h})( x)&= ( \pi\circ\widetilde{\Phi}^{T( x)}\circ\widetilde{f_0})( x)\\&=( \pi\circ \widetilde{f_0})( x)=( \rho( h)\circ\pi)( x)
\end{align*} 
which proves the desired result and finishes the proof of this proposition. 
\end{proof}

Fix $h\in G-G_+$, $f_h$ the homeomorphism given by Proposition \ref{p.homeoconstructalmostorient} and $\widetilde{f_h}$ the unique lift of $f_h$ on $W^+$ for which 
\begin{equation}\label{eq.projectionrelation}
       \rho( h)\circ \pi = \pi \circ \widetilde{f_h}  
\end{equation}

Since $f_h$ has no fixed points in $M_+$, $f_h$ acts properly discontinuously on $M_+$. Take $M$ to be the quotient of $M_+$ by the action of $f_h$. Using the fact that $\widetilde{\rho_+}$ coincides with the action of $\pi_1( M_+)$ on $W^+$ by deck transformations, it is easy to see that 
\begin{equation}\label{eq.fundgroupgisomorphism}
    \pi_1( M)\cong G \cong <\widetilde{\rho_+}( G_+),\widetilde{f_h}>
\end{equation} 
Next, since $f_h$ is an self-orbit equivalence of $\Phi_+$, the one dimensional foliation given by the orbits of $\Phi_+$ projects to an one dimensional foliation on $M$ that can be parametrized by a non-singular flow $\Phi$, thanks to Theorem 27A of \cite{Whitney}. Even more, using the fact that $\Phi_+$ is a topological Anosov flow, one can check that $\Phi$ is also a topological Anosov flow and that the projection on $M$ of the stable (resp. unstable) foliation of $\Phi_+$ coincides with the stable (resp. unstable) foliation of $\Phi$. 

Denote by $\widetilde{\Phi}$ the lift of $\Phi$ on $W^+$. By construction, the one dimensional foliation in $W^+$ defined by the orbits of $\widetilde{\Phi}$ coincides with the one dimensional foliation in $W^+$ defined by the orbits of $\widetilde{\Phi}_+$; hence, $\mathcal{P}$ is the orbit space of both $\widetilde{\Phi}$ and $\widetilde{\Phi}_+$. Similarly, the lifts of the stable (resp. unstable) foliations of $\Phi$ and $\Phi_+$ on $W^+$ coincide. We deduce that $( \mathcal{P},\mathcal{F}^s,\mathcal{F}^u)$ corresponds to the bifoliated plane of both $\Phi_+$ and $\Phi$. 

Finally, let $\widetilde{\rho}: G\rightarrow \text{Homeo}( W^+)$ be the action defined by $\widetilde{\rho}( h)=\widetilde{f_h}$ and $\widetilde{\rho}( g)=\widetilde{\rho_+}( g)$ for every $g\in G_+$. By construction, $\widetilde{\rho}$ is an extension of $\widetilde{\rho_+}$ and corresponds to the action by deck transformations of $\pi_1( M)$ on $W^+$. Recall now that since $\widetilde{\rho_+}$ coincides with the diagonal action of $\rho_+$ on $W^+$, we have that 
\begin{equation}\label{eq.projectionrelationrhoplus}
    \forall g\in G_+ \quad   \rho_+( g)\circ \pi = \pi \circ \widetilde{\rho_+}( g)  
\end{equation}
Using the previous fact, together with Item (4) of Proposition \ref{p.homeoconstructalmostorient}, we get that 
\begin{equation}\label{eq.projectionrho}
    \forall g\in G \quad   \rho( g)\circ \pi = \pi \circ \widetilde{\rho}( g)  
\end{equation}
We deduce that the action $\rho$ corresponds to the action of $\pi_1( M)\cong G$ on the bifoliated plane of $\Phi$.

Thanks to our previous arguments, in order to prove Theorem \ref{t.main} in the case of non-fully orientable strong Markovian actions, it suffices to prove that $\mathcal{R}$ is the projection on $\mathcal{P}$ of the lift on $W^+$ of a reduced Markov partition of $\Phi$. This is the object of our next and final section.

\subsection{From a non-fully orientable strong Markovian action to a Markov partition}

\begin{prop}\label{p.perturbmarkov}
    There exists $\widetilde{\mathcal{R}}$ a $\widetilde{\rho}$-invariant collection of rectangles in $W^+$ such that 
    
    \begin{itemize}
        \item $\widetilde{\mathcal{R}}$ is a lift of $\mathcal{R}$ on $W^+$ for which each rectangle of $\mathcal{R}$ admits a unique lift in $\widetilde{\mathcal{R}}$.
        \item The projection of $\widetilde{\mathcal{R}}$ on $M$ defines a reduced Markov partition $\mathcal{R}_M$ of $\Phi$. 
    \end{itemize}
\end{prop}
\begin{proof}
    Since $\rho_+$ is a fully orientable strong Markovian action and $\mathcal{R}$ satisfies the corner condition, by Proposition \ref{p.fullyorientablemarkov}, there exists $\widetilde{\mathcal{R}_+}$ a $\widetilde{\rho_+}$-invariant collection of rectangles in $W^+$ such that each rectangle of $\mathcal{R}$ lifts to a unique rectangle in $\widetilde{\mathcal{R}_+}$, $\pi( \widetilde{\mathcal{R}_+})=\mathcal{R}$, and also such that the projection of $\widetilde{\mathcal{R}_+}$ on $M_+$ defines a reduced Markov partition $\mathcal{R}_{M_+}$ of $\Phi_+$.

    If $\widetilde{f_h}=\widetilde{\rho}(h)$ preserves $\widetilde{\mathcal{R}_+}$, then, thanks to (\ref{eq.fundgroupgisomorphism}) and to our construction of $\Phi$, it is easy to see that $\mathcal{R}_{M_+}$ is invariant by $f_h$ and that $\widetilde{\mathcal{R}_+}$ projects on $M$ to a reduced Markov partition of $\Phi$, which proves the desired result. Assume from now on that $\widetilde{f_h}$ does not preserve $\widetilde{\mathcal{R}_+}$. Even though $\widetilde{\mathcal{R}_+}$ is not $\widetilde{f_h}$-invariant, we have that 
    \begin{equation}
        \widetilde{f^2_h}( \widetilde{\mathcal{R}_+})=\widetilde{\mathcal{R}_+}
    \end{equation}
    Indeed, this is an immediate consequence of the fact that $\widetilde{\rho_+}$ preserves $\widetilde{\mathcal{R}_+}$ and the fact that $\widetilde{f^2_h}= \widetilde{\rho}( h^2)= \widetilde{\rho_+}( h^2)$. 

    By the definition of $\widetilde{\mathcal{R}_+}$ and Equation~(\ref{eq.projectionrelation}), the projections of $\widetilde{\mathcal{R}_+}$ and $\widetilde{f_h}( \widetilde{\mathcal{R}_+})$ onto $\mathcal{P}$ both coincide with $\mathcal{R}$. Moreover, each rectangle in $\mathcal{R}$ admits a unique lift in $\widetilde{\mathcal{R}_+}$ and a unique lift in $\widetilde{f_h}( \widetilde{\mathcal{R}_+})$. More precisely, for any rectangle $R\in \mathcal{R}$ denote by $\widetilde{R}$ its lift inside $\widetilde{\mathcal{R}_+}$. Thanks to Equation (\ref{eq.projectionrelation}), we have that the lift of $R$ inside $\widetilde{f_h}( \widetilde{\mathcal{R}_+})$ is equal to $\widetilde{f_h}( \widetilde{\rho(h^{-1})( R)})$, the image by $\widetilde{f_h}$ of the lift of $\rho( h^{-1})( R)$ inside $\widetilde{\mathcal{R}_+}$. We will now define a $\widetilde{\rho}$-invariant collection of rectangles using the rectangles in   $\widetilde{\mathcal{R}_+}$ and in $\widetilde{f_h}( \widetilde{\mathcal{R}_+})$.

     For any $\widetilde{x_0}\in W^+$ and any $t_1,t_2\in \mathbb{R}$, if $\widetilde{x_1}:=\widetilde{\Phi}^{t_1}( \widetilde{x_0})$ and $\widetilde{x_2}:=\widetilde{\Phi^{t_2}}( \widetilde{x_0})$, then we define the \emph{middle-point} between $\widetilde{x_1}$ and $\widetilde{x_2}$ as: $$\frac{\widetilde{x_1}+ \widetilde{x_2}}{2}:=\widetilde{\Phi}^{\tfrac{t_1+t_2}{2}}( \widetilde{x_0})$$
    
    Notice that since $\widetilde{\Phi}$ is a flow, the point $\displaystyle{\frac{\widetilde{x_1}+ \widetilde{x_2}}{2}}$ does not depend on our choice of $\widetilde{x_0}$. Using this fact together with the continuity of $\widetilde{\Phi}$, it is easy to prove that the map associating to any pair $( \widetilde{x},t)\in W^+\times \mathbb{R}$ the point $\displaystyle{\frac{\widetilde{x}+ \widetilde{\Phi}^t( \widetilde{x})}{2}}$ is continuous. Furthermore, since $\widetilde{\Phi}$ is invariant by $\widetilde{\rho}$, we have that 
    \begin{equation}\label{eq.invariancerho}
        \forall g\in G \quad \widetilde{\rho}( g)( \frac{\widetilde{x_1}+ \widetilde{x_2}}{2})=\frac{\widetilde{\rho}( g)( \widetilde{x_1})+ \widetilde{\rho}( g)( \widetilde{x_2})}{2}
    \end{equation}
    
    For any $R\in \mathcal{R}$ and any $x\in R$ denote by $\widetilde{x}( {R},\widetilde{\mathcal{R}_+})$ (resp.  $\widetilde{x}( {R},\widetilde{f_h}( \widetilde{\mathcal{R}_+}))$) the unique lift of $x$ inside $\widetilde{R}\in \widetilde{\mathcal{R}_+}$ (resp. $\widetilde{f_h}( \widetilde{\rho( h^{-1})( R)})\in \widetilde{f_h}( \widetilde{\mathcal{R}_+})$). We define 
    
    \begin{equation}\label{eq.defirinv}
        \widetilde{R^{inv}}:=\left\{\frac{\widetilde{x}( {R},\widetilde{\mathcal{R}_+})+\widetilde{x}( {R},\widetilde{f_h}( \widetilde{\mathcal{R}_+}))}{2} ~\middle| ~x\in R\right\}
    \end{equation} and $\widetilde{\mathcal{R}}:=\{\widetilde{R^{inv}}|R\in \mathcal{R}\}$. For the sake of simplicity, from now on we will write $$\widetilde{R^{inv}}=\frac{\widetilde{R}+\widetilde{f_h}( \widetilde{\rho( h^{-1})( R)})}{2}$$
Let us now show that $\widetilde{\mathcal{R}}$ satisfies both of the properties given in the statement of this proposition. 

\noindent\textbf{Claim 1:} \textit{The set $\widetilde{R^{inv}}$  is a rectangle transverse to $\widetilde{\Phi}$ verifying $\pi( \widetilde{R^{inv}})=R$. Moreover, $\widetilde{\mathcal{R}}$ is a lift of $\mathcal{R}$ on $W^+$ for which each rectangle of $\mathcal{R}$ admits a unique lift in $\widetilde{\mathcal{R}}$.}
\begin{proof}[Proof of Claim 1.]
Indeed, fix $R\in \mathcal{R}$. Following the previous notations, the rectangles $\widetilde{R}\in \widetilde{\mathcal{R}_+}$ and $\widetilde{f_h}( \widetilde{\rho( h^{-1})( R)})\in \widetilde{f_h}( \widetilde{\mathcal{R}_+})$ are two lifts of $R$ in $W^+$ that are transverse to $\widetilde{\Phi}$. Thanks to the previous fact and to the  continuity of the middle-point map, $\widetilde{R^{inv}}$ is a closed topological disk in $W^+$ that is transverse to $\widetilde{\Phi}$ (indeed it intersects at most once every orbit of $\widetilde{\Phi}$) and that verifies $\pi( \widetilde{R^{inv}})=R$. We deduce that $\widetilde{R^{inv}}$ is a rectangle in $W^+$ and also that  $\pi( \widetilde{\mathcal{R}})=\mathcal{R}$. Finally, since $\widetilde{R}$ (resp. $\widetilde{f_h}( \widetilde{\rho( h^{-1})( R)})$) is the unique lift of $R$ inside $\widetilde{\mathcal{R}_+}$ (resp. $\widetilde{f_h}( \widetilde{\mathcal{R}_+})$), we get that $\widetilde{R^{inv}}$ is the unique lift of $R$ inside $\widetilde{\mathcal{R}}$. 
\end{proof}

\noindent\textbf{Claim 2:} \textit{The set $\widetilde{\mathcal{R}}$ is $\widetilde{\rho}$-invariant.}

\begin{proof}[Proof of Claim 2]Consider $R\in \mathcal{R}$ and $\widetilde{R}$ (resp. $\widetilde{f_h}( \widetilde{\rho( h^{-1})( R)})$, $\widetilde{R^{inv}}$) the unique lift of $R$ inside $\widetilde{\mathcal{R}_+}$ (resp. $\widetilde{f_h}( \widetilde{\mathcal{R}_+})$, $\widetilde{\mathcal{R}}$). By (\ref{eq.fundgroupgisomorphism}), in order to prove that $\widetilde{\mathcal{R}}$ is invariant by $\widetilde{\rho}$ it suffices to show that $\widetilde{\rho}( h)( \widetilde{\mathcal{R}})=\widetilde{f_h}( \widetilde{\mathcal{R}})=\widetilde{\mathcal{R}}$ and that $\widetilde{\rho}( g)( \widetilde{\mathcal{R}})= \widetilde{\mathcal{R}}$ for every $g\in G_+$. Indeed, \begin{align*}
    \widetilde{f_h}( \widetilde{R^{inv}})= \widetilde{\rho}( h)( \widetilde{R^{inv}})&=\widetilde{\rho}( h)\big( \frac{\widetilde{R}+\widetilde{f_h}( \widetilde{\rho( h^{-1})( R)})}{2}\big)\\
    &=\widetilde{\rho}( h)\big( \frac{\widetilde{R}+\widetilde{\rho}( h)( \widetilde{\rho( h^{-1})( R)})}{2}\big)
\end{align*}
By Equation (\ref{eq.invariancerho}) we have that
$$\widetilde{\rho}( h)( \widetilde{R^{inv}})=\frac{\widetilde{\rho}( h)( \widetilde{R})+\widetilde{\rho}( h^2)( \widetilde{\rho( h^{-1})( R)})}{2} $$
Since $\widetilde{R}\in \widetilde{\mathcal{R}_+}$ and $\widetilde{\rho}( h)=\widetilde{f_h}$, clearly we have that $\widetilde{\rho}( h)( \widetilde{R})\in \widetilde{f_h}( \widetilde{\mathcal{R}_+})$. Next, recall that $\widetilde{\rho( h^{-1})( R)}\in \widetilde{\mathcal{R}_+}$, that $\widetilde{\rho}( h^2)=\widetilde{\rho_+}( h^2)$ and that $\widetilde{\mathcal{R}_+}$ is invariant by $\widetilde{\rho_+}$. Using the previous facts we have that $\widetilde{\rho}( h^2)( \widetilde{\rho( h^{-1})( R)})\in \widetilde{\mathcal{R}_+}$. More specifically, by Equation (\ref{eq.projectionrho}) 
$$\pi( \widetilde{\rho}( h^2)( \widetilde{\rho( h^{-1})( R)}))=\rho( h)( R)\in \mathcal{R}$$

We deduce that $\widetilde{\rho}( h^2)( \widetilde{\rho( h^{-1})( R)})$ and $\widetilde{\rho}( h)( \widetilde{R})$ are the unique lifts of $\rho( h)( R)$ inside $\widetilde{\mathcal{R}_+}$ and $\widetilde{f_h}( \widetilde{\mathcal{R}})$ respectively and that $\widetilde{\rho}( h)( \widetilde{R^{inv}})\in \widetilde{\mathcal{R}}$.

Next, consider $g\in G_+$. Once again, by Equation (\ref{eq.invariancerho}), we have 

\begin{align*}
   \widetilde{\rho}( g)( \widetilde{R^{inv}})&=\widetilde{\rho}( g)( \frac{\widetilde{R}+\widetilde{f_h}( \widetilde{\rho( h^{-1})( R)})}{2})\\
    &=\widetilde{\rho}( g)( \frac{\widetilde{R}+\widetilde{\rho}( h)( \widetilde{\rho( h^{-1})( R)})}{2})\\
    &=\frac{\widetilde{\rho}( g)( \widetilde{R})+\widetilde{\rho}( gh)( \widetilde{\rho( h^{-1})( R)})}{2}\\
    &=\frac{\widetilde{\rho}( g)( \widetilde{R})+\widetilde{\rho}( h)\left( \widetilde{\rho}( h^{-1}gh)( \widetilde{\rho( h^{-1})( R)})\right)}{2}
\end{align*}

By an argument similar to the one used before, one can show that the rectangles  $\widetilde{\rho}( g)( \widetilde{R})$ and $\widetilde{\rho}( h)\left( \widetilde{\rho}( h^{-1}gh)( \widetilde{\rho( h^{-1})( R)}\right)$ are the unique lifts of $\rho( g)( R)$ inside $\widetilde{\mathcal{R}_+}$ and $\widetilde{f_h}( \widetilde{\mathcal{R}_+})$ respectively and that  $\widetilde{\rho}( g)( \widetilde{R^{inv}})\in \widetilde{\mathcal{R}}$. This  finishes the proof of Claim 2. 

\end{proof}

By Claims~1 and~2, in order to prove that $\widetilde{\mathcal{R}}$ satisfies both of the required properties of this proposition, it suffices to prove that $\widetilde{\mathcal{R}}$ projects on $M$ to a reduced Markov partition of $\Phi$. In order to do so, let us first show that

\textbf{Claim 3:} \textit{Any two distinct rectangles in $\widetilde{\mathcal{R}}$ have disjoint interiors.}

\begin{proof}
    The above fact is a consequence of the monotonicity of the middle-point map. More specifically, for every $\widetilde{x},\widetilde{y}\in W^+$ we write $\widetilde{x}>\widetilde{y}$ if there exists $t>0$ such that $\widetilde{\Phi}^t( \widetilde{x})=\widetilde{y}$. By the definition of the middle-point map, for every $\widetilde{x_1},\widetilde{y_1}, \widetilde{x_2},\widetilde{y_2}\in W^+$ belonging in the same $\widetilde{\Phi}$-orbit in $W^+$, if $\widetilde{x_1}>\widetilde{x_2}$ and $\widetilde{y_1}>\widetilde{y_2}$, then $$\frac{\widetilde{x_1}+\widetilde{y_1}}{2}>\frac{\widetilde{x_2}+\widetilde{y_2}}{2}$$

    Assume now by contradiction that there exist two distinct rectangles $\widetilde{R}^{inv}, \widetilde{S}^{inv} \in \widetilde{\mathcal{R}}$ such that $\mathrm{Int}( \widetilde{R}^{inv})\cap \mathrm{Int}( \widetilde{S}^{inv})\neq \emptyset$. In this case, setting $R=\pi( \widetilde{R}^{inv})\in \mathcal{R}$ and $S=\pi( \widetilde{S}^{inv})\in \mathcal{R}$, we have $\mathrm{Int}( R)\cap\mathrm{Int}( S)\neq \emptyset$. Hence, by Lemma~\ref{l.npredecessor}, we may assume without loss of generality that $R$ is a predecessor of some non-zero generation of $S$. Let  $\widetilde{R}, \widetilde{S}$ (resp. $\widetilde{f_h}( \widetilde{\rho( h^{-1})( R)})$, $\widetilde{f_h}( \widetilde{\rho( h^{-1})( S)})$) be the unique lifts of $R,S$ inside $\widetilde{\mathcal{R}_+}$ (resp. $\widetilde{f_h}( \widetilde{\mathcal{R}_+})$). Recall that for every $x\in R$ we denoted by $\widetilde{x}( {R},\widetilde{\mathcal{R}_+})$ and by   $\widetilde{x}( {R},\widetilde{f_h}( \widetilde{\mathcal{R}_+}))$ the unique lifts of $x$ inside $\widetilde{R}$ and $\widetilde{f_h}( \widetilde{\rho( h^{-1})( R)})$. Denote similarly by $\widetilde{x}( {R},\widetilde{\mathcal{R}})$ the unique lift of $x\in R$ inside $\widetilde{R}^{inv}$.  Since $\widetilde{\Phi}$ is invariant under $\widetilde{f_h}$ and $\widetilde{\mathcal{R}_+}$ is the lift of a reduced Markov partition of $\Phi_+$, we obtain that
$$\forall x\in \mathrm{Int}( R)\cap\mathrm{Int}( S) \quad \widetilde{x}( {S},\widetilde{\mathcal{R}_+})>\widetilde{x}( {R},\widetilde{\mathcal{R}_+}) \quad \text{and} \quad  \widetilde{x}( {S},\widetilde{f_h}( \widetilde{\mathcal{R}_+}))>\widetilde{x}( {R},\widetilde{f_h}( \widetilde{\mathcal{R}_+}))$$

    Using the previous fact, together with (\ref{eq.defirinv}) and the monotonicity of the middle-point map, we deduce that  
\begin{equation}\label{eq.sinvafterrinv}
        \forall x\in\mathrm{Int}( R)\cap\mathrm{Int}( S) \quad \widetilde{x}( {S},\widetilde{\mathcal{R}})>\widetilde{x}( {R},\widetilde{\mathcal{R}})
    \end{equation}

    This proves that $\mathrm{Int}( \widetilde{R}^{inv})\cap \mathrm{Int}( \widetilde{S}^{inv})=\emptyset$, contradicting our original hypothesis. 
\end{proof}
\textbf{Claim 4:} \textit{$\widetilde{\mathcal{R}}$ projects on $M$ to a reduced Markov partition of $\Phi$}
\begin{proof}
    By Claims 1 and 3, the rectangles in $\widetilde{\mathcal{R}}$ are transverse to $\widetilde{\Phi}$ and have mutually disjoint interiors. Moreover, by Claims 1 and 2,  $\widetilde{\mathcal{R}}$ is $\widetilde{\rho}$-invariant and consists of finitely many $\widetilde{\rho}$-orbits of rectangles. It follows that the projection of $\widetilde{\mathcal{R}}$ on $M$ is a finite set of rectangles that are transverse to $\Phi$ and have mutually disjoint interiors. In order to prove that this projection is indeed a reduced Markov partition of $\Phi$, it suffices to show that: 

    \begin{enumerate}
        \item for any two distinct rectangles $\widetilde{R}^{inv}, \widetilde{S}^{inv}\in \widetilde{\mathcal{R}}$ there does not exist a continuous function $\tau: \widetilde{R}^{inv} \rightarrow \mathbb{R}$ such that $\widetilde{\Phi}^{\tau}( \widetilde{R}^{inv})\subseteq \widetilde{S}^{inv}$. 
        \item for any $\widetilde{R}^{inv}, \widetilde{S}^{inv}\in \widetilde{\mathcal{R}}$ and any $\widetilde{x}\in \partial^s\widetilde{R}^{inv}$ (resp. $\widetilde{x}\in \partial^u\widetilde{R}^{inv}$) the future (resp. past) orbit of $\widetilde{x}$ by $\widetilde{\Phi}$ cannot intersect the interior of $ \widetilde{S}^{inv}$. 
    \end{enumerate}
    
   The first of the above statements follows directly from the equality $\pi(\widetilde{\mathcal{R}})=\mathcal{R}$ and the Markovian intersection axiom, according to which no two distinct rectangles in $\mathcal{R}$ are contained one in the other. We now prove the second statement. Suppose by contradiction that the future orbit by $\widetilde{\Phi}$ of some $\widetilde{x}\in \partial^s\widetilde{R}^{inv}$ intersects the interior of $ \widetilde{S}^{inv}$. Let $x:=\pi( \widetilde{x})$, $R=\pi( \widetilde{R}^{inv})$ and $S=\pi( \widetilde{S}^{inv})$. Since $x\in \partial^s R\cap \mathrm{Int}( S)$, we have that $\mathrm{Int}( R)\cap \mathrm{Int}( S)\neq \emptyset$ and more specifically that $R$ is a successor of some generation of $S$. Thanks to (\ref{eq.sinvafterrinv}),  if $ \widetilde{x}( {S},\widetilde{\mathcal{R}_+})$ denotes the unique lift of $x$ inside $\widetilde{S}^{inv}$, then $\widetilde{x}$ belongs in the future orbit of $\widetilde{x}( {S},\widetilde{\mathcal{R}_+})$, which contradicts our original hypothesis. We similarly show that the past orbit of every point in $\partial^u\widetilde{R}^{inv}$ cannot intersect the interior of $ \widetilde{S}^{inv}$. 
\end{proof}

This finishes the proof of Proposition~\ref{p.perturbmarkov}, and thus also  the proof of Theorem~\ref{t.main}.
\end{proof}

\end{document}